\documentclass[12pt]{amsart}
\usepackage{graphicx}
\usepackage[nocompress]{cite}
\usepackage{amssymb, comment, color}
\usepackage[toc,page]{appendix}
\usepackage{amssymb,times}
\usepackage[utf8]{inputenc}
\usepackage[english]{babel}
\usepackage{amsmath}
\usepackage{amsthm}
\usepackage{amsfonts}
\usepackage{amssymb}
\usepackage{amscd}
\usepackage{enumerate}
\usepackage{mathtools}
\usepackage{tikz-cd}
\usepackage{graphicx}
\usepackage{enumitem}
\usepackage{bbm}

\usepackage[colorlinks=true,urlcolor=green,bookmarks=true,bookmarksopen=true,citecolor=green]{hyperref}

\usepackage{parskip}

\newcommand{\id}{{\mathbbm{1}}}

\newcommand{\N}{\mathbbm{N}}                   
\newcommand{\Z}{\mathbbm{Z}}                    
\newcommand{\ZZ}{\mathbbm{Z}}                  
\newcommand{\R}{\mathbbm{R}}                     
\newcommand{\RR}{\mathbbm{R}}                  
\newcommand{\C}{\mathbbm{C}}                     
\newcommand{\D}{\mathbb{D}}                     
\newcommand{\BB}{\mathcal{B}}                    
\newcommand{\M}{\mathcal{M}}                    
\renewcommand{\P}{\mathcal{P}}                  

\newcommand{\set}[1]{\left\{ #1\right\}}        
\newcommand{\supp}{\mathrm{supp\,}}             
\newcommand{\Diff}{\mathrm{Diff}}               
\newcommand{\CP}{\mathbbm{CP}}                  
\newcommand{\crit}{\mathrm{Crit\,}}             
\newcommand{\Area}{\mathrm{Area}}               
\newcommand{\CZ}{\mu_{\mathrm{CZ}}}             
\newcommand{\Morse}{\mu_{\mathrm{Morse}}}       
\newcommand{\Ham}{\mathrm{Ham}}                 
\newcommand{\Symp}{\mathrm{Symp}}               
\newcommand{\QH}{{\mathrm{QH}}}                 
\newcommand{\CF}{{\mathrm{CF}}}                 
\newcommand{\HF}{{\mathrm{HF}}}                 
\newcommand{\CM}{{\mathrm{CM}}}                 
\newcommand{\HM}{{\mathrm{HM}}}                 
\newcommand{\im}{\mathrm{im}}

\newcommand{\PP}{{*_{\mathrm{PP}}}}
\renewcommand{\top}{{\mathrm{top}}}

\newcommand{\Int}{\mathrm{int}}             

\newcommand{\std}{\mathrm{std}}

\newcommand{\topo}{\mathrm{top}}

\newcommand{\dvol}{\mathrm{dvol}}

\newcommand{\norm}[1]{\left\lVert#1\right\rVert}    
\newcommand{\overbar}[1]{\mkern 2.8mu\overline{\mkern-2.8mu#1\mkern-1.0mu}\mkern 1.0mu}

\newcommand{\Hofer}{\mathrm{Hofer}}
\newcommand{\Spec}{\mathrm{Spec}}
\newcommand{\cont}{{[.]}}									

\newcommand{\red}[1]{{\color{red} #1}}

\newtheorem{TheoremX}{Theorem}

\newtheorem{thm}{Theorem}[section]               
\newtheorem*{thm*}{Theorem}               
\newtheorem{cor}[thm]{Corollary}        
\newtheorem*{cor*}{Corollary}        
\newtheorem{lem}[thm]{Lemma}  
\newtheorem*{lem*}{Lemma}
\newtheorem{prop}[thm]{Proposition}     
\newtheorem{ass}[thm]{Assumption}       
\newtheorem*{choice}{Standing choice}  

\newtheorem{defn}[thm]{Definition}             

\newtheorem{rem}[thm]{Remark}           
\newtheorem{question}{Question}
\newtheorem*{acknowledgement*}{\protect\acknowledgementname}
\newcounter{claim}

\newcommand{\blue}[1]{{\color{blue} #1}}

\providecommand{\acknowledgementname}{Acknowledgement}

\author{Marcelo R.R. Alves}
\thanks{M.R.R. Alves was supported by the Senior Postdoctoral fellowship of the Research Foundation - Flanders (FWO) in fundamental research 1286921N and by the Deutsche
Forschungsgemeinschaft (DFG, German Research Foundation) via the grant “Himmelsmechanik, Hydrodynamik und Turing-Maschinen” - 541525489.}
\address{Marcelo R.R. Alves, 
    Institute of Mathematics ,\\
	University of Augsburg,
	Chair Analysis and Geometry ,
	Universitätsstraße 14,
    DE-86159 Augsburg
	Germany.}
\email{\texttt{marcelorralves@gmail.com}}

\author{Matthias Meiwes}
\thanks{M. Meiwes was supported by  the ERC Starting Grant 757585 and the Israel Science Foundation Grant 938/22}
\address{Matthias Meiwes,
	School of Mathematical Sciences, Tel Aviv University, Ramat Aviv, Tel Aviv 69978, Israel.}
\email{\texttt{matthias.meiwes@live.de}}

\author{Beomjun Sohn}
\thanks{B. Sohn was supported by National Research Foundation of Korea (NRF) grant NRF-2020R1A5A1016126, RS-2023-00211186, and RS-2025-02317642.}
\address{Chair for Geometry and Analysis, RWTH Aachen University, Pontdriesch 10-12, DE-52062 Aachen, Germany}
\email{\textit{bsohn95@gmail.com}}

\title[Robustness of $h_{\rm top}$ under small area deformations]{Robustness of topological entropy under small area deformations}
\begin{document}

\begin{abstract}
In this paper, we establish a new type of stability phenomenon for the topological entropy of Hamiltonian diffeomorphisms of closed surfaces.
For a closed surface endowed with an area form $(\Sigma,\omega)$ and a Hamiltonian diffeomorphism $\phi$ of $(\Sigma,\omega)$, we show that for every $\varepsilon>0$ there exists $A=A(\phi,\varepsilon)>0$ such that
\[
h_{\mathrm{top}}(\phi') > h_{\mathrm{top}}(\phi)-\varepsilon
\]
for every Hamiltonian diffeomorphism $\phi'$ obtained from $\phi$ by a deformation supported in a disjoint union of disks 
each of area less than $A$. In particular, if $h_{\mathrm{top}}(\phi)>0$, then $\phi$ cannot be made to have zero entropy by an area-preserving deformation supported in disks of small area. This follows from the new braid stability result established in this paper with respect to the spectral distance recently introduced by Connery-Grigg.

\end{abstract}

\maketitle
\tableofcontents

\vspace{-2em}

\section{Introduction} \label{sec:mainresults}

In this article we study the stability of the topological entropy of Hamiltonian diffeomorphisms under perturbations. Our main result establishes that positivity of the topological entropy $h_{\rm top}$ cannot be destroyed if the region on which the perturbation is supported is a disjoint union of disks with sufficiently small area. In order to give the precise statement of our result we recall some necessary notions. 


Let $\Sigma$ be a closed surface and $\omega$ be an area form on $\Sigma$. As $\omega$ is a symplectic form on $\Sigma$, we will refer to the pair $(\Sigma,\omega)$ as a symplectic surface. As a normalization, we assume that the area of $(\Sigma,\omega)$ is $1$, i.e. $\int_\Sigma \omega = 1$.

A time-dependent Hamiltonian $H:S^1 \times \Sigma \to \mathbbm{R}$ gives rise to a time-dependent vector field $X_H$ on $\Sigma$, called the Hamiltonian vector field of $H$, given by the formula 
$	\iota_{X_{H(t,\cdot)}} \omega = -d_\Sigma H(t,\cdot)$,
where $d_\Sigma H(t,\cdot)$ is the differential of $H(t, \cdot) : \Sigma \to \R$ in $\Sigma$. The flow $\phi^t_H$ of $X_H$ is called the Hamiltonian flow of $H$. A Hamiltonian diffeomorphism $\phi$ is a diffeomorphism of $\Sigma$ which is the time $1$-map of the Hamiltonian flow of some time-depedent Hamiltonian $H$. If $\phi$ is a Hamiltonian diffeomorphism and $H$ is a Hamiltonian such that $\phi$ is the time $1$-map of $\phi_H$, we say that $H$ generates $\phi$.

The main result of the present article is the following:

\begin{TheoremX}\label{Main theorem: Topological entropy after perturbation supported on a disk}
	Let $(\Sigma,\omega)$ be a closed symplectic surface, and let $\phi$ be a Hamiltonian diffeomorphism of $(\Sigma,\omega)$. Then, for every $\varepsilon>0$, there exists a real number $A(\phi,\varepsilon)>0$ such that 
	\[h_{\top}(\psi \circ \phi)>h_{\top}(\phi)-\varepsilon\]
	for every area preserving diffeomorphism $\psi$ whose support is contained in a disjoint union $\sqcup^{n}_{i=1}D_{i}$ of open disks $D_i\subset \Sigma$ with $\mathrm{Area}(D_i)<A(\phi,\epsilon)$.
\end{TheoremX}

\begin{rem}
Notice that there is no restriction on the diameter of the disks $D_i$ appearing in the statement of Theorem~\ref{Main theorem: Topological entropy after perturbation supported on a disk}, but only on $\mathrm{Area}(D_i)$. Indeed, the disks $D_i$ may have the same diameter as $\Sigma$. 
For this reason, Theorem~\ref{Main theorem: Topological entropy after perturbation supported on a disk} cannot be obtained from the results of Nitecki \cite{nitecki} on the lower-semicontinuity of $h_\top$ with respect to the $C^0$-topology. 
\end{rem}

Hamiltonian diffeomorphisms form a group with respect to the composition of diffeomorphisms. We denote the group of Hamiltonian diffeomorphisms of $(\Sigma, \omega)$ by $\mathrm{Ham}(\Sigma,\omega)$. 
On a symplectic surface, $\Ham(\Sigma,\omega)$ is a subgroup of the group of orientation- and area-preserving diffeomorphisms. In the cases of the disk and the sphere, these groups coincide. Moreover, every area-preserving diffeomorphism of $S^2$ is either Hamiltonian or the composition of a Hamiltonian diffeomorphism with the reflection across the equator. We thus obtain the following corollary of Theorem~\ref{Main theorem: Topological entropy after perturbation supported on a disk}.



\begin{cor}
Let $\omega$ be an area form on $S^2$, and let $\phi$ be a diffeomorphism of $S^2$ preserving the measure $|\omega|$. Then, for every $\varepsilon>0$, there exists a real number $A(\phi,\varepsilon)>0$ such that
\[
h_{\top}(\psi \circ \phi)>h_{\top}(\phi)-\epsilon
\]
for every area preserving diffeomorphism $\psi$ whose support is contained in a disjoint union $\sqcup^{n}_{i=1}D_{i}$ of open disks $D_i\subset S^2$ with $\mathrm{Area}(D_i)<A(\phi,\epsilon)$.
\end{cor}

We note that Theorem~\ref{Main theorem: Topological entropy after perturbation supported on a disk} can also be viewed as a rigidity theorem in Hamiltonian dynamics. In contrast, such a statement does not hold among surface diffeomorphisms: the topological entropy generated by a Smale horseshoe can be removed by a perturbation supported in a single fork-shaped disk of arbitrarily small area as in Figure \ref{fig:Fork-shape disk} that pushes the whole invariant set, and eventually some isolating neighborhood, into a sink outside the horseshoe. Note that maps with sinks cannot be area-preserving diffeomorphisms.

\begin{figure}
    \centering
    \includegraphics[width=.59\linewidth]{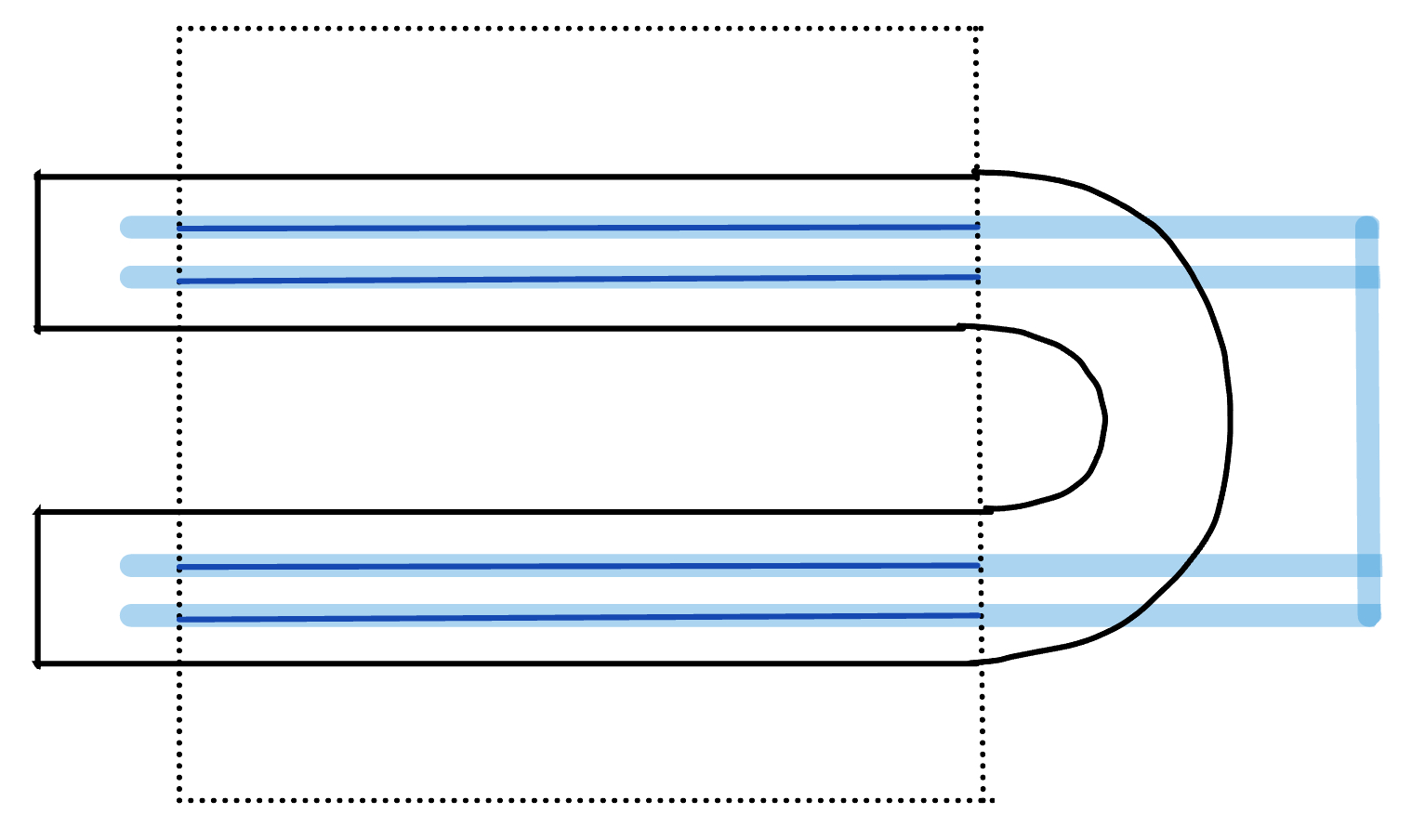}
    \caption{Disk containing the invariant set of a horseshoe}
    \label{fig:Fork-shape disk}
\end{figure}

\newpage

Theorem~\ref{Main theorem: Topological entropy after perturbation supported on a disk} is obtained by combining methods from symplectic topology with the theory of braids on surfaces. It is a consequence of a lower semicontinuity result for $h_{\rm top}$ with respect to the so-called PSS-image spectral distance on the group $\Ham(\Sigma,\omega)$, which follows from the new braid stability result respect to the same distance. In order to explain this result in more detail, we recall some necessary notions.

\subsection{Braids and Hamiltonian diffeomorphisms}
\label{Section: Braids and Hamiltonian diffeomorphisms}

Let $\mathcal{Y}$ be a finite subset of $\Sigma$. We call a subset $\BB$ of the product $\Sigma \times [0,1]$ a \textbf{geometric $n$-braid} based at $\mathcal{Y}$ if the following conditions hold:
\begin{enumerate}
	\item $\BB$ is the union of mutually disjoint $n$ embedded arcs $s_1, \ldots, s_n$.
	\item Each arc $s_i$ connects points $(0,x_i) \in \set{0}\times \mathcal{Y}$ and $(1,y_i)\in \set{1} \times \mathcal{Y}$. 
	\item Each arc intersects $\set{t}\times \Sigma$ exactly once for $0\le t \le 1$.
\end{enumerate}
Combining (1) and (2) we obtain that for $i\neq j \in \{1,...,n\}$, we have $x_i \neq x_j$ and $y_i \neq y_j$.   It follows that a braid induces a bijection $\mu_\BB:\mathcal{Y}\to \mathcal{Y}$ defined by $\mu_\BB(x_i)=y_i$. Two geometric braids based at $\mathcal{Y}$ are said to be \textbf{isotopic} if one can be continuously deformed to the other through geometric braids all of which are based at $\mathcal{Y}$. Each isotopy class of geometric braids based in $\mathcal{Y}$ is called \textbf{$n$-braid} in $\Sigma$ based at $\mathcal{Y}$. The set of $n$-braids in $\Sigma$ based at $\mathcal{Y}$ forms a group called the braid group, and is denoted by $B_n(\Sigma, \mathcal{Y})$. The product operation of the group $B_n(\Sigma, \mathcal{Y})$ is given by the concatenation. The $n$-braid $\BB$ is said to be \textbf{pure} if the bijection $\mu_\BB$ is the identity. The set of pure braids forms a subgroup of $B_n(\Sigma, \mathcal{Y})$ which we denote by $P_n(\Sigma, \mathcal{Y})$.

Two geometric $n$-braids possibly based at different sets are said to be \textbf{freely isotopic} if one can continuously {deform} one to the other through geometric $n$-braids, allowing base points to vary during the {deformation}. The number of connected components remains constant along the {deformation}. Pure geometric braids can only be freely isotopic to pure geometric braids.

When studying the fixed points or periodic orbits of Hamiltonian diffeomorphisms, it is better to understand braids as a subset of $S^1\times \Sigma$ by identifying $\set{0}\times\Sigma$ and $\set{1}\times \Sigma$. By conditions (2) and (3), braids are in one-to-one correspondence with links that are positively transverse to each $\set{t}\times \Sigma$. 

Let $\phi$ be a homeomorphism on $\Sigma$ isotopic to the identity $\id_\Sigma$, and let \linebreak $\Phi = \set{\phi^t}_{t\in S^1}$ be a path of homeomorphisms connecting $\id_\Sigma$ and $\phi$. Given a finite $\phi$-invariant set $\mathcal{Y}$ of $\phi$ (i.e.\  $\phi(\mathcal{Y})=\mathcal{Y}$), we define the geometric \textbf{braid} $\BB(\mathcal{Y},\Phi)$ by
\[\BB(\mathcal{Y},\Phi):=\bigcup_{t\in S^1} \set{t}\times \phi^t(\mathcal{Y}).\]

A priori, the braid $\BB(\mathcal{Y},\Phi)$ may depend on the choice of isotopy from $\id_\Sigma$ to~$\phi$. Let $\Phi_0=\set{\phi^t_0}_{t\in S^1}$ and $\Phi_1= \set{\phi^t_1}_{t\in S^1}$ be two paths of homeomorphisms connecting $\id_\Sigma$ and $\phi$. If  $\Phi_0$ and $\Phi_1$ are isotopic relative to their endpoints,
then it is not hard to see that the braids $\BB(\mathcal{Y},\Phi_0)$ and $\BB(\mathcal{Y},\Phi_1)$, defined respectively by $\Phi_0$ and $\Phi_1$, are isotopic as geometric $n$-braids based at $\mathcal{Y}$.
It follows that the braid $\BB(\mathcal{Y},\Phi)$ is well-defined up to the choice of the isotopy class of $\Phi$ relative to its endpoints. In particular, it is well-defined when $\Sigma$ has genus at least two.


We denote by $\mathcal{P}(H)$ the set of $1$-periodic orbits of $\phi^t_H$. 
There is a bijective correspondence between $1$-periodic orbits of the Hamiltonian flow $\{\phi^t_H\}$ and fixed points of the Hamiltonian diffeomorphism $\phi_H^1$. We can thus regard the elements in $\mathcal{P}(H)$ as fixed points of the Hamiltonian diffeomorphism $\phi$, and think of  $\mathcal{P}(H)$ as the set of fixed points of $\phi$. We will use both characterizations of elements of $\mathcal{P}(H)$ interchangeably in what follows.

Recall that a fixed point $y$ of $\phi_H^1$ is called non-degenerate if it does not admit $1$ as an eigenvalue for the linearization of $\phi_H^1$ at $y$. We say that $H$ is  non-degenerate if every fixed point of $\phi_H^1$ is non-degenerate. In that case, $\mathcal{P}(H)$ is a finite set.

Hamiltonian diffeomorphisms are, by definition, isotopic to the identity. Indeed, let $\phi \in \Ham(\Sigma,\omega)$, and let $H:S^1 \times \Sigma \to \R$ be a Hamiltonian generating $\phi$. The flow $\phi^t_H$ of $H$ for $t\in[0,1]$ gives a path of diffeomorphisms starting at $\id_\Sigma$ and ending at $\phi$. In particular, for every finite subset $\mathcal{Y}$ of fixed points of $\phi$ we define a geometric braid $\BB(\mathcal{Y},H)$ using the path $\phi^t_H$. We remark that since $\mathcal{Y}$ consists of fixed points of $\phi$, $\BB(\mathcal{Y},H)$ is a pure geometric braid.

\begin{figure}
    \centering
    \includegraphics[width=0.6\linewidth]{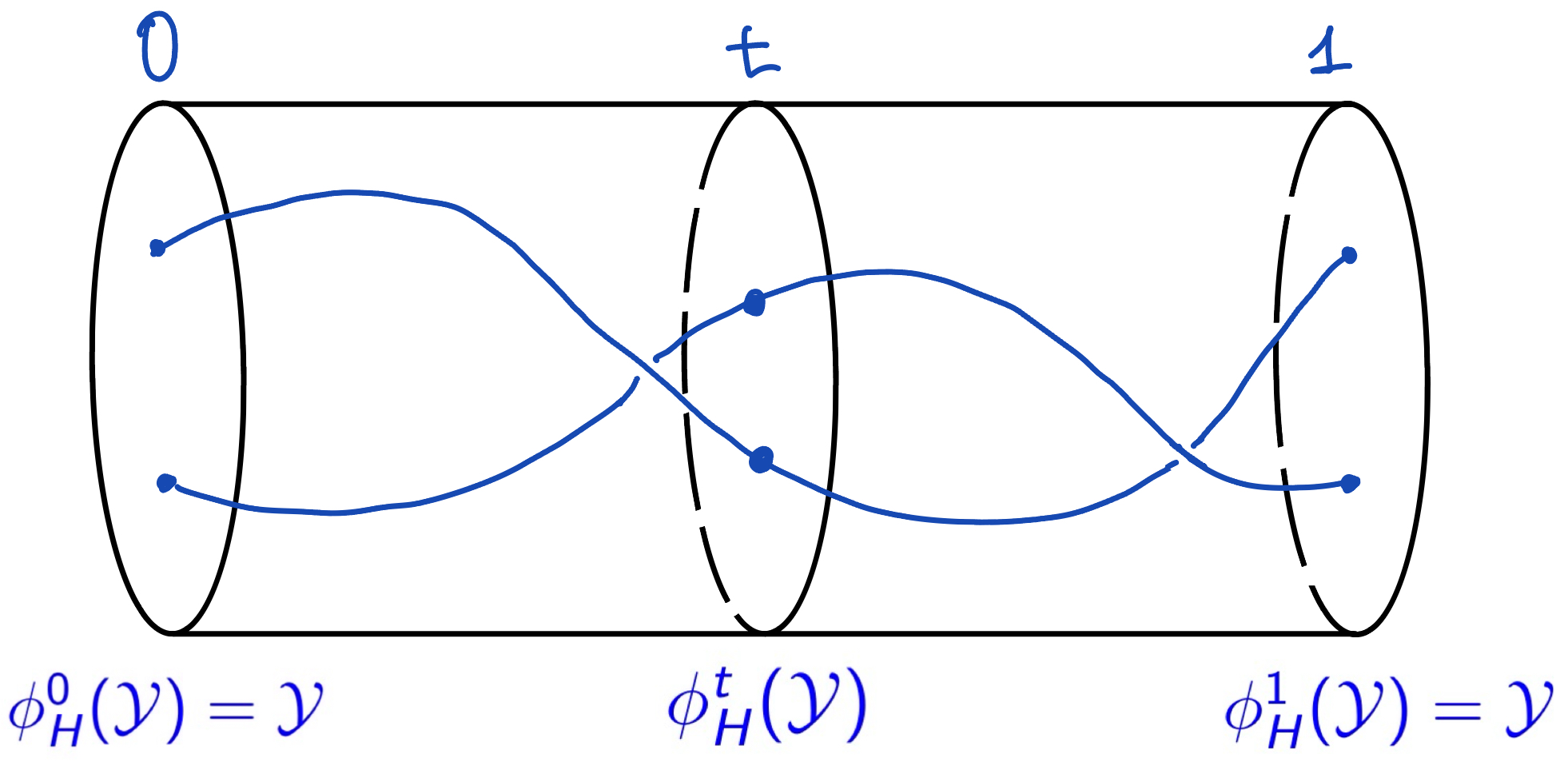}
    \caption{Geometric Braid $\mathcal{B}(\mathcal{Y},H)$ of $1$-periodic orbits}
    \label{fig:placeholder}
\end{figure}


Clearly, the braid $\BB(\mathcal{Y},H)$ depends on the choice of Hamiltonian $H$. 
{Nevertheless, the \textbf{braid type} of $\mathcal{Y}$ is well-defined for $\phi$; see Section \ref{Subsubsection: Proof of Theorem C} for the  precise definition of {braid type}}. For the moment, it will be sufficient to know that freely isotopic braids have the same braid type.

We remark that given $\phi^1_H \in \Ham(\Sigma,\omega)$ and $\mathcal{Y} \subset \mathcal{P}(H)$ the braid $\BB(\mathcal{Y},H)$ does not depend on the choice of a Hamiltonian $H$ when $\Sigma \ne S^2$. Indeed, this follows from the fact that the fundamental group $\pi_1(\Ham(\Sigma,\omega),\id_\Sigma)$ is trivial in this case.


In order to investigate the stability of the $h_{\rm top}$, we consider the following braid stability question:
\begin{center}
	\textit{Does the braid $\BB(\mathcal{Y}, H)$ persists under a small perturbation of $\phi$?}
\end{center}
The question above can be reformulated more precisely as follows:
\begin{question} 
Let $\mathcal{Y}$ be a finite invariant set of the Hamiltonian diffeomorphism $\phi$ and $H$ be a Hamiltonian generating $\phi$. Does there exist a small neighbourhood $\mathcal{U}$ of $\phi$, such that every Hamiltonian diffeomorphism $\phi' \in \mathcal{U}$ has a finite invariant set $\mathcal{Y}'$ and a Hamiltonian $H'$ generating $\phi'$ such that $\BB(\mathcal{Y}, H)$ and $\BB(\mathcal{Y}', H')$ are freely isotopic as braids?
\end{question}

The answer to this question clearly depends on the choice of topology on $\Ham(\Sigma,\omega)$. For example, the answer is positive for the $C^1$-topology by the implicit function theorem. The answer is still positive for the $C^0$-topology. Recently, it was proved by the first two authors \cite[Theorem 1]{Alves-Meiwes24} that the answer is again positive for the topology induced by the Hofer-distance $d_{\rm Hofer}$ on $\Ham(\Sigma,\omega)$ under the hypothesis that $\mathcal{Y}$ is a finite set of fixed points of $\phi$. In \cite{Hutchings23}, Hutchings generalized this for any finite invariant set $\mathcal{Y}$. We recall the precise statement of the braid stability result in \cite{Alves-Meiwes24}, as it will be important for us later on.

\begin{thm}[Alves-Meiwes 24] 
Let $(\Sigma,\omega)$ be a closed symplectic surface. Let $\phi \in \Ham(\Sigma,\omega)$, $H$ a Hamiltonian generating $\phi$ and $\mathcal{Y}$ be a finite collection of non-degenerate fixed points of $\phi$ in the same free homotopy class of loops $\alpha \in \tilde{\pi}_1(\Sigma)$.  Then, there exists $\epsilon>0$, such that for any non-degenerate Hamiltonian diffeomorphism $\psi$ whose Hofer distance to $\phi$ satisfies $d_{\Hofer}(\phi,\psi)<\epsilon$, there exist
	\begin{itemize}
		\item Hamiltonians $H,K$ generating $\phi, \psi$ respectively,
		\item and a finite collection $\mathcal{Z}$ of fixed points of $\psi$
	\end{itemize}
such that $\BB(\mathcal{Y},H)$ is freely isotopic as a braid to $\BB(\mathcal{Z},K)$. 
\end{thm}

The proof uses continuation cylinders between the Floer complexes $\CF(H)$ to $\CF(K)$ to construct the desired isotopy. More precisely, with a finite collection of disjoint continuation cylinders $\set{u_y}_{y\in\mathcal{Y}}$ such that
\begin{itemize}
	\item the negative end of $u_y$ is $y$, and
	\item the positive ends of $u_y$ are mutually distinct,
\end{itemize}
we construct a braid isotopy through a braid $\BB_s$ defined by
\[\BB_s:=\bigcup_{y\in\mathcal{Y}} \set{(t,u_y(s,t)):t\in S^1}.\]

\begin{figure}
    \centering
    \includegraphics[width=0.8\linewidth]{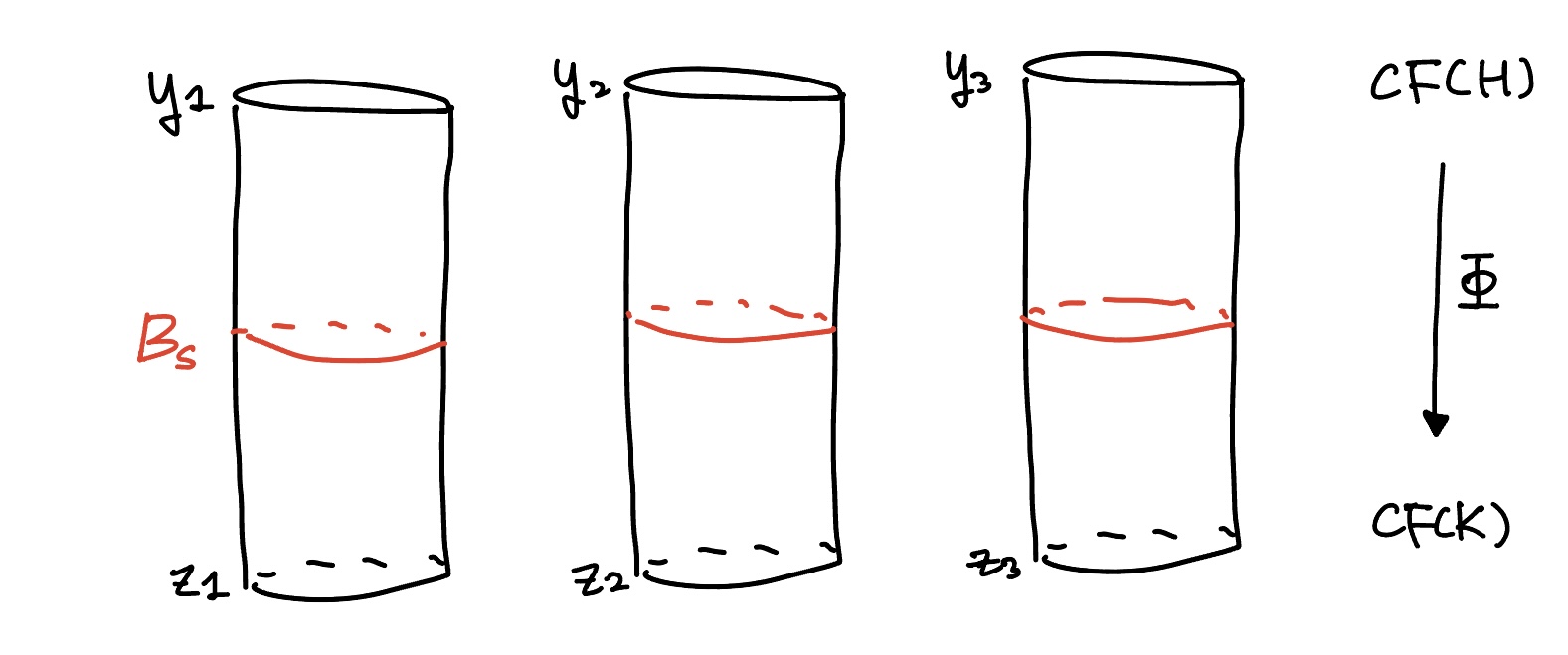}
    \caption{Free isotopy of braids along Floer continuation cylinders}
    \label{fig:plac1eholder}
\end{figure}

To prove that each $\BB_s$ is a braid, or equivalently $u_y(s,t)\ne u_{y'}(s,t)$ for $y \ne y'\in \mathcal{Y}$, we consider the graphs $\tilde{u}_\gamma$ on $\RR\times S^1\times \Sigma$ of the continuation cylinders. Using positivity of intersections for holomorphic curves and the fact that the graphs $\tilde{u}_y$ are holomorphic curves, one can show that $\BB_s$ is a braid for every $s \in \R$. Hence, the finite collection $\mathcal{Z}=\set{u_y(+\infty)}_{y\in\mathcal{Y}}$ of fixed points of $\psi$ satisfies:
\[\BB(\mathcal{Y}, H)\text{ is  freely isotopic as a braid to }\BB(\mathcal{Z},K).\]

The PSS-image spectral distance $d_\im$  on the group $\Ham(\Sigma,\omega)$ was recently introduced and studied by Connery-Grigg \cite{Connery-Grigg24}; it is induced by the PSS-image spectral norm $\gamma_{\im}$, see section \ref{Section:PSS-image-spectral-norm} for the definition and properties of $\gamma_{\im}$ and $d_{\im}$. It was inspired by the spectral distance on $\Ham(\Sigma,\omega)$, which was introduced by Schwarz in \cite{Schwarz00}, and which plays an important role in symplectic topology and dynamics; see also \cite{Viterbo92}. The distance $d_\im$ shares many properties with spectral distance. For example, it is $C^0$-continuous and Hofer-Lipschitz. Since braid stability holds for the $C^0$-topology and the Hofer-distance, one can ask if it also holds for the $d_\im$-distance. One of the goals of this paper is to give a positive answer to this question.

In our main braid stability result, which we will formulate now, we will additionally quantify the size of the allowed perturbations. To a finite collection of non-degenerate fixed points $\mathcal{Y}$ of $\phi$ 
we associate a positive number, its \textbf{Floer theoretic action gap} $\Delta(\mathcal{Y},\phi)$, see Section~\ref{sec:Actiongap} for the definition.

\begin{TheoremX}
\label{Main Theorem: Braid Stability under PSS-image spectral distance}
Let $(\Sigma,\omega)$ be a closed symplectic surface. Let $\phi \in \Ham(\Sigma,\omega)$, $H$ a Hamiltonian generating $\phi$ and $\mathcal{Y}$ be a finite collection of non-degenerate fixed points of $\phi$ in the same free homotopy class of loops $\alpha \in \tilde{\pi}_1(\Sigma)$. Then for any non-degenerate Hamiltonian diffeomorphism $\phi^\prime$ satisfying \[d_{\im}(\phi,\phi^\prime)<\Delta(\mathcal{Y}, \phi),\]
there exists
	\begin{itemize}
		\item $K$ generating $\phi^{\prime}$,
		\item and a finite collection $\mathcal{Z}$ of fixed points of $\phi^\prime$
	\end{itemize}
such that $\BB(\mathcal{Y},H)$ is freely isotopic as a braid to $\BB(\mathcal{Z},K)$.  

\end{TheoremX}

\begin{rem}
In the case that $\phi$ in the statement of Theorem \ref{Main Theorem: Braid Stability under PSS-image spectral distance} is non-degenerate, one can replace the Floer-theoretic action gap $\Delta(\mathcal{Y},\phi)$ by the usual action gap, namely, the minimal positive real number that can be expressed as the difference between the Hamiltonian actions of a fixed point $x \in \P(H)$ and a fixed point $y\in\mathcal{Y}$. This follows from the fact that usual action gap is always at most $\Delta(\mathcal{Y},\phi)$; see section \ref{sec:Actiongap} for precise definitions.   
\end{rem}


Our strategy to prove this is the following: we first take $H$ and $K$ generating $\phi$ and $\psi$ that satisfy \[d_\im(\phi, \psi) = \gamma_\im([\Phi_H]^{-1} \circ [\Phi_K])=\gamma_\im([\Phi_{\overbar{H}\#K}]).\footnote{See Definition \ref{def:concatenation-and-inverse-Hamiltonians} in Section \ref{Section: Product structures on Hamiltonian Floer homology} for the definitions of $\overbar{H}$ and $\overbar{H}\#K$.}\]
Next, we construct two Hamiltonian $1$-forms on a cylinder with Hofer norms bounded above by $c_\im([\Sigma], \overbar{H}\#K)$  and $c_\im([\Sigma], \overbar{K}\#H)$ each\footnote{For the definition of the PSS-image spectral invariant $c_{\im}$ see Definition \ref{Defintion: PSS-image spectral invariants}.}. The key difference, in comparison to \cite{Alves-Meiwes24}, is that we allow Hamiltonian $1$-forms possibly with non-zero $ds$ component: these cylinders will be obtained by gluing certain Floer disks (which appear in the PSS-map) and Floer pair-of-pants (which appear in the pair-of-pants product). This construction can be understood as the quantitative version of
\begin{center}
	Continuation map $=$ Product with the continuation element,
\end{center}
in the spirit of the work of Kislev and Shelukhin \cite{Kislev-Shelukhin21}.
For any Hamiltonian $G$, $c_\im([\Sigma],G)$ is bounded above by $\norm{G}_{\Hofer}$. Since every Hamiltonian $1$-form on a cylinder induces the same map on the level of Hamiltonian Floer homology, we call the filtered maps defined by such Hamiltonian $1$-forms `enhanced Floer continuation maps'. The construction of such Hamiltonian $1$-forms is presented in detail in Section \ref{Section: Proof of main Theorem subsection}. 

We are then in a position to adapt the arguments in \cite{Alves-Meiwes24} and construct a braid isotopy using the enhanced Floer continuation cylinders. More precisely, we show that $\gamma_\im(\phi,\psi) < \Delta(\mathcal{Y},\phi)$ implies that there exist a finite collection $\mathcal{Z}$ of fixed points of $\psi$ and bijections $\alpha:\mathcal{Y}\to \mathcal{Z}$ and $\beta:\mathcal{Z}\to\mathcal{Y}$ with the following property: for each $y\in\mathcal{Y}$ there exist an enhanced Floer continuation cylinder ${u}_y$ from $y$ to $\alpha(y)$ and an enhanced Floer continuation cylinder ${v}_y$ from $\alpha(y)$ to $\beta(y)$. By Gromov's trick, graphs of continuation cylinders are holomorphic cylinders in $\RR\times S^1 \times \Sigma$ endowed with an appropriate almost complex structure $J$. Using positivity of intersections for holomorphic curves and the fact that $\alpha$ and $\beta$ are bijections, we conclude that the graphs of the continuation cylinders $u_y$ from $y$ to $\alpha(y)$ are mutually disjoint. The intersections of the graphs of the $u_y$ with $\set{s}\times S^1\times \Sigma$ gives the free isotopy of geometric braids between $\BB(\mathcal{Y},H)$ and $\BB(\mathcal{Z},K)$. See Section \ref{Section: Isotopy via enhanced Floer cylinders} for the detailed proof.

Theorem \ref{Main Theorem: Braid Stability under PSS-image spectral distance} is the main ingredient in the proof of Theorem \ref{Main theorem: Topological entropy after perturbation supported on a disk}.  Another consequence of  Theorem \ref{Main Theorem: Braid Stability under PSS-image spectral distance} to the topological entropy is the following result: 
\begin{TheoremX} \label{Main Theorem: Lower semicontinuity of topological entropy}
	Let $(\Sigma,\omega)$ be a closed symplectic surface. The topological entropy of Hamiltonian diffeomorphisms of $(\Sigma,\omega)$ is lower semicontinuous with respect to the PSS-image spectral distance $d_\im$.
\end{TheoremX}


The lower semicontinuity of the topological entropy with respect to the Hofer distance was first established in \cite{Alves-Meiwes24} using braid stability for the Hofer distance and \cite[Theorem B.1]{Alves-Meiwes24}, which says that the topological entropy of a surface diffeomorphism of a surface can be fully recovered from the set of all free isotopy classes of braids defined by all possible finite invariant sets, i.e.,
\[\BB(\phi):=\set{[\BB(\mathcal{Y},\Phi)]: \phi(\mathcal{Y})=\mathcal{Y}, ~\# \mathcal{Y} <\infty},\]
where $\Phi$ is some choice of isotopy between $\id_\Sigma$ and $\phi$.  
It thus follows that braid stability with respect to a given topology implies the lower semicontinuity of the topological entropy in that topology. Braid stability for the Hofer distance was further studied in \cite{Hutchings23, Morabito25, Gruber}. We note that in \cite{Gruber}, braid stability was proved by establishing new Floer theory for braids without assuming that the braids lie in the same homotopy class as in \cite{Hutchings23}.


Combining Theorem \ref{Main Theorem: Lower semicontinuity of topological entropy} with the results of \cite{CGG} one obtains that the following stability result for the barcode entropy of Cineli-Ginzburg-Gurel in the setting of Hamiltonian diffeomorphisms of closed surfaces.

\begin{cor}
The barcode entropy $\hbar$ is lower semicontinuous with respect to the PSS-image spectral distance $d_{\im}$ on closed symplectic surfaces.
\end{cor}

The continuity properties of $h_{\rm top}$ with respect to different topologies on spaces of dynamical systems has been much studied. The topological entropy is a measure of the complexity of a dynamical system, and it is interesting to investigate if dynamical complexity is stable under perturbation of a system. For example, the combined results of Yomdin \cite{Yomdin1987} and Newhouse \cite{Newhouse} imply that $h_{\rm top}$ is continuous with respect to the $C^\infty$-topology on the space of $C^\infty$-smooth diffeomorphisms of a closed surface. The analogous result does not hold for higher dimensional manifolds, as showed in \cite{Misiurewicz}. Using the fundamental work of Katok \cite{Katok}, Nitecki showed in \cite{nitecki} that $h_{\topo}$ is lower semicontinuous with respect to the $C^0$-topology on the space of $C^{1+\epsilon}$-diffeomorphisms of a closed surface, for any $\epsilon>0$.   


\subsection{Related developments}

The relationship between symplectic topology and topological entropy of symplectomorphisms and Reeb flows has been studied extensively and fruitfully in recent years using various methods, in a line of research initiated by \cite{Polterovich02}. For Hamiltonian diffeomorphisms, a non-exhaustive list of works investigating the relationship of Floer theory and the topological entropy of Hamiltonian diffeomorphisms includes: \cite{Alves-Meiwes24, chor-meiwes, Cineli25, Frauenfelder-Schlenk05, Hutchings23, MeiwesLagrangian, CGG, Cineli-Ginzburg-Gurel22b}, and references therein.

A closely related question is the study the topological entropy of Reeb flows on contact manifolds using symplectic topological methods.
There is an abundance of contact manifolds for which the topological entropy or the exponential orbit growth rate is positive for all Reeb flows. Examples and dynamical properties of those manifolds are investigated in \cite{AASS, Alves-Cylindrical,Alves-Anosov, A2, AlvesColinHonda2017, AlvesMeiwes2018, FrauenfelderSchlenk2006, MacariniSchlenk2011}. Some of these results generalize to positive contactomorphisms \cite{Dahinden2,Dahinden}, and results on the dependence of some lower bounds on topological entropy with respect to their positive contact Hamiltonians have been obtained in \cite{DahindenC0}. The barcode entropy introduced in \cite{CGG}, which relates the topological entropy to persistence modules appearing from Morse and Floer homology, has also been studied in the setting of Reeb and geodesic flows: see \cite{Fender-Lee-Sohn23,Fernandes24,Cineli-Ginzburg-Gurel-Mazzucchelli24,Ginzburg-Gurel-Mazzucchelli22} and references therein.


In the setting of Reeb flows, a natural metric on the space of contact forms from a symplectic perspective is the \(C^{0}\)-distance. The results of \cite{Alves-Meiwes24} and the present paper have natural analogues in this context. For geodesic flows of Riemmanian metrics on surfaces, an analogue of the braid stability phenomenon is studied in \cite{ADMM,Alves-Mazzucchelli}. In the article  \cite{ADMP} of the first two authors, together with Lucas Dahinden and Abror Pirnapasov, it is proved that the lower semicontinuity of $h_{\topo}$ for \(3\)-dimensional Reeb and geodesic flows with respect to this \(C^{0}\)-distance holds for $C^\infty$-generic contact forms. This work mixes holomorphic curve techniques alongside the topological methods introduced in \cite{Meiwes-linking,AlvesPirnapasov}.

\newpage

\textbf{Acknowledgements}: This project was initiated at the conference \textit{Symplectic Dynamics (INdAM Meeting)}, held at Istituto Nazionale di Alta Matematica ''Francesco Severi" in Rome. We thank the organizers for their hospitality and for providing a stimulating environment. We also thank the organizers of the conference \textit{From Smooth to $C^0$ Symplectic Geometry: Topological Aspects and Dynamical Implications}, held at CIRM in Marseille, as well as the Chair of Analysis and Geometry at Universität Augsburg, where this research was further developed. Part of this work is based on material from the third author's PhD thesis, developed as part of this project. The third author thanks Yonghwan Kim for inspiring and helpful discussions, and Jungsoo Kang for his support. 

\section{Background}

\subsection{Hamiltonian Floer homology on closed symplectic surfaces} \label{sec:Hamonclosedsurfaces}

In this subsection, we review Hamiltonian Floer homology on closed symplectic surfaces. We first review the setup of Hamiltonian Floer homology on a general closed symplectic manifold $(M,\omega)$ under a technical assumption and then specialize to closed symplectic surfaces. 
We also discuss Hamiltonian Floer homology of non-contractible orbits. 

Let $\tilde{\pi}_1(M)$ be the set of free homotopy classes of loops in $M$, and denote by~$\cont$ the class of contractible loops. A disk capping of a smooth contractible  loop $y$ is a smooth map $w:\D^2 \to M$ with $w(e^{2\pi i t})=y(t)$: under these conditions we refer to the pair $(y,w)$ as a loop with a capping.

For each non-contractible class $\alpha \in \tilde{\pi}_1(M)$, we fix a smooth representative $\eta_\alpha$. A cylindrical capping of a smooth loop $y$ in $\alpha$ is a smooth map \linebreak $w:[0,1]\times S^1 \to \Sigma$ with $w(0,t)=\eta_{\alpha}(t)$ and $w(1,t)=y(t)$: under these conditions we refer to the pair $(y,w)$ as a loop with a capping.

Among cappings of a loop $y$, we consider the following equivalence relation:
\begin{equation}
	\label{Equation: Equivalence of cappings}
	(y,w)\sim (y, w^\prime) \quad\Longleftrightarrow\quad \omega(w \# \bar{w}^\prime)=0,~ c_1(w \# \bar{w}^\prime)=0,
\end{equation}
where $w {\#} \bar{w}^\prime$ is the sphere or torus obtained by gluing $w$ and $w^\prime$ along their matching boundaries. An equivalence class $[y,w]$ is called a capped loop.

We define Hamiltonian action $\mathcal{A}_H(y,w)$ of a loop with a capping $(y,w)$ by
\begin{equation}
	\label{Equation: Hamiltonian action functional}
	\mathcal{A}_H(y,w):=\int_0^1 H(t,y(t))~dt - \int  w^*\omega.
\end{equation}
The critical points of $\mathcal{A}_H$ are precisely the  loops with cappings $(y,w)$ where $y$ is a $1$-periodic orbits of $X_H$. 

Recall that we denote by $\mathcal{P}(H)$ the set of $1$-periodic orbits of $X_H$. For each $\alpha \in \tilde{\pi}_1(M)$, we let $\mathcal{P}(H;\alpha)$ be the subset of elements of $\mathcal{P}(H)$ which belong to $\alpha$: this gives a partition of $\mathcal{P}(H)$. We define $\widetilde{\mathcal{P}}(H;\alpha)$ as the set of all capped loops $[y,w]$ where $y \in \mathcal{P}(H;\alpha)$.

Recall that a $1$-periodic orbit $y \in \mathcal{P}(H)$ is called non-degenerate if it does not admit $1$ as an eigenvalue for its linearized return map, and that $H$ is said to be non-degenerate if every element of $\mathcal{P}(H)$ is non-degenerate. In this case, $\mathcal{P}(H)$ is a finite set.

When $H$ is non-degenerate and $y$ is a contractible $1$-periodic orbit of $\phi^t_H$, we can associate to each capped orbit $[y,w]$ an integer-valued Conley-Zehnder index $\CZ([y,w])$ as in \cite{Robbin-Salamon93}. 

We first define $\CZ$ for pairs $(y,w)$ where $y\in \mathcal{P}(H)$ and $w$ is a capping for $y$. We first treat the case that $y$ is  contractible. In this case one can define  $\CZ(y,w)$ using the recipe presented in \cite{Robbin-Salamon93}. We adopt the following convention for the $\CZ$. For a $C^2$-small Morse function $H:M\to\RR$, its critical points $x$ are also (constant) $1$-periodic orbits of $\phi^t_H$. We adopt the convention that for each critical point $x$ of $H$ we have
\[\CZ(x, w_x)=\Morse(x),\]
where $\Morse(x)$ is the Morse index of $x$, and $w_x$ denotes the constant capping. Note that this convention differs from that in \cite{AD}, where the Conley-Zehnder index  is given by $\CZ(x,w_x)=\Morse(x)-n$.

We now recall the definition of the Conley-Zehnder index for a pair $(y,w)$, where  $y\in \mathcal{P}(H)$ is non-contractible with free homotopy class $\alpha$ and $w$ is a cylindrical capping of $y$. We first fix a symplectic trivialization $\tau_\alpha$ on $\eta_\alpha^*TM$. This choice induces a homotopically canonical symplectic trivialization $\tau_{y,w}$ on $y^*TM$ for the loop with a capping $(y,w)$. The Conley-Zehnder index of an orbit with a capping, denoted by $\CZ(y,w)$, is then defined using the trivialization $\tau_{y,w}$ and the recipe from \cite{Robbin-Salamon93}.

For two different cappings $(y,w)$ and $(y,w^\prime)$, we have
\begin{equation}
	\label{Equation: Recapping}
	\begin{aligned}
	&\mathcal{A}_H(y,w)-\mathcal{A}_H(y,w^\prime)=- \omega(w\#\bar{w}^\prime),\\
	&\CZ(y,w)-\CZ(y;w^\prime)= - 2 c_1 (w\#\bar{w}^\prime).
	\end{aligned}
\end{equation}
It follows that the Conley-Zehnder index defined in this way is the same for equivalent cappings of the loop $y$. This implies that $\CZ([y,w])$ is indeed well defined for capped $1$-periodic orbits $[y,w]$.
Equation \eqref{Equation: Recapping} also implies that the Hamiltonian action is well-defined on the elements $[y,w]$ of $\widetilde{\mathcal{P}}(H)$. 

We define the spectrum $\Spec (H;\alpha)$ as
\[\Spec(H;\alpha):=\set{\mathcal{A}_H([y,w]): [y,w]\in\widetilde{\mathcal{P}}(H;\alpha)}.\]

\begin{rem}[Choice of generating Hamiltonian]
	The set $\mathcal{P}(H)$ corresponds to the set of fixed points of $\phi_H^1$, and the non-degeneracy of $H$ corresponds to the non-degeneracy of $\phi_H^1$. We know from \cite{Schwarz00}, that for a fixed point $y(0)$ of $\phi_H^1$, the free homotopy class of an $1$-periodic orbit is independent of the choice of the Hamiltonian $H$ that generates $\phi_H^1$. Hence, there exists a canonical bijection between $\mathcal{P}(H,\alpha)$ and $\mathcal{P}(G,\alpha)$ if $\phi_H^1=\phi_G^1$.	
\end{rem}



Recall that an almost complex structure $J$ on $M$ is said to be $\omega$-compatible if  $\omega(\cdot,\cdot) = \omega(J\cdot,J\cdot) $ and $\omega(\cdot,J\cdot)$ is a Riemannian metric on~$M$.

Let $(H,J)$ be a couple consisting of a non-degenerate Hamiltonian function $H:S^1\times M \to \RR$ and a smooth $S^1$-family $J=\set{J_t}_{t\in S^1}$ of $\omega$-compatible almost complex structures on $(M,\omega)$. The Floer equation associated to $(H, J)$ for a map $u: \R \times S^1 \to M$, is given by
\begin{equation}\label{eq:defFloereq}
	\partial_s u(s,t) + J_t(u(s,t))\big(\partial_t u (s,t) - X_{{H}}(t,u(s,t)\big) = 0.
\end{equation}
A solution of the equation is called a Floer cylinder for $(H,J)$. We abbreviate the Floer equation by $\mathcal{F}_{H,J}(u)=0$.

The energy ${E}(u)$ of a Floer cylinder is defined by the formula
\begin{equation}
	\label{Equation: definition of energy 1 - Floer equation}
	\begin{aligned}
	E(u):= &\int_{\RR \times S^1}\norm{\partial_s u(s,t)}_{J_t}^2 \,ds\,dt,\\
	=&\,\frac{1}{2}\int_{\RR\times S^1} \norm{\partial_s u(s,t)}_{J_t}^2+\norm{\partial_t u(s,t) - X_{H_{t}}(u(s,t))}_{J_t}^2\,ds \, dt.
	\end{aligned}
\end{equation}
where $\norm{v}_{J_{t}}^2=\omega(v, J_{t}v)$. Floer showed~\cite{Floer} that if a Floer cylinder has finite energy, then there exist $1$-periodic orbits $y_-$ and $y_+$ in $\mathcal{P}(H)$ such that \[\lim_{s\,\to\,\pm \infty}u(s,\cdot) = y_\pm(\cdot),\]
where the convergence can be taken in the $C^\infty$-topology. Here, we think of $u(s,\cdot)$ as a one parameter family of loops in $M$.
In short, we denote the limits $y_\pm$ by $u(\pm \infty)$.

The Floer cylinder $u$ gives a free homotopy between $u(-\infty)$ and $u(+\infty)$. Hence, Floer cylinders can only exist between $1$-periodic orbits in the same free homotopy class. It is well known that the energy of a Floer cylinder satisfies $E(u) = {\mathcal{A}}_H([y_-,w_-]) - {\mathcal{A}}_H([y_+,w_-\#u])$ for any capping $w_-$ of $y_-$; see for example \cite{AD}.

The energy of a Floer cylinder is by definition non-negative. The only Floer cylinders with energy equal to $0$, are those of the form $u(s,t)= y(t)$, where $y$ is a $1$-periodic orbit of $\phi^t_H$. Such a cylinder $u_y(s,t)= y(t)$ is called the trivial cylinder over $y$.

For each $\alpha \in \tilde{\pi}_1(M)$ and each integer $k\in \ZZ$, the Floer chain group $\CF_k(H;\alpha)$ is the $\Z_2$ vector space whose generators are the capped orbits $[y,w]$ in $\widetilde{\mathcal{P}}(H;\alpha)$ with Conley-Zehnder index $k$. More precisely, when the set of such capped orbits is finite, we define $\CF_k(H;\alpha)$ as follows:
\begin{equation}
\label{Equation: Floer chain complex}
\CF_k(H;\alpha) := \bigoplus_{\substack{{[y,w]\,\in\,{\widetilde{\mathcal{P}}}(H;\alpha),} \\ {\CZ([y,w])=k}}} \Z_2 \cdot [y,w].
\end{equation}

When the set of such capped orbits is infinite, we complete the direct sum over the set $\widetilde{\mathcal{P}}(H;\alpha)$ using downward completion. A detailed construction will be considered in the cases where $M$ is $T^2$ and $\alpha \ne \cont$. Otherwise, the set of equivalence classes of capped orbits of a fixed index is finite.

We now fix a free homotopy class $\alpha\in \tilde{\pi}_1(M)$. In the discussion that follows, all the capped orbits are assumed to be in $\widetilde{\mathcal{P}}(H;\alpha)$. Given distinct two capped orbits $[y_-,w_-]$ and $[y_+, w_+]$, we let ${\mathcal{M}}([y_-,w_-], [y_+,w_+]; H,J)$ denote the moduli space of Floer cylinders $u$ that are asymptotic to $y_\pm$ at $\pm \infty$ and satisfy $[y_+, w_- \# u]  = [y_+,w_+]$. {The elements in this moduli space are equivalent classes of Floer cylinders. Here two Floer cylinders $u$ and $u'$ are considered equivalent if there exists $s_0\in \RR$ such that $u'= s_0 \cdot u $ where $s_0 \cdot u(s,t):=u(s+s_0, t)$}. Such an $\RR$-action is free except for trivial cylinders. 

As shown in \cite{FHS}, the moduli space ${\mathcal{M}}([y_-,w_-], [y_+,w_+], {H},J)$ is a smooth manifold of dimension ${\CZ}(y_-,w_-) - {\CZ}(y_+,w_+) - 1$ for a $C^\infty$-generic choice of $J$. Such a choice of $J$ is called regular. A pair $(H,J)$, consisting of a non-degenerate Hamiltonian $H$ and a regular $S^1$-family of almost complex structures $J$, will be referred to as Floer data.

Each moduli space ${\mathcal{M}}([y_-,w_-], [y_+,w_+], {H},J)$ can be compactified by the Floer-Gromov compactification, namely, by adding tuples of Floer cylinders together with possible $J_t$-holomorphic spheres. We assume the following condition on~$(M,\omega)$.

\begin{ass}\label{Assumption: no bubbling}[\textbf{No bubbling}] 
There are no non-constant $J$-holomorphic spheres $u$ with $c_1(u)\leq 1$ for any $\omega$-compatible almost complex structure $J$.
\end{ass}

\medspace


Assumption \ref{Assumption: no bubbling} implies that each $0$-dimensional moduli space \linebreak ${\mathcal{M}}([y_-,w_-], [y_+,w_+]; H, J)$ is compact and therefore a finite set for a Floer data $(H,J)$. We then define
\[
C([y_-,w_-],[y_+,w_+])= \#_2 {\mathcal{M}}([y_-,w_-], [y_+,w_+]; H, J),
\]
for capped orbits with $\CZ([y_-,w_-])=\CZ([y_+,w_+])+1$. The Floer differential $d_{H,J}:\CF_k(H;\alpha)\to \CF_{k-1}(H;\alpha)$ is then defined by
\begin{equation}
\label{Equation: Floer differential}
d_{H,J} ([y_-,w_-])=\sum_{\substack{{[y_+,w_+]\,\in\,{\widetilde{\mathcal{P}}}(H;\alpha),} \\ {\CZ(y_+,w_+)=k-1}}} C([y_-,w_-],[y_+,w_+]) \cdot [y_+,w_+]
\end{equation}
for each $[y_-,w_-]$. The differential $d_{H,J}$ is extended linearly to all of $\CF_k(H;\alpha)$.

The compactification of the $1$-dimensional moduli spaces \linebreak $\mathcal{M}([y,w], [y',w']; H, J)$ is a compact $1$-dimensional manifold with boundary.
The boundary strata of these $1$-dimensional moduli spaces ensure that $d_{H,J}^2 = 0$. We define the Hamiltonian Floer homology $\HF_*(H,J;\alpha)$ as the homology of the Floer chain complex $\left(\CF_*(H;\alpha), d_{H,J}\right)$. 

For Floer data $(H_-,J_-)$ and $(H_+, J_+)$, there exists a natural map, called the continuation homomorphism,
\[\Phi_{(H_-,J_-)\to(H_+,J_+)}:\HF_*(H_-,J_-;\alpha) \to \HF_*(H_+, J_+;\alpha).\]
The map is defined by counting Floer continuation cylinders, which are solutions to a variation of the Floer equation \eqref{eq:defFloereq}. In particular, the map  $\Phi_{(H_-,J_-)\to(H_+,J_+)}$ gives an isomorphism between the two Hamiltonian Floer homologies $\HF_*(H_-,J_-;\alpha)$ and $\HF_*(H_+, J_+;\alpha)$. The construction of this map and further discussion on its properties are postponed to Section \ref{Subsubsection: Continuation maps}.

For closed symplectic surfaces $(\Sigma,\omega)$, the Hamiltonian Floer homology $\HF_k(H,J;\alpha)$ for an integer $k \in \ZZ$ is given by
\[\HF_k(H,J;\alpha)=\begin{cases}
	H_k(\Sigma,\ZZ_2), & \alpha=\cont\\
	0, &\alpha \ne \cont\\
\end{cases}
\]
for $\Sigma \ne S^2$, and
\[\HF_k(H,J;\cont)=\begin{cases}
	\ZZ_2, & k\text{ is even}\\
	0, & k\text{ is odd}
\end{cases}
\]
for $\Sigma=S^2$.

\subsubsection{Closed surfaces different from \texorpdfstring{$S^2$}{S2} and \texorpdfstring{$T^2$}{T2}}\label{sec:Hamonhighergenus} 

\

In this case, $(\Sigma,\omega)$ is symplectically aspherical and atoroidal, i.e., spherical and toroidal homology classes in $H_2(\Sigma,\mathbb{Z}_2)$ vanish for the cohomology classes $[\omega]$ and $c_1(T\Sigma)$ in $H^2(\Sigma, \mathbb{Z}_2)$. 
It follows that the equivalence \eqref{Equation: Equivalence of cappings} holds automatically, i.e.
\[\omega(w \# \bar{w}^\prime)=0,~ c_1(w \# \bar{w}^\prime)=0\]
for any two cappings $w, w^\prime$ over $y$. Hence, the set of capped orbits $\widetilde{\mathcal{P}}(H;\alpha)$ can be identified with the set of $1$-periodic orbits $\mathcal{P}(H;\alpha)$. This allows us to simplify the notation as follows. We denote the Hamiltonian action and Conley-Zehnder index of $[y,w]$ by $\mathcal{A}_H(y)$ and $\CZ(y)$, respectively, instead of $\mathcal{A}_H([y,w])$ and $\CZ([y,w])$.

When defining Floer homology, Assumption \ref{Assumption: no bubbling} holds because the symplectic area of any non-constant pseudo-holomorphic sphere is positive, while integration of $\omega$ over spherical homology classes is always zero, since every sphere in $\Sigma$ is contractible. For Floer data $(H,J)$, the number of generators of each Floer chain group is finite, and we thus have
\[
\CF_k(H;\alpha) := \bigoplus_{\substack{{y\,\in\,{\mathcal{P}}(H;\alpha),} \\ {\CZ(y)=k}}} \Z_2 \cdot y,
\]
for each free homotopy class $\alpha \in \tilde{\pi}_1(\Sigma)$ and for each integer $k\in\ZZ$. As mentioned above, the Hamiltonian Floer homology $\HF_*(H,J)$ is given by
\[\HF_k(H,J;\alpha)=\begin{cases}
	H_k(\Sigma,\ZZ_2), & \alpha=\cont\\
	0, &\alpha \ne \cont\\
\end{cases}.
\]
We refer the reader to \cite{AD} for a proof of this result.

\subsubsection{The case \texorpdfstring{$\Sigma = S^2$}{Sigma = S2}}\label{sec:Hamonsphere}

\ 

The fundamental group of $S^2$ is trivial, so the only class in $\tilde{\pi}_1(S^2)$ is the class of contractible loops. For simplicity, we omit the class $\cont$ when discussing the Hamiltonian Floer homology on $S^2$.

In this case, $(S^2,\omega)$ is a (positively) monotone sympletic manifold, meaning that there exists a constant $N>0$ such that
\[N \cdot \omega(A)= c_1(A),\]
for $A \in \pi_2(M)$. For $(S^2,\omega)$ with area equal to $1$, the constant $N$ is $2$. It follows that two cappings $w$ and $w^\prime$ over the same loop are equivalent if and only if
$\omega(w\# \bar{w}^\prime)=0$.

The group $\pi_2(S^2)$ acts on the set $\widetilde{\mathcal{P}}(H)$ by $A\cdot[y,w]:=[y,A\#w]$, where $A\#w$ denotes the equivalence class of the connected sum $v\#w$, with $v\in A$ satisfying $v(\infty)=w(0)$. This $\pi_2(S^2)$-action is free and transitive on the set of capped orbits over the same $1$-periodic orbit.

From Equation \eqref{Equation: Recapping}, we have the following relations:
\begin{align*}
&\mathcal{A}_H([y,w])-\mathcal{A}_H(A\cdot[y,w])=\omega(A),\\
&\CZ([y,w])-\CZ(A\cdot[y, w])=2 c_1 (A)=4\,\omega(A).
\end{align*}
By the second equation, distinct equivalence classes over the same $1$-periodic orbit have different Conley-Zehnder indices. As a result, for any $k\in \ZZ$, there are finitely many capped orbits with Conley-Zehnder index $k$.

When defining Floer homology, Assumption \ref{Assumption: no bubbling} holds: there are no $A\in\pi_2(S^2)$ with $\omega(A)>0$ and $c_1(A) \le 1$. The Floer chain complex $\CF_k(H)$ is defined as in \eqref{Equation: Floer chain complex}, and the Floer differential $d_{H,J}$ is defined as in \eqref{Equation: Floer differential}. The Hamiltonian Floer homology $\HF_*(H,J)$ is given by
\[\HF_k(H,J)=\begin{cases}
	\ZZ_2, & k\text{ is even}\\
	0, & k\text{ is odd}
\end{cases}.
\]

\medskip

\subsubsection{The case \texorpdfstring{$\Sigma = T^2$}{Sigma = T2}}\label{sec:Hamontorus}
In this case, $\Sigma$ is sympletically aspherical, meaning that the equivalence \eqref{Equation: Equivalence of cappings} given by
\[\omega(w \# \bar{w}^\prime)=0,~ c_1(w \# \bar{w}^\prime)=0\]
holds automatically for any two disk cappings $w, w^\prime$ over the same contractible loop. This implies that the Assumption \ref{Assumption: no bubbling} holds, since there are no non-constant pseudo-holomorphic spheres in $T^2$. As in Section \ref{sec:Hamonhighergenus}, one obtains that the Hamiltonian Floer homology $\HF_*(H,J,\cont)$ is given by
\[\HF_k(H,J;\cont)=H_k(T^2,\ZZ_2),\]
for each Floer data $(H,J)$ and each integer $k$; see \cite{AD}.

For the remainder of this section, we treat the case of free homotopy classes $\alpha\ne \cont$. For cylindrical cappings, the condition $c_1(w \# \bar{w}^\prime)=0$ still automatically holds since $c_1(TT^2)$ vanishes. However, $\omega(w \# \bar{w}^\prime)$ may not be zero. For each $\cont \ne \alpha \in \tilde{\pi}_1(T^2)$, we define the group $\Gamma_\alpha$ by\[\Gamma_\alpha:=\frac{\pi_1(\mathcal{L}T^2, \eta_\alpha)}{\ker \omega \cap \ker c_1}=\frac{\pi_1(\mathcal{L}T^2, \eta_\alpha)}{\ker \omega};\]
recall that $\eta_\alpha$ is the reference loop in the free homotopy class $\alpha$ that we fixed previously.

The group $\Gamma_\alpha$ acts on the set $\widetilde{\mathcal{P}}(H;\alpha)$ by $A\cdot[y,w]:=[y,A\#w]$, where $A\#w$ denotes the equivalence class of the cylindrical capping $v\#w$ of $y$, with $v\in A$. The class $A\#w$ is the element of  $\Gamma_\alpha$ represented by $v\#\widetilde{w}$ where $\widetilde{w}:[0,1]\times S^1 \to T^2$ is a representative of $w\in \Gamma_\alpha$ and $v:[0,1]\times S^1 \to T^2$ is a representative of $A \in \Gamma_\alpha$. 
This $\Gamma_\alpha$-action is free and transitive on the set of cappings of a fixed $1$-periodic orbit $y \in  {\mathcal{P}}(H,\alpha)$.

From Equation \eqref{Equation: Recapping}, we have the following relations:
\begin{align*}
	&\mathcal{A}_H([y,w])-\mathcal{A}_H(A\cdot[y,w])=\omega(A),\\
	&\CZ([y,w])-\CZ(A\cdot[y, w])=0.
\end{align*}
By the second equation, distinct equivalence classes over the same $1$-periodic orbit share the same Conley-Zehnder index. As a result, for any $k\in \ZZ$, there are either infinitely many or no elements in $\widetilde{\mathcal{P}}(H,\alpha)$ with Conley-Zehnder index $k$. {Note that the capped orbits over the same $1$-periodic orbit are classified by their Hamiltonian action.}

To define the Floer chain complex $\CF_k(H;\alpha)$ for each $k\in \ZZ$, we consider formal sums of the form
\[\beta=\sum_{\substack{{[y,w]\in \widetilde{\mathcal{P}}(H;\alpha)} \\ {\CZ(y,w)=k}}} a_{[y,w]}\cdot[y,w],\quad a_{[y,w]}\in \ZZ_2.\]
We call the formal sum $\beta$ a Floer chain of degree $k$ if
\begin{equation}
\label{Equation: downward Novikov condition}
	\# \set{[y,w]:a_{[y,w]}=1\text{ and }\mathcal{A}_H ([y,w]) \ge a}<+\infty
\end{equation}
for any $a \in \RR$. For each $k\in \ZZ$, we define $\CF_k(H;\alpha)$ the set of Floer chains of degree $k$. {Hence, the Floer chain complex is defined by the downward-completion of the direct sum \eqref{Equation: Floer chain complex}.}

The Floer differential $d_{H,J}:\CF_k(H;\alpha)\to\CF_{k-1}(H;\alpha)$ is defined as in \eqref{Equation: Floer differential} for each capped orbit $[y,w]$, and extend it linearly to all of $\CF_k(H;\alpha)$; here it is used that $T^2$ satisfies Assumption \ref{Assumption: no bubbling}. The image of the differential still satisfies \eqref{Equation: downward Novikov condition} from the fact that the formal sum $d_{H,J}([y,w])$ satisfies the condition \eqref{Equation: downward Novikov condition}. Moreover, $d_{H,J}^2=0$, and we define the Hamiltonian Floer homology $\HF_k(H,J;\alpha)$ as the homology of the Floer chain complex $\left(\CF_k(H;\alpha), d_{H,J}\right)$. Then, for $\alpha \ne 0$, it is known that
\[\HF_k(H,J;\alpha)=0.\]

\subsubsection{Action filtration on Hamiltonian Floer homology}

\

The following definitions apply to Floer chains in any of the versions of Floer chain complexes considered above beyond the symplectic surfaces.

\begin{defn} \label{Definition: Support of a chain}
Let $\sigma \in CF_*(H;\alpha)$ be a Floer chain. Then there exists a unique subset $\{[y_i, w_i]\}_{i \in \mathcal{I}}$ of distinct elements of $\widetilde{\mathcal{P}}(H;\alpha)$ such that $\sigma = \sum_{i \in \mathcal{I}}[y_i, w_i]$.
Then the support $\mathrm{supp}(\sigma)$ defined to be the set $\{[y_i,w_i]\}_{i \in \mathcal{I}}$.
\end{defn}

\begin{defn} \label{Definition: Action of a chain}
Let $\sigma \in CF_k(H;\alpha)$ be a Floer chain of degree $k$. Then, the action $\mathcal{A}_H(\sigma)$ is defined as $\max_{[y,w]\in \mathrm{supp}(\sigma)} \mathcal{A}_H([y,w])$.
\end{defn}

Note that in the case of $\Sigma=T^2$, $\alpha \ne \cont$, the condition \eqref{Equation: downward Novikov condition} ensures that the maximum appearing in Definition \ref{Definition: Action of a chain} exists.

Next, we introduce a filtered version of Hamiltonian Floer homology. For any $\alpha\in \tilde{\pi}_1(M)$ and $a\in \RR$, let $\widetilde{\mathcal{P}}^{a}(H;\alpha)$ denote the set of capped orbits in $\widetilde{\mathcal{P}}(H;\alpha)$  with the Hamiltonian action less than $a$. We define the subcomplex $(\CF_k^a(H;\alpha),d_{H,J})$ as the subset of chains of action less than $a$. Equivalently,
\[
\CF^a_k(H;\alpha) := \bigoplus_{\substack{{[y,w]\,\in\,{\widetilde{\mathcal{P}}}^a(H;\alpha),} \\ {\CZ([y,w])=k}}} \Z_2 \cdot [y,w],
\]
for the case $\Sigma \ne T^2$ or $\alpha = \cont$. For the case $\Sigma=T^2$ and $\alpha \ne \cont$, we define the subcomplex $\CF_k^a(H;\alpha)$ as the set of Floer chains of degree $k$ satisfying
\[\# \set{[y,w]:a_{[y,w]}=1\text{ and }\mathcal{A}_H ([y,w]) \ge a} = \varnothing. \] The Floer differential $d_{H,J}$ descends to the subcomplex because the Hamiltonian action decreases along the differential. We then define $\HF^a_k(H,J;\alpha)$ as the homology of the chain complex $\left(\CF^a_k(H;\alpha), d_{H,J}\right)$.

For $a\leq b$ the inclusion $\iota_k^{a,b}$ of $\CF^a_k(H;\alpha)$ into $\CF^b_k(H;\alpha)$ induces a 
natural map $\iota_k^{a,b}:\HF^a_k(H,J;\alpha)\to \HF^b_k(H,J;\alpha)$. Here, by natural, we mean that $\iota_k^{b,c}\circ \iota_k^{a,b}=\iota_k^{a,c}$ for every $a \le b \le c$.

\medspace

\subsection{Action gap} \label{sec:Actiongap}

We define a positive real number associated to a set of non-degenerate fixed points which we call \textbf{Floer theoretical action gap} and discuss its 
behavior under small perturbations. 
Here, for simplicity, we consider a closed symplectic surface $(\Sigma,\omega)$ of area $1$. We note, however, that the discussion below extends to arbitrary closed symplectic manifolds. We begin with a preliminary definition.

Let $H:\Sigma \times S^1 \to \R$ be a possibly degenerate Hamiltonian, and $\phi= \phi^1_H$ be the Hamiltonian diffeomorphism that is generated by $H$.

\begin{defn}	\label{Definition: Action gap}
For any two $1$-periodic orbits $y$ and $y^\prime$ in $\mathcal{P}(H;\alpha)$, we define the {action gap} $\delta(y,y^\prime;\phi)$ of this pair of $1$-periodic orbits by
\begin{equation} \label{eq:action-gap-pair-of-orbits}
    \delta(y,y^\prime,\phi):=\inf |\mathcal{A}_H([y,w])-\mathcal{A}_H ([y^\prime, w^\prime])|\, 
\end{equation}
where the infimum is taken over all cappings $w$ and $ w^\prime$ of  $y$ and $y'$ respectively, {satisfying
$\mathcal{A}_H([y,w])\ne \mathcal{A}_H ([y^\prime, w^\prime])$.}
\end{defn}

Again, although the Hamiltonian action $\mathcal{A}_H(y,w)$ depends on the choice of a reference loop $\eta_\alpha \in \mathcal{L}_\alpha\Sigma$ for $\alpha \ne 0$, the difference in Hamiltonian action does not depend on the choice of $\eta_\alpha$. To see this, the action difference $\mathcal{A}_H([y,w])-\mathcal{A}_H ([y^\prime, w^\prime])$ can be written as 
\begin{equation} \label{eq:formula-action-difference}
    \int_{S^1} H(y(t))-H(y'(t))\,dt - \int \left( w \# \bar{w}'\right) ^*\omega,
\end{equation}
and $\delta(y,y^\prime;\phi)$ can be rewritten as the infimum of \eqref{eq:formula-action-difference} over the homotopy classes of cylinders $w \# \bar{w}'$ from $y$ to $y'$. We also note that the action gap also does not depend on the choice of generating Hamiltonian function; see \cite{Schwarz00}.

A direct computation shows that 
\[
\delta(y,y, \phi) = \inf \set{ \omega(u) > 0: u \in \pi_1(\mathcal{L}_\alpha \Sigma, \eta_\alpha)}.
\]
In particular, the action gap of $y$ and itself is given by
\begin{equation*}
\delta(y,y, \phi)=
\begin{cases}
    1 & \quad \text{if } \Sigma = S^2\\
    \mathrm{cov}(\alpha) & \quad \text{if } \Sigma = T^2 \text{ and } \alpha \ne 0\\
    \infty & \quad \text{else } 
\end{cases}
\end{equation*}
Here, $\mathrm{cov}(\alpha)$ is the covering number of $\alpha$; i.e. the largest positive integer $n$ such that $\alpha = (\beta)^n$ for some $\beta \in \pi_1(T^2)$.

\begin{defn}
Let $\mathcal{Y} $ be a finite subset of $ \mathcal{P}(H; \alpha)$. We define the action gap $\delta(\mathcal{Y},\phi)$ by
\begin{equation} \label{eq:definition-classical-action-gap}
\delta(\mathcal{Y},\phi):=\inf \set{\delta(y, y^\prime;\phi)>0: y \in \mathcal{Y}, ~y^\prime \in \mathcal{P}(H;\alpha)}.
\end{equation}
\end{defn}

In order to assume only non-degeneracy of the fixed points in $\mathcal{Y}$ rather than of $\phi$, as in Theorem \ref{Main Theorem: Braid Stability under PSS-image spectral distance}, we consider small perturbations of $\phi$ away from the set $\mathcal{Y}$. In this context, the classical action gap $\delta(\mathcal{Y}, \phi)$ presents two drawbacks: first, it is not necessarily lower semicontinuous even with respect to the $C^\infty$-topology; second, it can even vanish when $\phi$ is degenerate.

The Floer theoretical action gap below (Definition \ref{defn:FTAG}) allows us to address both issues. To define it, let $J = \{J_t\}_{t \in S^1}$ be a $S^1$-family of compatible almost complex structures on $(\Sigma, \omega)$.


\begin{defn}\label{defn:quasi-isolation}
 Let $\mathcal{Y} $ be a finite subset of non-degenerate elements of $ \mathcal{P}(H)$. We say that $\mathcal{Y}$ 
 is  \textbf{Floer $\delta$-isolated} (with respect to $J$) for a constant $\delta \ge0$  if 

\begin{itemize}
    \item There is no non-constant 
 $u:\R \times S^1 \to \Sigma$ and no $y \in \mathcal{Y}$ such that 
 \begin{itemize}
     \item $\mathcal{F}_{H,J}(u) = 0$ 
     \item $E(u) \leq \epsilon$
     \item $\lim_{s\to \infty}u(s,t)  = y(t)$ or $\lim_{s\to -\infty}u(s,t)  = y(t)$. 
     \end{itemize}
     \end{itemize}
  \end{defn}
Note that even though the given Hamiltonian $H$ is possibly degenerate, the limits  $\lim_{s\to \infty}u(s,t)  = y(t)$ or $\lim_{s\to -\infty}u(s,t)  = y(t)$ make sense since every periodic orbit $y \in \mathcal{Y}$ is nondegnerate.

\medskip

\begin{defn}\label{defn:FTAG}  
Let $\mathcal{Y} $ be a finite subset of non-degenerate elements of $ \mathcal{P}(H)$. The \textbf{Floer theoretical action gap} $\Delta({\mathcal{Y}},H)$ of $\mathcal{Y}$ is defined by
\begin{equation}
  \Delta({\mathcal{Y}},H)  := \sup_J \,\sup \{ \delta \geq 0 \ | \ \mathcal{Y} \mbox{ is Floer } \delta\mbox{-isolated with respect to } J \}, 
\end{equation}
where $J$ runs over all $S^1$-families of compatible almost complex structures. 
\end{defn}
We note that a similar idea was introduced in \cite{Alves-Meiwes24} under the term ``$\delta$-quasi-isolated'', although we work with a less restricted version here.

It is important to notice that in Definitions \ref{defn:quasi-isolation} and \ref{defn:FTAG} we do not require $(H,J)$ to be Floer data; i.e. we require neither non-degeneracy of $H$ nor regularity of $J$. In particular, these definitions make sense for degenerate Hamiltonians $H$. 
Moreover, the Floer theoretical action gap $\Delta(\mathcal{Y},H)$ depends only on the time-$1$ map $\phi_H^1$ and not on the specific Hamiltonian $H$. 
The reason for this is that if $\phi_H^1=\phi_G^1$, then there exists a bijection between the Floer cylinder of $(H, J)$ and those of $(G, J^\prime)$, where\footnote{For the definition of $G \# \overbar{H}$ see \ref{def:concatenation-and-inverse-Hamiltonians}.}
\[J^\prime_t := \left(\phi^t_{G \# \overbar{H}}\right)^* J_t\quad \text{for all }t\in S^1.\]
It then follows that the collection $\mathcal{Y}$ is automatically Floer $\Delta(\mathcal{Y}, H)$-isolated.
Because of this independence from $H$, we will from now denote by $\Delta(\mathcal{Y},\phi)$.

\begin{rem}
    For a Floer cylinder of a degenerate Hamiltonian, the limit of the loops exists,
    \[ \lim_{i\to\infty} u(s_i,-), \]
    for any sequence $s_i\to\pm\infty$, after passing to a subsequence. Although this limit may not be well-defined as an asymptotic orbit, the Hamiltonian action is still well-defined.
\end{rem}

By definition,  every finite subset $\mathcal{Y}$ is Floer $\delta(\mathcal{Y},\phi)$-isolated with respect to every $J$. Hence, we have
\begin{equation}
    \Delta(\mathcal{Y},\phi) \geq \delta(\mathcal{Y},\phi).
\end{equation}
Moreover, for any $\omega$-compatible almost complex structure $J$, the collection of non-degenerate orbits $\mathcal{Y}$ is Floer $\delta$-isolated for some $\delta > 0$.  
This is proved in Proposition \ref{Proposition: Stabillity of quasi-isolatedness} below, which also establishes the stability of $\Delta(\mathcal{Y}, \phi)$ with respect to the $C^\infty$-topology. 

\medskip

\begin{prop}
	\label{Proposition: Stabillity of quasi-isolatedness}
    Let $\mathcal{Y} $ be a finite subset of non-degenerate elements of $ \mathcal{P}(H)$. Then the following holds.
    \begin{enumerate}
        \item[a)] For any $S^1$-family of $\omega$-compatible almost complex structures $J$, there exists $\delta > 0$ such that $\mathcal{Y}$ is Floer $\delta$-isolated with respect to $J$.
        \item[b)] Suppose that $\{G_i\}_{i\in\mathbb{N}}$ is a sequence of Hamiltonians compactly supported in $(S^1\times\Sigma)\setminus\mathcal{B}(\mathcal{Y},H)$, converging to $0$ in the $C^\infty$-topology.
	Then for any given $\epsilon > 0$, there exists $N_\epsilon \in \N$ satisfying
	\[\Delta(\mathcal{Y},\phi_{H_i}^1) > \Delta(\mathcal{Y},\phi_{H}^1)-\epsilon,\]
    for every $i>N_\epsilon$, where $H_i=G_i\#H$.
    \end{enumerate}
\end{prop}

In statement b), note that $\mathcal{Y}$ remains a finite subset of $\mathcal{P}(H_i)$ consisting of nondegenerate $1$-periodic orbit. In particular, this statement says that the Floer theoretical action gap $\Delta(\mathcal{Y},\phi_H^1)$ is lower semicontinuous with respect to the $C^\infty$-topology on $C_c^\infty\bigl(S^1\times\Sigma\setminus\mathcal{B}(\mathcal{Y},H),\mathbb{R}\bigr).$

\begin{proof}
We first prove statement b), as the proof is more involved. We then explain how to adapt it to prove a). 
Let $J$ be a $S^1$-family of $\omega$-compatible almost complex structures such that $\mathcal{Y}$ is Floer $(\Delta(\mathcal{Y},\phi_{H}^1)-\epsilon)$-isolated with respect to $J$, i.e.,
\[E_{H,J}(u) > \Delta(\mathcal{Y},\phi_{H}^1)-\epsilon.\]
holds for every nontrivial Floer cylinder $u$ asymptotic to some $y\in\mathcal{Y}$.

We prove by contradiction. Take any sequence of almost complex structures $\set{J_i}_{i \in \N}$  such that $(H_i, J_i)$ converges to $(H,J)$ in $C^\infty$-topology. 
	Assume, for contradiction, that there exists a sequence of non-constant Floer cylinders $u_i$, i.e. $\mathcal{F}_{H_i, J_i}(u_i)=0$, satisfying the energy bound 
	\[
	0 < E_{H_i, J_i}(u) \le  \Delta(\mathcal{Y},\phi_{H}^1)-\epsilon.
	\]
	for every $i$, and  having at least one of the asymptotics $u(\pm \infty)$  in $\mathcal{Y}$. 
	
	Without loss of generality, and after passing to a further subsequence if necessary, we may assume that each $u_i$ is asymptotic to a fixed $y \in \mathcal{Y}$ at $+\infty$, i.e.,
	\[u_i(+\infty)=y.\]
    The case $u_i(-\infty)=y$ for a fixed $y \in \mathcal{Y}$ at $-\infty$ is treated identically.

    Since $y$ is nondegenerate, we can take a $S^1$-family of arbitrarily small Darboux balls $U_{y,t}\subset\Sigma$ centered at $y(t)$ for each $t\in S^1$, such that no other $1$-periodic orbits intersect the closure $\overbar{U}_y$ of their union $U_y := \bigcup\set{t}\times U_{y,t}$. 
	Since each $u_i$ is not a trivial cylinder, the other asymptotic $u_i(-\infty)$ is not contained in $\overbar{U}_y$. After reparametrizing in the $\RR$-direction, we assume
	\begin{align*}
		&u_i(0,t_i)\in \partial U_{y, t_i} \quad \text{for some }t_i\in S^1,\quad\text{and}\\
		&u_i(s, t) \in U_{y,t} \hspace{.9cm} \text{for all } (s,t) \in [0, \infty)\times S^1.
	\end{align*}
	
    Taking a subsequence of $u_i$, there exists a cylinder map $u:\mathbb{R}\times S^1 \setminus \Gamma \to \Sigma$ possibly with punctures at a finite subset $\Gamma\subset\mathbb{R}\times S^1$, that $u_i$ converges to $u$ in $C^\infty_{\rm loc}(\mathbb{R}\times S^1 \setminus \Gamma,\Sigma)$-topology. The limit solves the Floer equation of $(H,J)$, and has a finite energy:
    \[E(u) \le \liminf_{i\to\infty} E(u_i) =  \Delta(\mathcal{Y},\phi_{H}^1)-\epsilon.\]
    
    As in the removal of singularity of pseudoholomorphic curves, $u$ extends to a genuine cylinder map that solves the Floer equation for $(H,J)$ which we denote again by $u$. The set $\Gamma$ is characterised by the bubbling phenomenon. Moreover, since the symplectic form is exact on each Darboux ball $U_{y,t}$, there are no bubbling points in $[0,\infty)\times S^1$.
    
    Therefore, the extension satisfies
	\begin{align*}
		&u(0,t)\in \partial U_{y, t} \quad \text{for some }t\in S^1,\quad\text{and}\\
		&u(s, t) \in \overbar{U}_{y,t} \hspace{.7cm} \text{for all } (s,t) \in [0, \infty)\times S^1.
	\end{align*}
    Since $y$ is the only $1$-periodic orbit in $\overbar{U}_y$, the limit $u$ is a nontrivial Floer cylinder asymptotic to $y$ at $+\infty$. This contradicts the assumption that every nontrivial Floer cylinder for $(H,J)$ asymptotic to $y$ has energy at least $\Delta(\mathcal{Y},\phi_H^1)-\epsilon$. Thus, no such sequence $u_i$ can exist, and we conclude that $\mathcal{Y}$ remains Floer $(\Delta(\mathcal{Y},\phi_H^1)-\epsilon)$-isolated under sufficiently small perturbations. This completes the proof of b).

    To prove a), we argue by contradiction and suppose that $u_i$ is a sequence of nontrivial Floer cylinders of $(H,J)$ having at least one of the asymptotics $u(\pm \infty)$ in $\mathcal{Y}$ and such that
    \begin{equation} \label{eq:zeroenergy}
        \lim_{i \to +\infty} E(u_i)=0.
    \end{equation}
    After passing to a further subsequence if necessary, we may assume that each $u_i$ is asymptotic to a fixed $y \in \mathcal{Y}$ at $+\infty$, i.e.,
	\[u_i(+\infty)=y,\]
    where the case $u_i(-\infty)=y$ for a fixed $y \in \mathcal{Y}$ at $-\infty$ is treated identically.

As in the proof of b), since $y$ is non-degenerate, there exists a neighbourhood $U_{y,t}\subset \Sigma$ of $y(t)$ for each $t \in S^1$, such that no other $1$-periodic orbits intersect the union $U_y := \bigcup\set{t}\times U_{y,t}$.
Because $u_i$ is not a trivial cylinder, reasoning as in the proof of b) we may assume that
	\begin{align*}
		&u_i(0,t_i)\in \partial U_{y, t_i} \quad \text{for some }t_i\in S^1,\quad\text{and}\\
		&u_i(s, t) \in U_{y,t} \hspace{.9cm} \text{for all } (s,t) \in [0, \infty)\times S^1.
	\end{align*}
We then consider the $C_{\rm loc}^\infty$-limit $u := \lim_{i \to \infty} u_i$. As in the proof of b) above this is a Floer cylinder of $(H,J)$ satisfying
\begin{align*}
		&u(0,t)\in \partial U_{y, t} \quad \text{for some }t\in S^1,\quad\text{and}\\
		&u(s, t) \in \overbar{U}_{y,t} \hspace{.7cm} \text{for all } (s,t) \in [0, \infty)\times S^1.
\end{align*}
As $y$ is the only $1$-periodic orbit in $\overbar{U}_{y, t}$, we must have $u(+\infty)=y$. This, combined with the fact that $u(0,t)\in \partial U_{y, t}$ implies that $u$ is a nontrivial Floer cylinder. But, by \eqref{eq:zeroenergy}, $E(u)=0$ which implies that $u$ is a trivial Floer cylinder. This contradiction implies a). 
\end{proof}

The previous proposition shows that a finite set $\mathcal{Y}$ of non-degenerate elements of $\mathcal{P}(H,\alpha)$ is,  for any $J$, Floer $\Delta$-isolated with respect to $J$ for some $\Delta>0$ which depends on $J$. 
In particular, it follows directly from  Definition~\ref{defn:FTAG} that for any finite set $\mathcal{Y}$ of non-degenerate elements of $\mathcal{P}(H,\alpha)$ we have \[\Delta(\mathcal{Y},\phi_H^1)>0.\]

Moreover, by a similar argument, if $\mathcal{Y}$ is Floer $\Delta$-isolated with respect to $J$, then $\mathcal{Y}$ is Floer $\Delta'$-isolated with respect to any $J'$ sufficiently close to $J$ in the $C^\infty$-topology, for any chosen $\Delta'<\Delta$. The following proposition is a direct consequence of Proposition \ref{Proposition: Stabillity of quasi-isolatedness} and this discussion.
\begin{prop}
\label{Proposition: Property of Floer-theoretic action gap}
Let $\mathcal{Y}$ be a collection of non-degenerate $1$-periodic orbits of a possibly degenerate Hamiltonian $H$. Then, the Floer theoretical action gap
\[\Delta(\mathcal{Y}, \phi^1_H)\]
is strictly positive. Moreover, if $H$ is non-degenerate, then for any $\epsilon>0$ there exists a regular $J$ such that $\mathcal{Y}$ is Floer $(\Delta(\mathcal{Y}, \phi_H^1) - \epsilon)$-isolated with respect to $J$.
\end{prop}

\subsection{Structures on Hamiltonian Floer homologies}
In this section, we recall the construction of continuation maps and of the pair of pants product for Hamiltonian Floer homology.

\subsubsection{Continuation maps}
\label{Subsubsection: Continuation maps}
We begin by introducing the continuation maps previously mentioned in Section \ref{sec:Hamonclosedsurfaces}. For two Floer data $(H_-, J_-)$ and $(H_+, J_+)$, there exists a homomorphism
\[\Phi_{(H_-,J_-)\to(H_+,J_+)}: \HF_*(H_-, J_-;\alpha)\to \HF_*(H_+, J_+;\alpha)\]
called a continuation map.
Continuation maps satisfy the following properties:
\begin{itemize}
	\item ({\bf Naturality}) \ For any three Floer data $(H_-, J_-)$, $(H_0, J_0)$, and \linebreak $(H_+, J_+)$, we have
	\[\Phi_{(H_-,J_-)\to(H_+,J_+)} = \Phi_{(H_0,J_0)\to(H_+,J_+)} \circ \Phi_{(H_-,J_-)\to(H_0,J_0)}.\]
	\item ({\bf Identity}) \ For any Floer data $(H,J)$, the continuation map
	\[\Phi_{(H,J)\to(H,J)}: \HF_*(H, J;\alpha)\to \HF_*(H, J;\alpha).\]
	is the identity map. 
\end{itemize}
It follows directly that continuation maps are isomorphisms.


For two Floer data $(H_\pm, J_\pm)$, we construct $\Phi_{(H_-,J_-)\to(H_+,J_+)}$ by the following steps. Take a smooth homotopy of Hamiltonians $\mathcal{H}=\set{H_{s}}_{s\in \RR}$ such that $H_{s} = H_{+}$ for $s$ sufficiently large and $H_{s} = H_{-}$ for $-s$ sufficiently large. Similarly, choose a smooth homotopy of $S^1$-families  $\mathcal{J}=\set{J_{s}}_{s\in \RR}$ of almost complex structures  with $J_{s} = J_{+}$ for $s$ sufficiently large and $J_{s} = J_{-}$ for $-s$ sufficiently large.

Consider the following variation of the Floer equation for $u:\RR\times S^1 \to \Sigma$:	
\begin{equation}
	\label{eq:defFloerContinuationeq}
	\partial_s u(s,t) +{J}_{s,t}(u(s,t)) \left(\partial_{t}u(s,t)-X_{H_s}(t,u(s,t))\right)=0,
\end{equation}
abbreviated as $\mathcal{F}_{\mathcal{H},\mathcal{J}}(u)=0$. 

Similar to \eqref{Equation: definition of energy 1 - Floer equation}, we define the energy of $u$ with respect to $(\mathcal{H},\mathcal{J})$ as
\begin{equation}
	\label{Equation: definition of energy 2 - Floer continuation equation}
	E_{\mathcal{H}, \mathcal{J}}(u)=\frac{1}{2}\int_{\RR\times S^1} \norm{\partial_s u}^2_{J_{s,t}}+\norm{\partial_t u - X_{H_{s,t}}}^2_{J_{s,t}} ds \, dt,
\end{equation}
where $\norm{v}_{J_{s,t}}^2=\omega(v, J_{s,t}v)$. The solutions of equation \eqref{eq:defFloerContinuationeq} with finite energy are called Floer continuation cylinders.

From a direct computation, we have
\begin{align*}
	E_{\mathcal{H}, \mathcal{J}}(u)
	&=-\int_{\RR}\frac{d}{ds} \mathcal{A}_{H_s}(u(s))-		
	\int_{\RR \times S^1}\partial_s{H_{s,t}} \left(u(s,t)\right) ds\,dt.
\end{align*}

If $u$ is a Floer continuation cylinder (i.e. a finite-energy solution of \eqref{eq:defFloerContinuationeq}) then $u(s,\cdot)$ converges to $1$-periodic orbits $y_\pm \in \mathcal{P}(H_\pm)$ for $s\to \pm\infty$ in the $C^\infty$-topology. We denote the limit $1$-periodic orbits $y_\pm$ by $u(\pm \infty)$. Given any capping $w_-$ of $y_-$, the energy satisfies:
\begin{equation}
\label{Equation: Floer continuation cylinder Energy}
\begin{aligned}
E_{\mathcal{H},\mathcal{J}}(u)
&= \mathcal{A}_{H_-}(y_-,w_-)
 - \mathcal{A}_{H_+}(y_+,w_-\# u) \\
&\quad - \int_{\RR\times S^1}
\partial_s H_{s,t}\bigl(u(s,t)\bigr)\, ds\,dt .
\end{aligned}
\end{equation}
Again, a Floer continuation cylinder $u$ gives a free homotopy between the limit $1$-periodic orbits $u(\pm \infty)$.

Fix a free homotopy class $\alpha \in \tilde{\pi}_1(\Sigma)$. For capped orbits $[y_-, w_-]\in \widetilde{\mathcal{P}}(H_-;\alpha)$ and $[y_+, w_+]\in \widetilde{\mathcal{P}}(H_+;\alpha)$, define the moduli space $${\mathcal{M}}([y_-,w_-], [y_+,w_+], \mathcal{H}, \mathcal{J})$$ as the space of Floer continuation cylinders $u$ asymptotic to $y_\pm$ at $\pm \infty$, satisfying $[y_+,w_+]=[y_+,w_-\# u]$. 

For a $C^\infty$-generic choice of $\mathcal{J}$, the moduli space ${\mathcal{M}}([y_-,w_-], [y_+,w_+], \mathcal{H}, \mathcal{J})$ is a smooth manifold of dimension ${\CZ}(y_-,w_-) - {\CZ}(y_+,w_+)$ . Such $\mathcal{J}$ is called regular, and the pair $(\mathcal{H},\mathcal{J})$ is referred to as Floer continuation data.
By Assumption \ref{Assumption: no bubbling}, each $0$-dimensional moduli space is compact and, therefore, a finite set. We then define
\[
K([y_-,w_-],[y_+,w_+]):= \#_2 {\mathcal{M}}([y_-,w_-],[y_+,w_+], \mathcal{H},\mathcal{J}).
\]

The compactness of the moduli space is a consequence of an apriori energy bound for the continuation cylinders. Indeed, a direct computation shows that for $u\in {\mathcal{M}}([y_-,w_-], [y_+,w_+], \mathcal{H}, \mathcal{J})$, we have
\begin{equation}
	\label{Equation: Apriori Energy bounds for continuation cylinders}
	E_\mathcal{H,J}(u) \le \mathcal{A}_{H_-}(y_-,w_-) - \mathcal{A}_{H_+}(y_+,w_+)+E^-(\mathcal{H})
\end{equation}
where $E^-(\mathcal{H})$ is given by
\[
E^-(\mathcal{H}):=\int_{\RR\times S^1} -\min \partial_s{H_{s,t}} \,ds\,dt
\]
Since $H_{s} = H_{+}$ for $s$ sufficiently large and $H_{s} = H_{-}$ for $-s$ sufficiently large, $E^-(\mathcal{H})$ is finite and we obtain the desired apriori energy bound for Floer continuation cylinders.

\medskip

\begin{rem} We define $E^+(\mathcal{H})$ by
\[E^+(\mathcal{H}):=\int_{\RR\times S^1} \max \partial_s{H_{s,t}} \,ds\,dt \]
and the Hofer norm $\norm{\mathcal{H}}_\Hofer$ by
\[\norm{\mathcal{H}}_\Hofer := E^+(\mathcal{H})+E^-(\mathcal{H}).\] 
The Hofer norm $\norm{\mathcal{H}}_\Hofer$ corresponds to the length of a path $\phi^1_{H_s}$ interpolating $\phi^1_{H_-}$ and $\phi^1_{H_+}$. This is the path length used to define the Hofer norm on the group of Hamiltonian diffeomorphisms.
\end{rem}


The chain-level continuation map $\Phi_{\mathcal{H},\mathcal{J}}:\CF_k(H_-;\alpha)\to \CF_{k}(H_+;\alpha)$ is defined by
\[\Phi_{\mathcal{H},\mathcal{J}}([y_-,w_-]):=\sum_{\substack{{[y_+,w_+]\,\in\,{\widetilde{\mathcal{P}}}(H_+;\alpha),} \\ {\CZ([y_+,w_+])=k}}} K\left([y_-,w_-], [y_+,w_+]\right) \cdot [y_+,w_+]\]
for each capped orbit $[y_-,w_-]$. The map $\Phi_{\mathcal{H},\mathcal{J}}$ is extended to all of $\CF_k(H_-)$ linearly. In the case of $\Sigma=T^2$ and $\alpha\neq\cont$, the condition \eqref{Equation: downward Novikov condition} is preserved under the continuation map $\Phi_{\mathcal{H},\mathcal{J}}$.

The structure of the boundary strata of $1$-dimensional moduli spaces of Floer continuation cylinders ensure that $\Phi_{\mathcal{H},\mathcal{J}}$ is a chain map, i.e. $$\Phi_{\mathcal{H},\mathcal{J}}\circ d_{H_-, J_-}=d_{H_+,J_+}\circ \Phi_{\mathcal{H},\mathcal{J}}.$$ Hence, $\Phi_{\mathcal{H},\mathcal{J}}$ defines a map between Hamiltonian Floer homology which we also denote by
\[
\Phi_{\mathcal{H},\mathcal{J}}:\HF_k(H_-,J_-;\alpha)\to\HF_k(H_+,J_+;\alpha).
\]

The construction of $\Phi_{\mathcal{H},\mathcal{J}}$ requires the choice of continuation data $(\mathcal{H}, \mathcal{J})$. However, the map  $\Phi_{\mathcal{H},\mathcal{J}}$ is independent of this choice on the homology level. To see this, we consider, for two continuation data $(\mathcal{H}_0, \mathcal{J}_0)$ and $(\mathcal{H}_1, \mathcal{J}_1)$, a homotopy of continuation data $(\mathfrak{H},\mathfrak{J})=\set{(\mathcal{H}_\tau,\mathcal{J}_\tau)}_{\tau\in[0,1]}$, often referred to as a homotopy of homotopies. We assume that $H_\tau$ and $J_\tau$ are  constant near $0$ and $1$.

For generic $\mathfrak{J}$, the parametrized moduli space ${\mathcal{M}}([y_-,w_-],[y_+,w_+], \mathfrak{H},\mathfrak{J})$ defined by the set
\begin{align*}
\{(\tau,w): \tau \in [0,1],~&u:\RR\times S^1\to \Sigma \text{ such that } \mathcal{F}_{\mathcal{H}_\tau, \mathcal{J}_\tau}(u)=0,\\
&u(\pm\infty)=y_\pm, \text{ and } [y_+, w_-\#u]=[y_+, w_+]\}.
\end{align*}
is a smooth manifold with dimension $\CZ(y_-,w_-)-\CZ(y_+,w_+)+1$. Such a $\mathfrak{J}$ is called regular.  Assumption \ref{Assumption: no bubbling} implies that each $0$-dimensional moduli space of this type is compact and, therefore, a finite set. We then define
\[
N([y_-,w_-],[y_+,w_+]):= \#_2 {\mathcal{M}}([y_-,w_-],[y_+,w_+], \mathfrak{H},\mathfrak{J}).
\]
The chain homotopy $S_{\mathfrak{H},\mathfrak{J}}:\CF_k(H_-;\alpha)\to \CF_{k+1}(H_+;\alpha)$ is defined by
\[S_{\mathfrak{H},\mathfrak{J}}([y_-,w_-]):=\sum_{\substack{{[y_+,w_+]\,\in\,{\widetilde{\mathcal{P}}}(H_+;\alpha),} \\ {\CZ([y_+,w_+])=k+1}}} N\left([y_-,w_-], [y_+,w_+]\right) \cdot [y_+,w_+]\]
for each capped orbit $[y_-,w_-]$. The map $S_{\mathfrak{H},\mathfrak{J}}$ is extended linearly to all of $\CF_k(H_-)$. In the case of $\Sigma=T^2$ and $\alpha\neq\cont$, the condition \eqref{Equation: downward Novikov condition} is preserved under the chain homotopy $S_{\mathfrak{H},\mathfrak{J}}$.

The structure of the boundary strata of the $1$-dimensional moduli spaces of type ${\mathcal{M}}([y,w],[y',w'], \mathfrak{H},\mathfrak{J})$ ensures that $S_{\mathfrak{H},\mathfrak{J}}$ is a chain homotopy between $\Phi_{\mathcal{H}_0,\mathcal{J}_0}$ and $\Phi_{\mathcal{H}_1,\mathcal{J}_1}$ , i.e.:
\begin{equation}
\label{Equation: Chain homotopy}
\Phi_{\mathcal{H}_0, \mathcal{J}_0}-\Phi_{\mathcal{H}_1, \mathcal{J}_1}=S_{\mathfrak{H},\mathfrak{J}}\circ d_{H_-, J_-}+ d_{H_+, J_+}\circ S_{\mathfrak{H},\mathfrak{J}}.
\end{equation}
Hence, $\Phi_{\mathcal{H}_0, \mathcal{J}_0}$ and $\Phi_{\mathcal{H}_1, \mathcal{J}_1}$ define the same map on the homology-level. We thus denote the well-defined homology level map by $\Phi_{(H_-,J_-)\to(H_+,J_+)}$.

To establish the {\bf (Identity)} property of Floer continuation maps, we consider the constant homotopy from $(H,J)$ to itself. In this case, the $0$-dimensional moduli spaces appearing in the definition of the Floer continuation map consist only of trivial cylinders. Then, the associated continuation map is the identity map both at the chain-level and at the homology level.

The proof of the {\bf (Naturality)} property relies on the Gluing principle, which will be recalled in more detail in Section \ref{Section: Product structures on Hamiltonian Floer homology}. For generic continuation data $(\mathcal{H}_-, \mathcal{J}_-)$ from $(H_-,J_-)$ to $(H_0,J_0)$ and $(\mathcal{H}_+, \mathcal{J}_+)$ from $(H_0,J_0)$ to $(H_+,J_+)$, we have the following relation:
\[
\Phi_{\mathcal{H}_+,\mathcal{J}_+}\circ \Phi_{\mathcal{H}_-,\mathcal{J}_-}=
\Phi_{\mathcal{H}_- \#_\ell \mathcal{H}_+, \mathcal{J}_- \#_\ell \mathcal{J}_+}
\]
for sufficiently large gluing parameter $\ell \in \RR$. Here, $(\mathcal{H}_- \#_\ell \mathcal{H}_+, \mathcal{J}_- \#_\ell \mathcal{J}_+)$ denotes the glued continuation data from $(H_-, J_-)$ to $(H_+, J_+)$ with gluing parameter $\ell$. Genericity is needed to ensure that $(\mathcal{H}_- \#_\ell \mathcal{H}_+, \mathcal{J}_- \#_\ell \mathcal{J}_+)$ does indeed give continuation data for every sufficiently large $\ell \in \RR$.

Since the Floer continuation map at the homology level is independent of the choice of continuation data, we obtain the {\bf (Naturality)} property:
\[\Phi_{(H_-,J_-)\to(H_+,J_+)} = \Phi_{(H_0,J_0)\to(H_+,J_+)} \circ \Phi_{(H_-,J_-)\to(H_0,J_0)}.\]

The construction described above gives filtered continuation maps between filtered Floer homologies: \[\Phi^a_{\mathcal{H},\mathcal{J}}:\HF_k^a(H_-,J_-;\alpha)\to\HF_k^{a+E^-(\mathcal{H})}(H_+,J_+;\alpha).\]
The filtration shift is due to the apriori energy bound \eqref{Equation: Apriori Energy bounds for continuation cylinders}. Similarly, the filtered chain homotopy is defined between
\[S^a_{\mathfrak{H},\mathfrak{J}}:\CF_k^a(H_-;\alpha)\to\CF_{k+1}^{a+E^-(\mathfrak{H})}(H_+;\alpha),\]
where $E^-(\mathfrak{H}):=\max_{\tau\in [0,1]} E^-(\mathcal{H}_\tau)<+\infty$. By definition, these filtered version maps commute with the persistence maps $\iota$.


As a filtered version of equation \eqref{Equation: Chain homotopy}, the filtered continuation maps and the filtered chain homotopies satisfy
\begin{equation}
\label{Equation: Filtered Chain homotopy}
	\begin{aligned}
		\iota_{H_+}^{a+E^-(\mathcal{H}_0), a+E^-(\mathfrak{H})} \circ
		\Phi^a_{\mathcal{H}_0,\mathcal{J}_0}-
		&\iota_{H_+}^{a+E^-(\mathcal{H}_1), a+E^-(\mathfrak{H})} \circ
		\Phi^a_{\mathcal{H}_1,\mathcal{J}_1}\\
		=~&S^a_{\mathfrak{H},\mathfrak{J}}\circ d_{H_-, J_-}+ d_{H_+, J_+}\circ S^a_{\mathfrak{H},\mathfrak{J}},
	\end{aligned}
\end{equation}
as a map from $\CF^a_k(H_-;\alpha) $ to $\CF^{a+E^-(\mathfrak{H})}_k(H_+;\alpha)$. Hence, we have
\begin{equation}
\label{Equation: Invariance of Filtered Continuation map upto}
\iota_{H_+}^{a+E^-(\mathcal{H}_0), a+E^-(\mathfrak{H})} \circ
\Phi^a_{\mathcal{H}_0,\mathcal{J}_0}=
\iota_{H_+}^{a+E^-(\mathcal{H}_1), a+E^-(\mathfrak{H})} \circ
\Phi^a_{\mathcal{H}_1,\mathcal{J}_1},
\end{equation}
as a map from $\HF^a_k(H_-,J_-;\alpha)$ to $\HF^{a+E^-(\mathfrak{H})}_k(H_+,J_+;\alpha)$.

From now on, we will omit the persistence maps $\iota$ when these are used to adjust the filtration levels. For example, we will abbreviate Equation \eqref{Equation: Filtered Chain homotopy} as:
\begin{align*}
\Phi^a_{\mathcal{H}_0,\mathcal{J}_0}-\Phi^a_{\mathcal{H}_1,\mathcal{J}_1}
=S^a_{\mathfrak{H},\mathfrak{J}}\circ d_{H_-, J_-}+ d_{H_+, J_+}\circ S^a_{\mathfrak{H},\mathfrak{J}}.
\end{align*}

The properties of continuation maps extend to filtered continuation maps in the following ways:
\begin{itemize}
	\item {\bf (Naturality)} For generic continuation data $(\mathcal{H}_-, \mathcal{J}_-)$ from $(H_-,J_-)$ to $(H_0,J_0)$ and $(\mathcal{H}_+, \mathcal{J}_+)$ from $(H_0,J_0)$ to $(H_+,J_+)$, we have,
	\[
	\Phi^{a+E^-(\mathcal{H}_-)}_{\mathcal{H}_+,\mathcal{J}_+}\circ \Phi^a_{\mathcal{H}_-,\mathcal{J}_-}=
	\Phi^a_{\mathcal{H}_- \#_\ell \mathcal{H}_+, \mathcal{J}_- \#_\ell \mathcal{J}_+}
	\]
	as a map from $\HF_*^a(H_-, J_-;\alpha)$ to $\HF_*^{a+E^-(\mathcal{H}_-)+E^-(\mathcal{H}_+)}(H_+, J_+;\alpha)$, for any $a\in\RR$ and for sufficiently large gluing parameter $\ell \in \RR$.
	\item {\bf (Identity)} For the constant homotopy $(\mathcal{H},\mathcal{J})$ from $(H,J)$ to itself,
	\[\Phi^a_{\mathcal{H},\mathcal{J}}: \HF^a_*(H, J;\alpha)\to \HF^a_*(H, J;\alpha).\]
	is the identity map. 
\end{itemize}

For (Naturality), we use the following identity, which will follow
directly from the Gluing principle introduced later:
\[E^-(\mathcal{H}_-) + E^-(\mathcal{H}_-) = E^-(\mathcal{H}_-\#\mathcal{H}_+).\]

The properties above imply that the filtered Floer homology is independent of the choice of regular $J$. To see this, consider the constant homotopy of Hamiltonian $\mathcal{H}=H$. By the Gluing principle, there exist regular homotopies $\mathcal{J_-}$ and $\mathcal{J_+}$ each connecting $J_-$ to $J_+$ and $J_+$ to $J_-$, respectively, such that:
\[
\Phi^{a}_{H,\mathcal{J}_+}\circ \Phi^a_{H,\mathcal{J}_-}=
\Phi^a_{H, \mathcal{J}_- \#_\ell \mathcal{J}_+},\quad
\Phi^{a}_{H,\mathcal{J}_-}\circ \Phi^a_{H,\mathcal{J}_+}=
\Phi^a_{H, \mathcal{J}_+ \#_\ell \mathcal{J}_-}
\]
each as a map from $\HF^a(H,J_-;\alpha)$ and $\HF^a(H,J_+;\alpha)$ to itself.

Considering the constant homotopy of homotopies $\mathfrak{H}=H$, it follows that $\Phi^a_{(H,\mathcal{J})}$ does not depend on the choice of a regular homotopy $\mathcal{J}$.  In particular, we may use the constant homotopy to calculate $\Phi^a_{(H,\mathcal{J})}$. Together with the {\bf (Identity)} property, we conclude that
\[
\Phi^{a}_{H,\mathcal{J}_+}\circ \Phi^a_{H,\mathcal{J}_-}=\id_{\HF^a(H,J_-;\alpha)},\quad
\Phi^{a}_{H,\mathcal{J}_-}\circ \Phi^a_{H,\mathcal{J}_+}=\id_{\HF^a(H,J_+;\alpha)}.
\]
In other words, the filtered continuation maps $\Phi^{a}_{H,\mathcal{J}_+} $ and $ \Phi^a_{H,\mathcal{J}_-}$ associated to the continuation data $(H,\mathcal{J_-})$ and $(H,\mathcal{J_+})$ are isomorphisms between the filtered Floer homologies $\HF^a(H,J_-;\alpha)$ and $\HF^a(H,J_+;\alpha)$.

For a possibly non-constant homotopy $(\mathcal{H},\mathcal{J})$ from $(H,J)$ to itself, there exists a homotopy of homotopies $(\mathfrak{H}, \mathfrak{J})$ connecting $(\mathcal{H},\mathcal{J})$ to the constant homotopy and satisfying $E^-(\mathfrak{H}) = E^-(\mathcal{H})$. Using the chain homotopy $S^a_{\mathfrak{H},\mathfrak{J}}$, we obtain that
\[\Phi^a_{\mathcal{H},\mathcal{J}}=\iota_*^{a,a+E^-(\mathcal{H})}.\]

\medskip

\subsubsection{Product structures on Hamiltonian Floer homology}
\label{Section: Product structures on Hamiltonian Floer homology}

\

In this section, we review the ring structure on Hamiltonian Floer homology, which is called the pair-of-pants product. We start introducing some operations on Hamiltonians. 

\begin{defn} \label{def:concatenation-and-inverse-Hamiltonians}
Let $H,K : S^1 \times M \to \mathbb{R}$
be time-dependent Hamiltonians with flows $\phi_H^t$ and $\phi_K^t$.
The Hamiltonian $H \# K$ is defined by
\[
(H \# K)_t(x)
=
H_t(x) + K_t\big((\phi_H^t)^{-1}(x)\big).
\]
Its flow satisfies
\[
\phi_{H \# K}^t = \phi_H^t \circ \phi_K^t .
\]
The Hamiltonian $\overbar{H}$ is defined by
\[
\overbar{H}_t(x)
=
- H_t(\phi_H^t(x)).
\]
Its flow satisfies
\[
\phi_{\overbar{H}}^t = (\phi_H^t)^{-1}.
\]
\end{defn}

More precisely, for Floer data $(H_0,J_0), (H_1,J_1)$ and $(H_2, J_2)$, the pair-of-pants product is defined as a map:
\[*_{\mathrm{PP}}: \HF_k(H_1,J_1;\cont)\otimes \HF_l(H_2, J_2;\cont) \to \HF_{k+l-2n}(H_0,J_0;\cont),\]
where, as before, $\cont$ denotes the free homotopy class of contractible loops.

We begin by focusing on the pair-of-pants product between Hamiltonian Floer homology of the contractible class $\cont$; see Remark \ref{Remark: Pair-of-pants for non-contractible orbits} for pair-of-pants products involving classes other than $\cont$. In this paper, we will also consider the following maps:
\begin{align*}
&*_{\mathrm{PP}}: \HF_k(H_1,J_1;\alpha)\otimes \HF_l(H_2, J_2;\cont) \to \HF_{k+l-2n}(H_0,J_0;\alpha), \text{ and}\\
&*_{\mathrm{PP}}: \HF_k(H_1,J_1;\cont) \otimes \HF_l(H_2, J_2;\alpha) \to \HF_{k+l-2n}(H_0,J_0;\alpha).
\end{align*}

For the contractible class $\cont$, (we omit $\cont$ from now on), the pair-of-pants product satisfies the following properties:
\begin{itemize}
	\item {\bf (Commutativity)} For $C_1 \in \HF_k(H_1,J_1)$ and $C_2 \in \HF_l(H_2,J_2)$,
	\[\PP(C_1,C_2)=\PP(C_2, C_1).\]
	as an element of $\HF_{k+l-2n}(H_0,J_0)$.
	\item {\bf (Associativity)} For additional Floer data $(H_{1. 5}, J_{1.5})$ and $(H_{2.5}, J_{2.5})$, compositions of maps
	\begin{align*}
		&*_\mathrm{PP}:\HF_*(H_1, J_1)\otimes \HF_*(H_2,J_2)\to \HF_*(H_{1.5}, J_{1.5}) \text{ with }\\
		&*_\mathrm{PP}:\HF_*(H_{1.5}, J_{1.5})\otimes \HF_*(H_3,J_3)\to \HF_*(H_0, J_0).
	\end{align*}
	and
	\begin{align*}
		&*_\mathrm{PP}:\HF_*(H_2, J_2)\otimes \HF_*(H_3,J_3)\to \HF_*(H_{2.5}, J_{2.5}) \text{ with }\\
		&*_\mathrm{PP}:\HF_*(H_1, J_1)\otimes \HF_*(H_{2.5},J_{2.5})\to \HF_*(H_0, J_0).
	\end{align*}
	are the same as a map
	\[*_\mathrm{PP}\circ *_\mathrm{PP}:\HF_*(H_1,J_1)\otimes \HF_*(H_2,J_2) \otimes \HF_*(H_3,J_3) \to \HF_*(H_0,J_0).\]
\end{itemize} 

In the context of Floer continuation cylinders, the domain $\RR\times S^1$ with the canonical complex structure can be viewed as $\mathbb{C}\setminus \set{0}$, or equivalently as the punctured Riemann sphere $\mathbb{CP}^1\setminus \set{0,\infty}$. By generalizing Floer-type equations to arbitrary punctured Riemann spheres, one can define maps by counting solutions. The pair-of-pants product is defined by considering the case where the domain is the pair-of-pants, i.e.\ the open Riemann surface obtained by removing three points from $\CP^1$.

This approach was first introduced in \cite{Schwarz-thesis} and \cite{PSS96} for any punctured Riemann surfaces, possibly with additional marked points. In this paper, we restrict our attention to maps defined over $\CP^1$ with punctures. Our approach follows \cite{Schwarz00, Oh05}, where Oh-Schwarz spectral invariants were defined.  

Let $z_0, z_1,$ and $z_2$ be distinct points on the Riemann sphere $\mathbb{CP}^1$, and let $P$ denote the pair-of-pants obtained by removing these points. We endow $P$ with the complex structure $j$ inherited from $\mathbb{CP}^1$. The point $z_0$ is considered a positive puncture, meaning that we fix a neighbourhood $U_0$ of $z_0$ and holomorphic identification $\epsilon_0$ of $U_0 \setminus z_0$ with the positive half-cylinder $Z^+:=[0, \infty) \times S^1$, equipped with the standard complex structure. The punctures $z_1$ and $z_2$ are regarded as negative punctures. For $i \in \{ 1,2\}$ we fix a neighbourhood $U_i$ of $z_i$ and a holomorphic identification $\epsilon_i$ of $U_i \setminus z_i$ with $Z^-:=(-\infty,0] \times S^1$, equipped with the standard complex structure.
Holomorphic identifications are referred to as cylindrical ends.


To define a Floer equation on $P$, we consider a Hamiltonian-valued 1-form $\mathcal{H}$ on $P$ and a $P$-family of $\omega$-compatible almost complex structures $\mathcal{J}=\set{J_z}_{z\in P}$. For given Floer data $\set{(H_i,J_i)}_{i=0,1,2}$, we consider $(\mathcal{H},\mathcal{J})$ that equals $(H_i\,dt, J_i)$ on each cylindrical end $\epsilon_i$ outside some compact subset of $P$. In this case, we say that $(\mathcal{H},\mathcal{J})$ is asymptotic to $(H_i, J_i)$ at each puncture $z_i$.

In holomorphic coordinates $(s,t)$ near each point $z\in P$, the Hamiltonian-valued $1$-form  $\mathcal{H}$ can locally be written as $F(s,t) ds + G(s,t) dt$, where $F$ and $G$ are Hamiltonians on the symplectic manifold $(M,\omega)$. 
Moreover, every Hamiltonian valued $1$-form $\mathcal{H}$ defines a Hamiltonian vector field-valued $1$-form $X_\mathcal{H}$, which is locally given by $X_F \, ds + X_G \, dt$.

For a given $(\mathcal{H},\mathcal{J})$ asymptotic to $(H_i,J_i)$, we define the Floer product equation for $u :P\to M$ as:
\begin{equation}
	\label{Equation: Floer product equation}
	(Du(z)-X_\mathcal{H}(z,u))+J_z(u) \circ (Du(z)-X_\mathcal{H}(z,u)) \circ j(z) = 0.
\end{equation}
On the cylindrical ends $\epsilon_i$, the Hamiltonian vector field-valued $1$-form $X_\mathcal{H}$ is given by $X_{H_i}\, dt$, and the equation \eqref{Equation: Floer product equation} reduces to the Floer equation for the Floer data $(H_i, J_i)$ outside a compact set. We abbreviate the equation \eqref{Equation: Floer product equation} by \[\mathcal{F}_{\mathcal{H},\mathcal{J}}(u)=0.\]

Similar to \eqref{Equation: definition of energy 2 - Floer continuation equation}, we define the energy of $u$ with respect to $(\mathcal{H},\mathcal{J})$ as:
\begin{equation}
	\label{Definition: Energy of pair-of-pants}
	E_{\mathcal{H},\mathcal{J}}(u):=\frac{1}{2}\int_P \norm{Du - X_\mathcal{H}}_{J_z}^2 \dvol_P(z),
\end{equation}
where the norm is extended from $\norm{v}_{J_z}^2=\omega(v,J_zv)$ to
vector field-valued $1$-forms on $\Sigma$. For solutions $u$ of $\mathcal{F}_{\mathcal{H},\mathcal{J}}(u)=0$ with finite energy $E_{\mathcal{H},\mathcal{J}}(u)$, there exist $1$-periodic orbits $y_0 \in \mathcal{P}(H_0)$, $y_1 \in \mathcal{P}(H_i)$ and $y_2 \in  \mathcal{P}(H_2)$  such that  $\lim_{s \to + \infty} u(\epsilon_0(s,t))=y_0(t)$ and $\lim_{s \to - \infty} u(\epsilon_i(s,t))=y_i(t)$ for $i \in \{1,2\}$. This follows from the fact that, outside a compact set, the solution defines a finite-energy Floer half-cylinder for each Hamiltonian. We denote the limits $y_0$, $y_1$ and $y_2$ by respectively $u(z_0)$, $u(z_1)$ and $u(z_2)$. At this point, the limits need not be contractible.

The curvature $2$-form of $\mathcal{H}$, a Hamiltonian valued $2$-form on $P$, is defined as: $R_\mathcal{H} d{\rm vol_P} := d\mathcal{H} -\set{\mathcal{H},\mathcal{H}}$, where $\set{\cdot, \cdot}$ denote the Poisson braket. In our convention, the Poisson bracket is given by \[\set{F,G}:=-\omega(X_F,X_G)=dF(X_G).\]

When $u$ converges to $y_i \in \mathcal{P}(H_i;\cont)$ at each puncture $z_i$, the energy $E_{\mathcal{H},\mathcal{J}}(u)$ satisfies
\begin{equation*}
	\begin{aligned}
	E_{\mathcal{H},\mathcal{J}}(u)&= \mathcal{A}_{H_1}(y_1,w_1)
	+\mathcal{A}_{H_2}(y_2,w_2)\\&-\mathcal{A}_{H_0}(y_0,w_0)+ \int_P R_\mathcal{H} (z,u(z))~\dvol_{P}.
	\end{aligned}
\end{equation*}
for any disk cappings $w_i$ satisfying $[y_0, (w_1 \cup w_2)\# u] = [y_0, w_0]$. Note that if $y_1$ and $y_2$ are contractible, such a solution $u$ can exist only if $y_0$ is also contractible.

Motivated by the definition of $E^-$ for Floer continuation cylinders, we define $E^-(\mathcal{H})$ by
\[E^-(\mathcal{H}):= \int_P \left(-\min_{x\in M}
R_\mathcal{H}(z,x)\right)\dvol_P(z)\]
Since $\mathcal{H}$ is asymptotic to $H_i$ at each puncture $z_i$, the curvature $2$-form vanishes outside some compact set. This ensures that $E^-(\mathcal{H})$ is indeed finite, and gives the apriori energy bound:
\begin{equation}
\label{Equation: Apriori equation for pair-of-pants}
	E_{\mathcal{H},\mathcal{J}}(u)\le \mathcal{A}_{H_1}(y_1,w_1)
	+\mathcal{A}_{H_2}(y_2,w_2)-\mathcal{A}_{H_0}(y_0,w_0)+ E^-(\mathcal{H}).
\end{equation}

\begin{rem} We define $E^+(\mathcal{H})$ by
\[E^+(\mathcal{H}):= \int_P \left(\max_{x\in M}R_\mathcal{H}(z,x)\right)\dvol_P(z) \in [0,\infty),\]
and define the Hofer norm $\norm{\mathcal{H}}_\Hofer$ by
\[\norm{\mathcal{H}}_\Hofer := E^+(\mathcal{H})+E^-(\mathcal{H}).\] 
\end{rem}
Let $\mathcal{M}_P([y_1,w_1],[y_2,w_2],[y_0,w_0];\mathcal{H},\mathcal{J})$ denote the set of solutions $u$ of the the Floer equation \eqref{Equation: Floer product equation} asymptotic to $y_i\in\mathcal{P}(H_i;\cont)$ at each puncture $z_i$, satisfying $[y_0, (w_1 \cup w_2)\# u] = [y_0, w_0]$. For a $C^\infty$-generic choice of $\mathcal{J}$, this moduli space is a smooth manifold of dimension $\CZ([y_1, w_1])+\CZ([y_2,w_2])-\CZ([y_0,w_0])-2n.$ Such $\mathcal{J}$ is called regular, and the pair $(\mathcal{H},\mathcal{J})$ is referred to as Floer product data.


Similarly to what happens in the construction of continuation maps, $0$-dimensional moduli spaces of solutions of \eqref{Equation: Floer product equation} define a chain-level map:
\[\PP: \CF_k(H_1,J_1)\otimes \CF_l(H_2, J_2) \to \CF_{k+l-2n}(H_0,J_0)\]
by the mod 2 count of elements in the associated moduli spaces. The structure of the boundary strata of the  $1$-dimensional moduli spaces ensures that $\PP$ induces a map on homology. Furthermore, using the cobordism method, the pair-of-pants product between Hamiltonian Floer homology does not depend on the choice of Floer product data asymptotic to a given Floer data $(H_i, J_i)$.
\begin{figure}
    \centering
    \includegraphics[width=\linewidth]{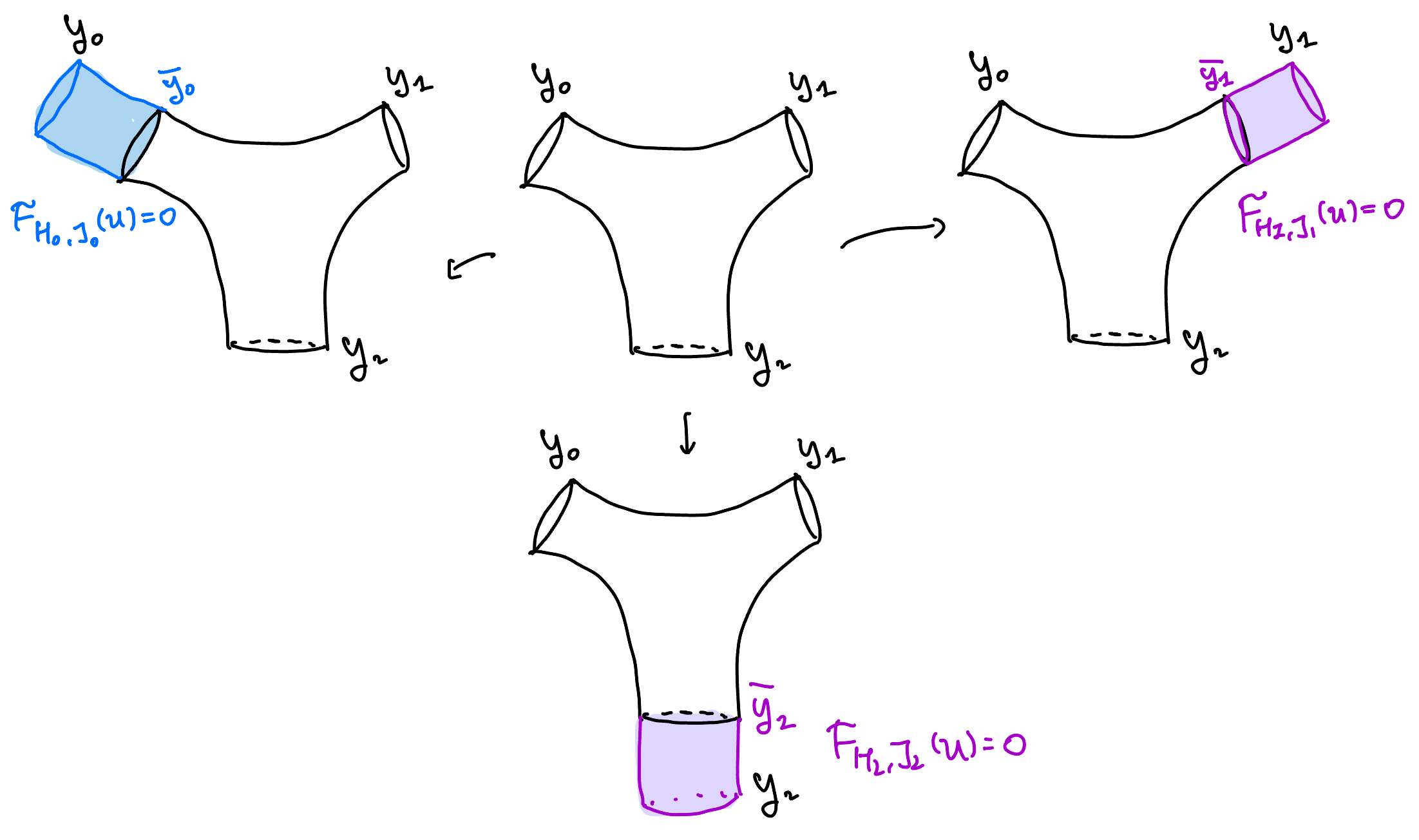}
    \caption{Boundary of $1$-dimensional moduli spaces for $\PP$}
    \label{fig: bdry of pop}
\end{figure}

The construction described above defines filtered pair-of-pants products
\begin{equation}
\label{Equation: Filtered pair-of-pants product}
  \PP:\HF^{a}_k(H_1,J_1)\otimes \HF^{b}_l (H_2, J_2)\to \HF^{a+b+E^-(\mathcal{H})}_{k+l-2n}(H_0,J_0),  
\end{equation}
for a Floer product data $(\mathcal{H}, \mathcal{J})$. Often it is useful to choose $\mathcal{H}$ such that $E^-(\mathcal{H})$ is as small as possible. In the ideal situation, we can take $\mathcal{H}$ with $E^{-}(\mathcal{H})=0$ that defines a pair-of-pants product without the term $E^-(\mathcal{H})$.

By definition, $E^-(\mathcal{H})=0$ if and only if the curvature $2$-form $R_\mathcal{H} d{\rm vol_P}$ vanishes everywhere on $P$. A Hamiltonian-valued $1$-form with vanishing curvature is called flat. From \cite{Schwarz00}, a flat Hamiltonian $1$-form exists under the assumption $H_1\# H_2=H_0$:
\begin{prop}
\label{Proposition: Flat Hamiltonian 1-form}
	For Hamiltonians $H_0$, $H_1$, and $H_2$ (not necessarily non-degenerate) satisfying $H_1\# H_2=H_0$, there exists a flat Hamiltonian-valued $1$-form $\mathcal{H}$ on $P$ asymptotic to $H_i~dt$ at the punctures $z_i$.
\end{prop}

Hence, for $H_0=H_1 \# H_2$, we have a filtered pair-of-pants product,
\[\PP:\HF^{a}_k(H_1,J_1)\otimes \HF^{b}_l(H_2,J_2)\to \HF_{k+\ell-2n}^{a+b}(H_1\#H_2,J_0).\]
Since the space of flat $1$-forms is connected, the pair-of-pants construction is independent of the choice of flat $1$-form. The converse also follows by considering the monodromies around loops in $P$. Nevertheless, we will not need this fact here and omit it from the proof.

In general, we have the following proposition:
\begin{prop}
	\label{Proposition: Hofer-small Hamiltonian 1-form}
	Let $H_0$, $H_1$, and $H_2$ be (not necessarily non-degenerate) Hamiltonians that satisfy
	\[\norm{\overbar{H}_0 \# H_1 \# H_2}_\Hofer<\delta,\]
	for a real number $\delta >0$. Then, there exists a Hamiltonian-valued $1$-form $\mathcal{H}$ on $P$ asymptotic to $H_i dt$ at the punctures $z_i$, satisfying $E^{-}(\mathcal{H}) < \delta$.
\end{prop}
\begin{proof}
	By Proposition \ref{Proposition: Flat Hamiltonian 1-form}, there exists a flat Hamiltonian-valued $1$-form $\mathcal{H}_0$ asymptotic to $H_1 dt$,   $H_2\, dt$,  $H_1 \# H_2\, dt$   at $z_1$, $z_2$ and $z_0$, respectively.
	By assumption, there exists a homotopy of Hamiltonians $\mathcal{H}_1$ connecting $H_1 \# H_2$ and $H_0$, such that $E^{-}(\mathcal{H}_1) < \delta$. For a sufficiently large real number $\ell > 0$, we define the Hamiltonian-valued $1$-form $\mathcal{H}_0 \#_\ell \mathcal{H}_1$ on $P$ asymptotic to $H_i dt$ at the punctures $z_i$ as follows:
	\begin{equation}
	\label{Equation: Almost flat Hamiltonian-valued 1-form}
	\begin{aligned}
		\mathcal{H}_0 \#_\ell \mathcal{H}_1(z)=
		\begin{cases}
			\mathcal{H}_0(z), & \text{if }z \in P\setminus \epsilon_0([2\ell,\infty)\times S^1),\\
			\mathcal{H}_1(s-3\ell,t), & \text{if } z \in \epsilon_0([\ell,\infty)\times S^1),
		\end{cases}
	\end{aligned}
	\end{equation}
    where we identify $(s,t)\in[0,\infty)\times S^1$ with $z\in P$ by $\epsilon_0$. Here, $\mathcal{H}_0$ and $\mathcal{H}_1$ agree and are both given by $H_1\#H_2\,dt$ with respect to the cylindrical coordinates on the overlapping domain $\epsilon_0([\ell,2\ell]\times S^1)$ for $\ell \gg 0$.
	
	By construction, $R_{\mathcal{H}_0\#_\ell \mathcal{H}_1}$ vanishes except for $\epsilon_0([2\ell,\infty)\times S^1)\subset P$ and $R_{\mathcal{H}_0\#_\ell \mathcal{H}_1}(z)=R_{\mathcal{H}_1}(s-3\ell, t)$ on $\epsilon_0([2\ell,\infty)\times S^1)$. Thus, we conclude that
	\[E^{-}(\mathcal{H}_0 \#_\ell \mathcal{H}_1)=E^{-}(\mathcal{H}_1)<\delta.\]
	This completes the proof.
\end{proof}


For the {\bf (Commutativity)} property, we consider a biholomorphism $\phi$ on $\CP^1$ that fixes the positive puncture and exchanges the negative punctures, i.e. $\phi(z_0)=z_0$, $\phi(z_1)=z_2$, and $\phi(z_2)=z_1$. Given Floer product data $(\mathcal{H},\mathcal{J})$, we consider its pullback under $\phi$, denoted by $(\phi^*\mathcal{H}, \phi^*\mathcal{J})$. Comparing the chain-level pair-of-pants associated to $(\mathcal{H},\mathcal{J})$ and $(\phi^*\mathcal{H}, \phi^*\mathcal{J})$, we conclude that commutativity already holds at the chain-level for these choices of Floer product data.

Before proving the {\bf (Associativity)} property of pair-of-pants products, we review the {\bf (Gluing principle)}.

{\bf (The Gluing Principle)} Let $S$ and $S^\prime$ be punctured Riemann spheres with exactly one positive puncture and assume that $S'$ has $m \geq 1$ negative punctures. Then, we have
\begin{equation}
\label{Equation: Gluing principle without perturbation data}
\Phi_{S \#_\ell S^\prime}=\Phi_{S^\prime} \circ \Phi_{S}.
\end{equation}
where $S\#_{\ell}S^\prime$ denotes the punctured Riemann surface obtained by gluing the positive puncture $z_0$ of $S$ to a negative puncture $z^\prime_k$ of $S^\prime$. The parameter $\ell \geq 0$ keeps track of where in the cylindrical ends near  $z_0$ and $z^\prime_k$ we perform the gluing. The larger $\ell$ is, the larger are the two disjoint regions in $S\#_{\ell}S^\prime$ that coincide with regions of, respectively, $S$ and $S'$. This is explained in more detail below; see also Figure~\ref{fig:gluging principle} for a pictorial illustration.

Let $\mathcal{H}$ and $\mathcal{H}^\prime$ be Hamiltonian valued $1$-forms on $S$ and $S^\prime$, respectively, asymptotic to given Floer data at each puncture. Suppose that the same Floer data is assigned to both  $z_0$ and $z^\prime_k$. Then, for a generic choice of a pair $(\mathcal{J}, \mathcal{J}^\prime)$, the following maps are well-defined and satisfy
\begin{equation}
	\label{Equation: Gluing Principle}
	\Phi_{\mathcal{H}\#_\ell \mathcal{H}^\prime, \mathcal{J}\#_\ell \mathcal{J}^\prime} = \Phi_{\mathcal{H}^\prime ,\mathcal{J}^\prime}\circ\Phi_{\mathcal{H},\mathcal{J}}
\end{equation}
at the homology-level.

\begin{figure}
    \centering
    \includegraphics[width=\linewidth]{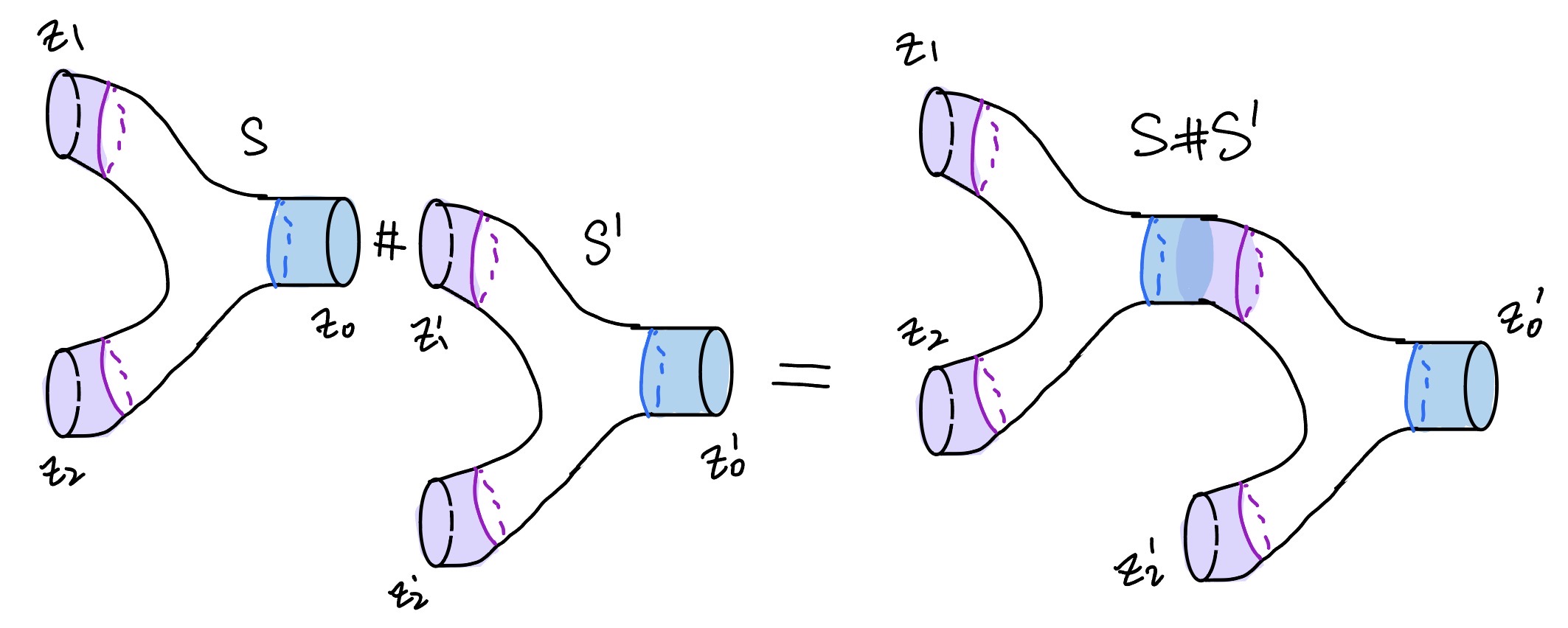}
    \caption{Punctured Riemann surface $S\#S'$ obtained by gluing}
    \label{fig:gluging principle}
\end{figure}

We now define the map $\Phi_{\mathcal{H},\mathcal{J}}$ associated with a punctured $\CP^1$. Let $S$ be a $\CP^1$ with one positive puncture $z_0$ and $d$ negative punctures $z_1, z_2,\ldots z_d$ where $d\ge 0$. We equip $S$ with the complex structure inherited from $\CP^1$. Near each puncture $z_i$, we assign cylindrical ends $\epsilon_i$ such that their images do not overlap. We also assign Floer data $(H_i, J_i)$ at each puncture $z_i$.

A perturbation data on $S$ consists of a pair $(\mathcal{H},\mathcal{J})$, where $\mathcal{H}$ is a Hamiltonian-valued $1$-form on $S$, and $\mathcal{J}$ is an $S$-family of $\omega$-compatible almost complex structures. We say that a perturbation data $(\mathcal{H}, \mathcal{J})$ is asymptotic to $(H_i,J_i)$ if they coincide with $(H_i~dt, J_i)$ on each cylindrical end $\epsilon_i$ outside some compact set.

We then consider the perturbed Cauchy-Riemann equation \eqref{Equation: Floer product equation} for a smooth map $u:S\to \M$, and define the moduli space of solutions
\[\mathcal{M}_S([y_0, w_0], [y_1, w_1],\ldots, [y_d, w_d];\mathcal{H},\mathcal{J})\]
following the same recipe that we used for the pair of pants $P$.
For a $C^\infty$-generic choice of $\mathcal{J}$, these moduli spaces are smooth manifolds. We refer to such a choice of $\mathcal{J}$ as regular.

The $0$-dimensional moduli spaces define a chain-level map $\Phi_{\mathcal{H},\mathcal{J}}$, and the boundary strata of the $1$-dimensional moduli spaces ensure that $\Phi_{\mathcal{H}, \mathcal{J}}$ induces a map on homology. As in the case of the pair-of-pants product, each solution $u$ has an associated energy $E_{\mathcal{H},\mathcal{J}}(u)$, together with an a priori energy bound involving the error term $E^-(\mathcal{H})$. The definition is the same as in the pair-of-pants case; see \eqref{Definition: Energy of pair-of-pants} and \eqref{Equation: Apriori equation for pair-of-pants}.

Using the cobordism method, we conclude that, on the homology level, $\Phi_{\mathcal{H},\mathcal{J}}$ is independent of the choice of perturbation data asymptotic to the given Floer data $\{(H_i,J_i)\}$. Hence, each punctured Riemann sphere $S$ defines a unique map, which we denote by $\Phi_S$. In fact, this map depends only on $S$ as a topological manifold and even not on the choice of cylindrical ends ${\epsilon_i}$. We refer interested readers to \cite{Schwarz-thesis} for further details. Nevertheless, this fact will not be needed in the remainder of the paper. 

The filtered version of the map $\Phi_S$ is also defined as in the case of the pair-of-pants product; see \eqref{Equation: Filtered pair-of-pants product}. Two filtered pair-of-pants products agree up to a shift in the energy level by $E^-(\mathfrak{H})=\max_{\tau \in [0,1]}E^-(\mathcal{H}_\tau)$ for a homotopy $(\mathfrak{H}, \mathfrak{J})$ of pairs $\{ (\mathcal{H}_\tau, \mathcal{J}_\tau)\}_{\tau \in (0,1)}$.

For specific values of $d$, the maps $\Phi_S$ correspond to the previously introduced maps. The case $d=2$ gives the pair-of-pants product, while the case $d=1$ corresponds to Floer continuation maps. The case $d=0$ will be discussed in Section \ref{Section:PSS-image-spectral-norm}.
In these cases, we fix the punctures at $z_0=\infty$, $z_1=0$, and $z_2=1$ on $ \CP^1 \cong \mathbb{C}\cup \{\infty\}$.

To define the map $\Phi_{\mathcal{H}\#_\ell \mathcal{H}^\prime, \mathcal{J}\#_\ell \mathcal{J}^\prime}$, we need to glue two punctured Riemann spheres and perturbation data on them. Given a positive puncture $z_0$ on $(S,j)$, a negative puncture $z_k^\prime$ on $(S^\prime, j^\prime$), and a real parameter $\ell>0$, called the gluing parameter, we construct the glued surface $S \#_\ell S^\prime$ as follows:
\begin{align*}
	\bar{S} &= S \setminus \epsilon_0((2\ell,\infty)\times S^1),\\
	\bar{S}^\prime &= S^\prime \setminus \epsilon^\prime_{k}((-\infty, -2\ell)\times S^1), \text{ and }\\
	S \#_{\ell} S^\prime &=  (\bar{S} \sqcup \bar{S}^\prime)/  \sim,
\end{align*}
where the equivalence relation on the ends is given by $\epsilon_0(s, t)=\epsilon^\prime_{k}(s-3\ell,t)$ for $(s,t)\in [\ell,2\ell]\times S^1$. Then, $S\#_\ell S^\prime$ is also a punctured Riemann sphere. The cylindrical ends at the punctures are inherited from $S$ and $S^\prime$.

If the same Floer data is assigned to both $z_0$ and $z_k^\prime$, we can also glue the corresponding perturbation data, resulting in a new pair $(\mathcal{H}\#_\ell \mathcal{H}^\prime, \mathcal{J}\#_\ell \mathcal{J}^\prime)$, which is well-defined for sufficiently large $\ell>0$. As an example, the Hamiltonian-valued $1$-form constructed $\mathcal{H}_0\#_\ell \mathcal{H}_1$, constructed in the proof of Proposition \ref{Proposition: Hofer-small Hamiltonian 1-form}, is the example of such a glued Hamiltonian-valued $1$-form.

The fact that $\mathcal{J}$ and $\mathcal{J}^\prime$ are regular does not imply that $\mathcal{J}\#_\ell\mathcal{J}^\prime$ is regular. However, for a generic choice of a pair $(\mathcal{J}, \mathcal{J}^\prime)$, we can ensure that $\mathcal{J} \#_\ell \mathcal{J}^\prime$ is regular for all $\ell$ greater than some constant $\ell_0$. Moreover, for a sufficiently large $\ell_1$, we can guarantee that for every $\ell>\ell_1$, there exists a bijection 
\begin{equation}
\label{Equation: Bijection by gluing}
    \begin{aligned}
	\#_\ell: \{(u,u^\prime) \in \mathcal{M}_S(\mathcal{H}, \mathcal{J}) &\times
	\mathcal{M}_{S^\prime} (\mathcal{H}^\prime, \mathcal{J}^\prime): u(z_0)=u^\prime(z_k^\prime)\}\\ &\rightarrow \mathcal{M}_{S\#_{\ell} S^\prime}(\mathcal{H}\#_\ell \mathcal{H}^\prime, \mathcal{J}\#_\ell \mathcal{J}^\prime)\\
	(u,u^\prime)&\mapsto u\#_\ell u^\prime,
    \end{aligned}
\end{equation}
where $u\# u^\prime$ and $u\#_\ell u^\prime$ are $C^0$-close.
In particular, $u \# u^\prime$ is homotopic to $u \#_\ell u^\prime$ relative to the asymptotics for sufficiently large $\ell$.

We thus obtain that for sufficiently large $\ell$ the equality \eqref{Equation: Gluing Principle} holds already at the chain-level; i.e.:
\begin{equation*}
\Phi_{\mathcal{H}\#_\ell \mathcal{H}^\prime, \mathcal{J}\#_\ell \mathcal{J}^\prime} = \Phi_{\mathcal{H}^\prime ,\mathcal{J}^\prime}\circ\Phi_{\mathcal{H},\mathcal{J}}
\end{equation*}
already at the chain-level.
Consequently, the equation \eqref{Equation: Gluing Principle} also holds at the homology. Since the map $\Phi_{\mathcal{H}, \mathcal{J}}$ depends only on the Riemann surfaces at the homology level, we conclude equation \eqref{Equation: Gluing principle without perturbation data}. 


Since we are considering filtered Floer homology, the choice of perturbation data $(\mathcal{H}, \mathcal{J})$ on $S$ plays a crucial role. Therefore, from now on,  whenever we refer to the {\bf (Gluing) principle}, we mean the existence of the bijection $\#_\ell$ together with equation \eqref{Equation: Gluing Principle}.

We now briefly outline the proof of the {\bf (Associativity)} property of the pair-of-pants product. Given arbitrary Floer data $(H_{1.5}, J_{1.5})$ and $(H_{2.5}, J_{2.5})$, we take four Floer product data:
\begin{itemize}
	\item $(\mathcal{H}_{0,0},\mathcal{J}_{0,0})$ for $\CF(H_1, J_1)\otimes \CF(H_2, J_2) \to \CF(H_{1.5},J_{1.5})$
	\item $(\mathcal{H}_{0,1},\mathcal{J}_{0,1})$ for $\CF(H_{1.5}, J_{1.5})\otimes \CF(H_3, J_3) \to \CF(H_0,J_0)$
	\item $(\mathcal{H}_{1,0},\mathcal{J}_{1,0})$ for $\CF(H_2, J_2)\otimes \CF(H_3, J_3) \to \CF(H_{2.5},J_{2.5})$
	\item $(\mathcal{H}_{1,1},\mathcal{J}_{1,1})$ for $\CF(H_1, J_1)\otimes \CF(H_{2.5}, J_{2.5}) \to \CF(H_0,J_0)$
\end{itemize}
The composition $\PP(\PP(\diamond,\circ),\star)$ corresponds to $\Phi_{\mathcal{H}_{1,0},\mathcal{J}_{1,0}} \circ \Phi_{\mathcal{H}_{0,0},\mathcal{J}_{0,0}}$, and $\PP(\diamond, \PP(\circ,\star))$ corresponds to
$\Phi_{\mathcal{H}_{1,1},\mathcal{J}_{1,1}}\circ \Phi_{\mathcal{H}_{0,1},\mathcal{J}_{0,1}}$. To prove that
\begin{equation}
\label{Equation: Associaitivity of pair-of-pants product}
	\Phi_{\mathcal{H}_{1,0},\mathcal{J}_{1,0}}\circ \Phi_{\mathcal{H}_{0,0},\mathcal{J}_{0,0}}
	=\Phi_{\mathcal{H}_{1,1},\mathcal{J}_{1,1}}\circ \Phi_{\mathcal{H}_{0,1},\mathcal{J}_{0,1}}
\end{equation}
at the homology-level, one uses a cobordism argument. 

Each composition can be viewed as a map defined on a stable curve of genus $0$ with four marked points, $z_0, z_1, z_2,$ and $z_3$. The left-hand side corresponds to a stable curve $S_0$ consisting of two Riemann spheres, each with two marked points and one nodal point: $z_1$ and $z_2$ lie on one sphere $P_{0,0}$, while $z_0$ and $z_3$ lie on the other sphere $P_{0,1}$. Similarly, the right hand side corresponds to a stable curve $S_1$ with $z_2, z_3$ on $P_{1,0}$ and $z_0, z_1$ on $P_{1,1}$.

Both for $S_0$ and $S_1$, we consider $z_0$ as a positive puncture and $z_1, z_2$, and $z_3$ as negative punctures. The nodal point connecting $P_{0,0}$ and $P_{0,1}$ is a positive puncture on $P_{0,0}$ and a negative puncture on $P_{0,1}$. Similarly, the nodal point connecting $P_{1,0}$ and $P_{1,1}$ is a positive puncture on $P_{1_0}$ and a negative puncture on $P_{1,1}$. We denote the nodal points by $z_\pm$ according to their sign. We then associate a positive or negative  cylindrical end for each marked points.

Consider a smooth path $\set{S_\tau}_{\tau\in[0,1]}$ in the moduli space of stable curves of genus $0$ with four marked points, interpolating between $S_0$ and $S_1$. We equip each stable curve $S_\tau$ with cylindrical ends $\epsilon_{i, \tau}$ for each puncture $z_i$. Moreover, when $\tau$ is close to $0$ or $1$, we assume that $S_\tau$ is obtained by gluing near the nodal points, with cylindrical ends inherited from $S_0$ and $S_1$.

For each stable curve $S_\tau$, we choose perturbation data $(\mathcal{H}_\tau, \mathcal{J}_\tau)$ asymptotic to the given Floer data at each puncture. For $\tau$ close to $0$, we assume that $(\mathcal{H}_\tau, \mathcal{J}_\tau)$ is obtained by gluing $(\mathcal{H}_{0,0},\mathcal{J}_{0,0})$ on $P_{0,0}$ and $(\mathcal{H}_{0,1},\mathcal{J}_{0,1})$ on $P_{0,1}$. For $\tau$ close to $1$, we assume $(\mathcal{H}_\tau, \mathcal{J}_\tau)$ is obtained by gluing $(\mathcal{H}_{1,0},\mathcal{J}_{1,0})$ on $P_{1,0}$ and $(\mathcal{H}_{1,1},\mathcal{J}_{1,1})$ on $P_{1,1}$. In particular, each $H_\tau$ agree with $H_i dt$ on each cylindrical ends $\epsilon_{i,\tau}$; see Figure~\ref{fig:proof of associativity} below.

We consider the parametrized moduli spaces of solutions of the Floer equation for the $[0,1]$-family of perturbation data $(\mathfrak{H},\mathfrak{J})=\set{(\mathcal{H}_\tau,\mathcal{J}_\tau)}_{\tau\in[0,1]}$. The $0$-dimensional moduli spaces define a degree $(1-4n)$ map
\[\Psi_{\mathfrak{H},\mathfrak{J}}:\HF_*(H_1,J_1)\otimes\HF_*(H_2,J_2)\otimes\HF_*(H_3,J_3)\to\HF_*(H_0,J_0).\]
The boundary of related $1$-dimensional moduli spaces ensures that
\begin{align*}
	&\Phi_{\mathcal{H}_{1,1},\mathcal{J}_{1,1}}\circ \Phi_{\mathcal{H}_{0,1},\mathcal{J}_{0,1}}-\Phi_{\mathcal{H}_{1,1},\mathcal{J}_{1,1}}\circ \Phi_{\mathcal{H}_{0,1},\mathcal{J}_{0,1}}\\
	=~&\Psi\circ d_{H_1,J_1}+ \Psi\circ d_{H_2,J_2} + \Psi\circ d_{H_3,J_3} - d_{H_0, J_0} \circ \Psi.
\end{align*}
The moduli spaces of pair-of-pants maps appear as part of the boundary of the $1$-dimensional moduli spaces for this perturbation data due to our construction and the {\bf (Gluing) principle}. This shows the {\bf (Associativity)} property of the pair-of-pants product \eqref{Equation: Associaitivity of pair-of-pants product} holds at the homology level.

\begin{figure}
    \centering
    \includegraphics[width=\linewidth]{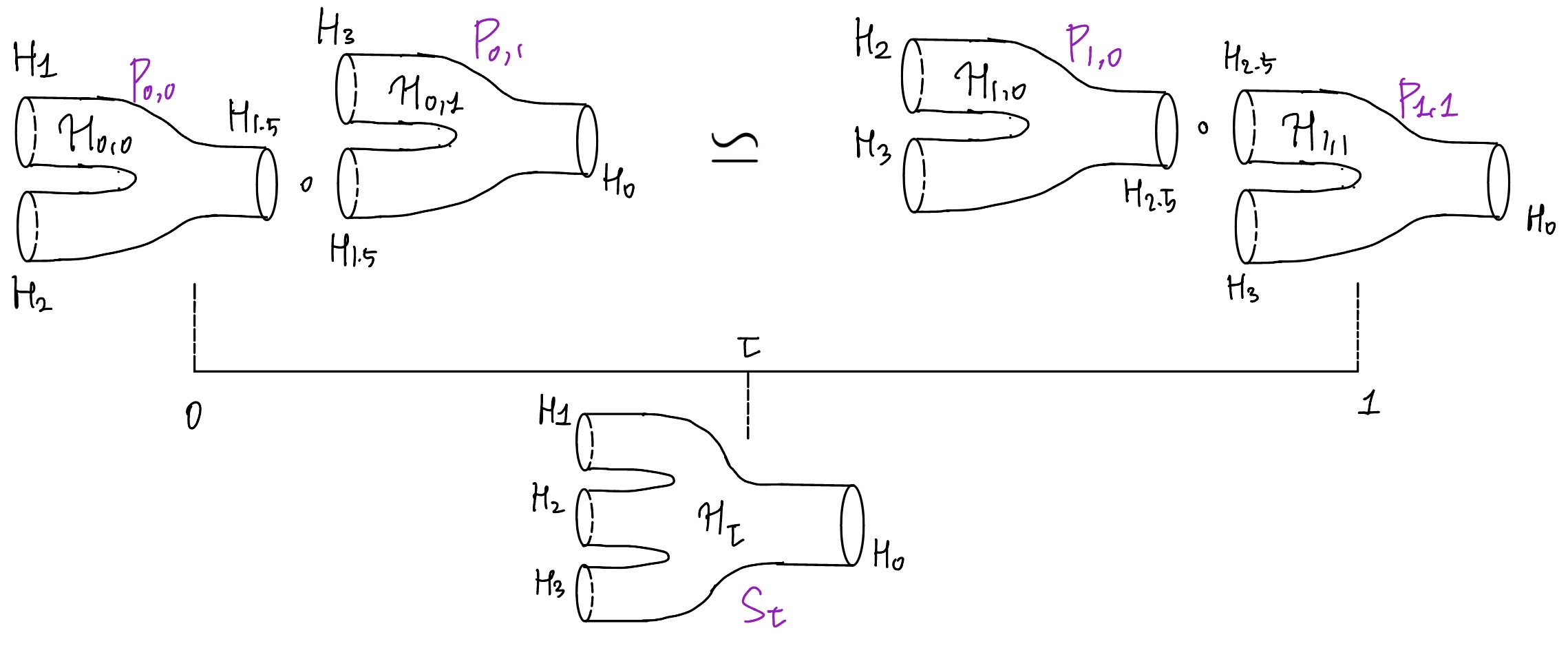}
    \caption{Proof of the (Associativity) property}
    \label{fig:proof of associativity}
\end{figure}

Considering the filtration, the {\bf (Associativity)} property is expressed as
\[\Phi_{\mathcal{H}_{1,0},\mathcal{J}_{1,0}}\circ \Phi_{\mathcal{H}_{0,0},\mathcal{J}_{0,0}}
=\Phi_{\mathcal{H}_{1,1},\mathcal{J}_{1,1}}\circ \Phi_{\mathcal{H}_{0,1},\mathcal{J}_{0,1}}\]
as a map between filtered Floer homology groups:
\[
\HF^a_*(H_1, J_1)\otimes \HF^b_*(H_2, J_2)\otimes \HF^c_*(H_3,J_3)\to \HF_*^{a+b+c+E^{-}(\mathfrak{H})}(H_0,J_0),
\]
where $E^{-}(\mathfrak{H})$ is defined as the maximum of $E^{-}(\mathcal{H}_\tau)$ for $\tau \in (0,1)$. Note that $E^{-}(\mathcal{H}_\tau)$ is constant near $0$ and $1$.

As in Proposition \ref{Proposition: Hofer-small Hamiltonian 1-form}, we obtain the existence of a family of Hamiltonian-valued $1$-forms $\mathfrak{H}=\set{\mathcal{H}_\tau}_{\tau \in [0,1]}$ satisfying $E^-(\mathfrak{H})<\delta$, provided that
\[\norm{\overbar{H}_0 \# H_1 \# H_2 \# H_3}_\Hofer<\delta.\]

\begin{prop}
	\label{Proposition: Hofer-small Family of Hamiltonian 1-form}
	Let $H_0, H_1, H_2$, and $H_3$ be Hamiltonians (not necessarily non-degenerate) such that
	\[\norm{\overbar{H}_0 \# H_1 \# H_2 \# H_3}_\Hofer<\delta,\]
	for some real number $\delta >0$. Then, there exist
    \begin{itemize}
        \item two non-degenerate Hamiltonians $H_{1.5}$ and $H_{2.5}$,
        \item four Hamiltonian-valued $1$-forms on the pair of pants $\mathcal{H}_{0,0}, \mathcal{H}_{0,1}, \mathcal{H}_{1,0}$ and $\mathcal{H}_{1,1}$,
    \end{itemize}
    along with a $[0,1]$-family of Hamiltonian-valued $1$-forms
    $\mathfrak{H}=\set{\mathcal{H}_\tau}_{\tau\in[0,1]}$ each defined on family of stable curves $\{S_\tau\}_{\tau \in [0,1]}$ satisfying
    \[E^-(\mathfrak{H})<\delta.\]
\end{prop}

With the chosen family of Hamiltonian perturbation data $\mathfrak{H}$, we conclude that the following composition relation holds:
\[\Phi_{\mathcal{H}_{1,0},\mathcal{J}_{1,0}}\circ \Phi_{\mathcal{H}_{0,0},\mathcal{J}_{0,0}}
=\Phi_{\mathcal{H}_{1,1},\mathcal{J}_{1,1}}\circ \Phi_{\mathcal{H}_{0,1},\mathcal{J}_{0,1}}\]
as a map defined between filtered Floer homology:
\[\HF^a_*(H_1, J_1)\otimes \HF^b_*(H_2, J_2)\otimes \HF^c_*(H_3,J_3)\to \HF_*^{a+b+c+\delta}(H_0,J_0).\]
\begin{proof}
	We begin by choosing $H_{1.5}$ and $H_{2.5}$. Since non-degeneracy is a generic condition, we can find an arbitrary Hofer-small Hamiltonian $G$ such that the Hamiltonians
	\[H_{1.5} := H_1 \# G \# H_2,\quad H_{2.5}:= \overbar{G} \# \overbar{H}_1 \# H_0\]
	are non-degenerate.
	
	From the proof of \ref{Proposition: Hofer-small Hamiltonian 1-form}, we obtain a Hamiltonian-valued $1$-form $\mathcal{H}_{0,0}$ on $P_{0,0}$ asymptotic to $H_1 dt, H_2 dt$ at the negative punctures $z_1, z_2$ and to $H_{1.5}dt$ at the positive puncture $z_+$, satisfying
	\[E^-(\mathcal{H}_{0,0})<\norm{G}_\Hofer.\]
	Moreover, we assume that the curvature $2$-form is supported in the cylindrical end $\epsilon_1$. Similarly, we construct $\mathcal{H}_{0,1}$ on $P_{0,1}$ asymptotic to $H_{1.5} dt$ and to $H_3 dt$ at, respectively, $z_-$ and $ z_3$ and to $H_{0}dt$ at $z_0$, satisfying
	\[E^-(\mathcal{H}_{0,1})\le \norm{G}_\Hofer + \norm{\overbar{H}_0 \# H_1 \# H_2 \# H_3}_\Hofer,\]
	and the curvature $2$-form is supported in $\epsilon_3$. We also construct $\mathcal{H}_{1,0}$ and $\mathcal{H}_{1,1}$ on $P_{1,0}$ and $P_{1,1}$, respectively.
	
	For $\tau$ close to $0$ or $1$, we define $\mathcal{H}_\tau$ by gluing. Then, we have
	\[E^-(\mathcal{H}_{\tau})<2 \cdot \norm{G}_\Hofer + \norm{\overbar{H}_0 \# H_1 \# H_2 \# H_3}_\Hofer.\]
	Our construction ensures that $\mathcal{H}_\tau$ is obtained by gluing $\tau$-independent $1$-forms	$\mathcal{G}_1$ and $\mathcal{G}_3$ on $\RR\times S^1$ to a flat $\mathcal{H}^\prime_\tau$ on $S_\tau$. More precisely, $\mathcal{H}^\prime_\tau$ is a flat Hamiltonian-valued $1$-form on $S_\tau$ asymptotic to \[H_0 dt, H_1^\prime  dt := H_1 \# G dt, H_2 dt,\text{  and }H_3^\prime dt := (\overbar{H}_2 \# \overbar{G} \# \overbar{H}_1 \# H_0) dt\] at $z_0,z_1,z_2$, and $z_3$, respectively. The forms $\mathcal{G}_1$ and $\mathcal{G}_3$ are Hamiltonian-valued $1$-forms on $\RR\times S^1$ asymptotic to
	\[\mathcal{G}_1: H_1dt\text{ and }H_1^\prime dt,\quad\mathcal{G}_3: H_3 dt\text{ and }H_3^\prime dt\]
	at $-\infty$ and $+\infty$, respectively.
	
	Since the space of flat Hamiltonian-valued $1$-forms on Riemann surfaces is connected, 
    we can extend $\mathcal{H}^\prime_\tau$ to a smooth $(0,1)$-family of flat $1$-forms on $S_\tau$, asymptotic to $H_0 dt, H_1^\prime  dt, H_2 dt$, and $H_3^\prime dt$ at $z_0, z_1, z_2$, and $z_3$, respectively. By gluing $\mathcal{G}_1$ and $\mathcal{G}_3$ at $z_1$ and $z_3$, we define $\mathcal{H}_\tau$ for every $\tau \in (0,1)$.
	
	Then, $E^-(\mathcal{H}_\tau)$ satisfies the following inequality for every $\tau \in (0,1)$:
	\begin{align*}
		E^-(\mathcal{H}_\tau) = E^-(\mathcal{G}_1)+E^-(\mathcal{G}_3) \le 2 \cdot \norm{G}_\Hofer + \norm{\overbar{H}_0 \# H_1 \# H_2 \# H_3}_\Hofer.
	\end{align*}
	For a sufficiently Hofer small $G$, we conclude $E^-(\mathcal{H}_\tau)<\delta$, completing the proof of the proposition.
\end{proof}

\medskip


\begin{rem}[Pair of pants product for a non contractible class]	
\label{Remark: Pair-of-pants for non-contractible orbits}

\

As mentioned above, we also consider in this paper the following versions of the pair of pants product for a non contractible free homotopy class $\alpha$: 
\begin{align} \label{Equation: Product non contractible left}
&*_{\mathrm{PP}}: \HF_k(H_1,J_1;\alpha)\otimes \HF_l(H_2, J_2;\cont) \to \HF_{k+l-2n}(H_0,J_0;\alpha),
\end{align}
and 
\begin{align} \label{Equation: Product non contractible right}
&*_{\mathrm{PP}}: \HF_k(H_1,J_1;\cont) \otimes \HF_l(H_2, J_2;\alpha) \to \HF_{k+l-2n}(H_0,J_0;\alpha).
\end{align}
Note that we use the same reference data $(\eta_\alpha,\tau_\alpha)$ for the Hamiltonian Floer homology of non-contractible $1$-periodic orbits.

To define \eqref{Equation: Product non contractible left} we proceed similarly as in the contractible case. We consider Floer data $(\mathcal{H},\mathcal{J})$ in the pair of pants $P$ asymptotic to $(H_i, J_i)$ at each puncture $z_i$ of $P$.  Given capped $1$-periodic orbits $[y_1,w_1] \in \widetilde{\mathcal{P}}(H_1,\alpha)$, $[y_2,w_2] \in \widetilde{\mathcal{P}}(H_1,\cont)$ and $[y_0,w_0] \in \widetilde{\mathcal{P}}(H_0,\alpha)$ we consider the moduli space
$$\mathcal{M}_P([y_1,w_1],[y_2,w_2],[y_0,w_0];\mathcal{H},\mathcal{J})$$
of solutions $u$ of \eqref{Equation: Floer product equation} such that $[y_0, (w_1 \cup w_2)\# u] = [y_0, w_0]$. Note that if either $y_1$ or $y_2$ is contractible, then $y_0$ has the same homotopy type as the other one, provided that such a solution $u$ exists.

Again, for a $C^\infty$-generic choice of $\mathcal{J}$, this moduli space is a smooth manifold of dimension $\CZ([y_1, w_1])+\CZ([y_2,w_2])-\CZ([y_0,w_0])-2n.$ Such $\mathcal{J}$ is called regular, and, following our previous terminology, we refer to the pair $(\mathcal{H},\mathcal{J})$ as Floer product data.
The $\Z_2$-count of elements in the $0$-dimensional moduli spaces of this type defines the product on the chain-level. Reasoning as in the contractible case, one shows that this defines a product on the homology-level, thus obtaining \eqref{Equation: Product non contractible left}. 

A similar {\bf (Associativity)} property holds for the products on such cases, and by composing pair of pants products we obtain maps:
\[\HF_*(H_1, J_1;\alpha)\otimes \HF_*(H_2, J_2;\cont)\otimes \HF_*(H_3,J_3;\cont)\to \HF_*(H_0,J_0;\cont).\]
Under the hypotheses of Proposition \ref{Proposition: Hofer-small Family of Hamiltonian 1-form}, we thus obtain by composing pair of pants products a map
\begin{align*}
\HF^a_*(H_1,J_1;\alpha)\otimes \HF^b_*(H_2,J_2;\cont)
\otimes \HF^c_*(&H_3,J_3;\cont) \\
&\longrightarrow
\HF_*^{a+b+c+\delta}(H_0,J_0;\alpha).
\end{align*}

The construction of \eqref{Equation: Product non contractible right} and related discussions are identical, except that now the capped $1$-periodic orbits are elements in 
\[[y_1,w_1] \in \widetilde{\mathcal{P}}(H_1,\cont), [y_2,w_2] \in \widetilde{\mathcal{P}}(H_1,\alpha)\text{ and }[y_0,w_0] \in \widetilde{\mathcal{P}}(H_0,\alpha).\]


\end{rem}

\medskip

\subsection{The PSS-image spectral norm} \label{Section:PSS-image-spectral-norm}

\ 

In this section, we recall the definition of PSS-image spectral norm  recently introduced by Connery-Grigg in \cite{Connery-Grigg24}.

For a given Floer data $(H,J)$, a PSS-data is defined as a pair $\mathcal{D}=(\mathcal{H},\mathcal{J})$ where:
\begin{itemize}
	\item $(\mathcal{H},\mathcal{J})$ is a Floer continuation data, asymptotic to $(H,J)$ at $+\infty$ and asymptotic to $(0,J_0)$ at $-\infty$, for some $\omega$-compatible almost complex structure $J_0$.
\end{itemize}

Let $u: \R \times S^1 \to M$ satisfy $\mathcal{F}_{\mathcal{H},\mathcal{J}}(u)=0$ and $E_{\mathcal{H},\mathcal{J}}(u)<+\infty$. 
As observed previously, we know that $u$ converges (as $s\to +\infty$) to a $1$-periodic orbit $u(+\infty) \in \mathcal{P}(H)$ at its positive end.
By the removal of singularity for pseudo-holomorphic curves, the condition ${E}_{\mathcal{H},\mathcal{J}}(u)<+\infty$ ensures that the limit $u(-\infty)$ exists in $M$, and the extension of $u$ to the compactification $\C=\set{-\infty} \cup \RR \times S^1$ is $J_0$ holomorphic at $-\infty$. These imply that the $1$-periodic orbit $u(+\infty) \in \mathcal{P}(H)$ must be contractible, and this extension provides a capping disk for $u(+\infty)$.


 Given a $1$-periodic orbit $y \in \mathcal{P}(H)$, we consider the moduli space
\begin{align*}
	\mathcal{M}(y;\mathcal{D})=\{&u: \RR \times S^1 \to M :~
		 \mathcal{F}_{\mathcal{H},\mathcal{J}}(u)=0,\\
	& E_{\mathcal{H},\mathcal{J}}(u)<+\infty, ~u(+\infty)=y \}.
\end{align*}

We partition $\mathcal{M}(y;\mathcal{D})$ into $\mathcal{M}([y,w];\mathcal{D})$ according to the equivalence class of the capped orbit $[y,w]$. For generic choices of PSS-data, the moduli space $\mathcal{M}([y,w];\mathcal{D})$ is a smooth manifold of dimension
\[\dim \mathcal{M}([y,w];\mathcal{D}) = 2n - \CZ([y,w]).\]
Such a choice of $\mathcal{D}$ is called regular. We denote the set of regular PSS-data by $PSS(H,J)$.

Thus, the dimension of the moduli space $\mathcal{M}([y,w];\mathcal{D})$ is $0$ exactly when \linebreak $\CZ([y,w]) = 2n$.
Under Assumption \ref{Assumption: no bubbling}, the $0$-dimensional moduli spaces are compact, and their mod $2$ counts are well-defined.

When $\CZ([y,w]) = 2n-1$, the moduli space $\mathcal{M}([y,w];\mathcal{D})$ is 1-dimensional and admits a compactification by adding boundary strata:
\begin{align*}
	\bigcup_{\substack{{[y^\prime,w^\prime] \,\in\, \widetilde{\mathcal{P}}(H;\cont)} \\{ \CZ([y^\prime,w^\prime]) = 2n-1}}} \mathcal{M}([y^\prime,w^\prime];\mathcal{D})&\times \mathcal{M}([y^\prime,w^\prime], [y,w] ; H,J).
\end{align*}

\medskip
\begin{defn} \label{Definition:PSS-fundamental-class}
    The $0$-dimensional moduli spaces define the chain 
\begin{equation} \label{eq:PSS-fundamental-class}
{\Phi}_{\mathcal{D}}([M]) := \sum_{\CZ([y,w])=2n}\#_2\mathcal{M}([y,w];\mathcal{H},\mathcal{J})\cdot [y,w] \in \CF_{2n}(H).    
\end{equation}
The compactification of the $1$-dimensional moduli spaces mentioned above ensures that ${\Phi}_{\mathcal{D}}([M])$ is a cycle, i.e. $ d_{H,J}( {\Phi}_{\mathcal{D}}([M]))=0$. We call ${\Phi}_{\mathcal{D}}([M])$ the PSS-fundamental class; see Remark \ref{Remark: PSS-class-through-PSS-iso}.
\medskip
\end{defn}



\begin{rem}
We have the following equivalent description of the moduli spaces $\mathcal{M}([y,w];\mathcal{H},\mathcal{J})$:
\begin{align*}
		\mathcal{M}([y,w];\mathcal{H},\mathcal{J})=\{\bar{v}: \C \to M:~
		\mathcal{F}_{\overbar{\mathcal{H}},\overbar{\mathcal{J}}}(&\bar{v})=0,~E_{\overbar{\mathcal{H}},\overbar{\mathcal{J}}}(\bar{v})<+\infty,\\ &\bar{v}(+\infty)=y,~[y,\bar{v}]=[y,w]\}.
\end{align*}
where $(\overbar{\mathcal{H}},\overbar{\mathcal{J}})$ is the Floer perturbation data on $\C$, obtained by extending $(\mathcal{H},\mathcal{J})$ via the biholomorphic identification of $\RR\times S^1$ with $\C\setminus \set{0}$. 
\medskip
\end{rem}

\begin{rem} \label{Remark: PSS-class-through-PSS-iso}
As the notation ${\Phi}_{\mathcal{D}}([M])$ suggests, the cycle ${\Phi}_{\mathcal{D}}([M])$ can be obtained as the image of the fundamental class $[M]$ in the Morse homology complex $HM_*(M,\Z_2)$ by a map from $HM_*(M,\Z_2)$ to $HF_*(H,J)$. This is the approach followed in \cite{Connery-Grigg24}, and the map in question is the PSS-isomorphism. In the present paper we prefer to use Definition \ref{Definition:PSS-fundamental-class}, which expresses ${\Phi}_{\mathcal{D}}([M])$ directly from the moduli spaces, since it is more suitable for our arguments.
\end{rem}

\subsubsection{PSS-image spectral invariants}

\

Throughout this subsection, we assume that the symplectic manifold $(M,\omega)$ satisfies Assumption~\ref{Assumption: no bubbling}. Again, this includes all closed symplectic surfaces $(\Sigma,\omega)$. For $\Sigma=T^2$, we do not need to discuss Condition~\eqref{Equation: downward Novikov condition}, since every capped orbit on the right-hand side of Definition~\ref{Definition:PSS-fundamental-class} is contractible.

\begin{defn}
\label{Defintion: PSS-image spectral invariants}
	Let $(H,J)$ be a Floer data. The PSS-image spectral invariant of the fundamental class $[M]$ is defined as
	\[c_{\im}([M]; H,J):=\inf_{\mathcal{D}\in \text{PSS}(H,J)}\set{\mathcal{A}_H( {\Phi}_{\mathcal{D}}([M]))},  \]
where $ {\Phi}_{\mathcal{D}}([M])$ is as defined in \eqref{eq:PSS-fundamental-class}, and $\mathcal{A}_H( {\Phi}_{\mathcal{D}}([M]))$ is as defined in Definition \ref{Definition: Action of a chain}.
\end{defn}

A priori, the PSS-image spectral invariant $c_{\im}([M]; H,J)$ is defined only for Floer data. However, by the {\bf (Hofer continuity)} property stated below, they can be extended to any pair $(H,J)$ as the limit of spectral invariants along an approximating sequence of Floer data. This limit is well-defined and independent of the choice of the approximating sequence. The {\bf (Hofer continuity)} property also implies that $c_{\im}([M]; H,J)$ does not depend on the choice of $J$, allowing us to omit $J$ and write $c_{\im}([M]; H)$.
The next theorem presents the main basic properties of  $c_{\im}([M]; H,J)$.

\begin{thm}[Connery-Grigg \cite{Connery-Grigg24}]

\

The PSS-image spectral invariants $c_{\im}([M]; H,J)$
satisfies the following properties:
	\begin{enumerate}[leftmargin=5.5mm]
		\item ({\bf Finiteness}) $c_\im([M]; H) \in \R$.
		\item ({\bf Spectrality}) $c_\im([M]; H)\in \Spec(H;\cont)$.
		\item ({\bf Normalization}) For any  function $b:M\times[0,1]\to\R$ only depending on the real coordinate, 
		\[c_\im([M]; H+b)=c_\im([M]; H)+\int_0^1  b(t) dt.\] 
		\item ({\bf Hofer-Continuity})  For any Hamiltonians $H,K$
		\[ \, |c_\im([M]; H)-c_\im([M]; K) | \le \norm{H-K}_\Hofer.\]
		\item ({\bf Weak triangle inequality}) For any Hamiltonians $H,K$,
		\begin{equation}
		\label{Equation: Triangle Inequality for PSS image spectral invariant}
			\begin{aligned}
				c_\im([M]; H \# K ) \le c_\im([M]; H) + c_\im([M]; K).
			\end{aligned}
		\end{equation}
		\item ({\bf Homotopy invariance}) For two normalized Hamiltonians $H,K$ such that the paths $\set{\phi^t_H}_{t\in [0,1]},\set{\phi^t_K}_{t\in [0,1]} $ are homotopic relative to their initial and final points,
		we have $c_\im ([M]; H)=c_\im ([M],K)$.
		\item ({\bf Symplectic invariance}) For any symplectomorphism $\psi$ and any Hamiltonian $H$,
		$c_\im ([M]; H)=c_\im ([M]; \psi^*H).$
	\end{enumerate}
\end{thm}
For the (\textbf{Homotopy invariance}) property, we say a Hamiltonian $H$ is normalized if $\int_M H(t,\cdot)\,\omega^{\wedge n}=0$ for each $t\in S^1$. 

\begin{rem}
Spectral invariants were introduced in symplectic topology by Viterbo in his foundational paper \cite{Viterbo92}: a Floer theoretical interpretation of Viterbo's construction was given by Oh \cite{Oh05} and Schwarz \cite{Schwarz00}: we let $c([M];H)$ denote the Oh-Schwarz spectral invariant. As observed by Connery-Grigg \cite{Connery-Grigg24}, the spectral invariant $c_\im ([M]; H)$ satisfies the following inequality:
\begin{equation}
	\label{Equation: Two spectral invariants}
	c([M];H)\le c_\im ([M]; H).
\end{equation}
\end{rem}

By the {\bf (Homotopy invariance)} property, the PSS-image spectral invariant descends to the group $\widetilde{\Ham}(M,\omega)$ defining
\[c_\im([M], [\Phi_H]) := c_\im ([M]; H) -\int_0^1\int_M H_t\omega^n \, dt.\]

Applying the {\bf (Weak triangle inequality)} property, we obtain
\begin{align*}
	0 = c_\im([M], \,\id) &\le c_\im([M],\, [\Phi_H])+c_\im([M],\, [\Phi_{\overbar{H}}]),
\end{align*}
for any Hamiltonian $H$.
These two observations lead to the definition of a pseudo-norm on the groups $\Ham(M,\omega)$ and $\widetilde{\Ham}(M,\omega)$.

\begin{defn}[Connery-Grigg \cite{Connery-Grigg24}]
	\label{Definition: PSS image Spectral norm}

    \
    
	For $\Phi\in \widetilde{\Ham}(M,\omega)$, its PSS-image spectral pseudo-norm is defined as
	\[\gamma_\im(\Phi):=c_\im([M], \Phi) + c_\im([M], \Phi^{-1}).\]
	For $\phi \in \Ham(M,\omega)$, its {PSS-image spectral norm} is given by
    \begin{equation}\label{eq: definition of gamma norm}
    \gamma_\im(\phi)=\inf_{\pi(\Phi)=\phi} \gamma_\im(\Phi),  
    \end{equation}
	where $\pi(\Phi)$ denotes the endpoint of $\Phi\in \widetilde{\Ham}(M,\omega)$.
\end{defn}

\medskip

The $\gamma_\im$-norm on $\Ham(M,\omega)$ satisfies the following properties.
\begin{thm}
	\label{Theorem: Properties of PSS image Spectral norm}
	The PSS-image spectral norm \[\gamma_\im:\Ham(M,\omega) \to \RR_{\ge 0}\] satisfies the following properties: 
	\begin{enumerate}
		\item ({\bf Non-degeneracy}) For any $\phi \in \Ham(M,\omega)$, we have
		\[ \gamma_\im(\phi)=0~\text{ if and only if }~\phi=\id.\]
		\item ({\bf Triangle inequality}) For any $\phi,\phi^\prime \in \Ham(M,\omega)$,
		\[\gamma_\im(\phi \,\circ\, \phi^\prime) \le \gamma_\im (\phi) + \gamma_\im(\phi^\prime).\]
		\item ({\bf Symplectic invariance}) For any $\psi \in \Symp(M,\omega)$,
		\[\gamma_\im(\psi^{-1}\circ\phi\circ\psi)=\gamma_\im(\phi).\]
		\item ({\bf Inverse}) $\gamma_\im(\phi)=\gamma_\im(\phi^{-1})$ for any $\phi \in \Ham(M, \omega)$.
		\item ({\bf Hofer-continuity}) $\gamma_\im(\phi) \le \norm{\phi}_{\Hofer}\,$ for any $\phi \in \Ham(M, \omega)$.
	\end{enumerate}
\end{thm}

\begin{defn} [Connery-Grigg \cite{Connery-Grigg24}]

\

We define a distance $d_\im$ on the group $\Ham(M,\omega)$ by
\[d_\im(\phi, \phi^\prime):=\gamma_\im(\phi^{-1} \circ \phi^\prime).\]

\end{defn}

The ({\bf Non-degeneracy}), ({\bf Inverse}), and ({\bf Triangle inequality}) properties ensure that $d_\im$ defines a metric on $\Ham(M,\omega)$, and the ({\bf Symplectic invariance}) property guarantees that $d_\im$ is a bi-invariant metric on $\Ham(M,\omega)$.

The ({\bf Hofer-Continuity}) property, together with the inequality \eqref{Equation: Two spectral invariants}, implies
\begin{equation}
	\label{Equation: Comparision of three norms}
	\gamma(\phi)\le \gamma_\im (\phi) \le \norm{\phi}_{\Hofer},
\end{equation}
for $\phi \in \Ham(M,\omega)$. Here, $\gamma(\phi)$ denotes the Oh-Schwarz spectral norm. 

Below, we will use two properties of $\gamma_\im$ that holds for symplectic surfaces. These properties follow from \cite[Theorem~4]{Buhovsky-Humiliere-Seyfaddini21} and \cite{Humiliere-LeRoux-Seyfaddini16}, respectively:
\begin{enumerate}\setcounter{enumi}{5}	
	\item ({\bf $C^0$-continuity}) For any symplectic surface $(\Sigma,\omega)$, the norm $\gamma_\im$ is continuous with respect to $C^0$-topology on the group $\Ham(\Sigma, \omega)$.
	\item ({\bf Max property}) For a symplectic surface $(\Sigma,\omega)$ of positive genus, the following holds. For Hamiltonians $H_i  : [0,1]\times \Sigma \to \RR$ supported in pairwise disjoint disks $D_i \subset \Sigma$ ($i=1,2,\ldots, k$), we have
	\begin{equation*}
		c_\im(H_1+H_2+\ldots +H_k ; [\Sigma]) = \max_{i \in \set{1,2,\ldots, k}} c_\im(H_i;[\Sigma]).
	\end{equation*}
\end{enumerate}
In fact, $\gamma_{\mathrm{im}}$ is Lipschitz continuous with respect to the $C^0$-distance; see \cite{Serraille24}. Nevertheless, continuity is sufficient for the remainder of the paper.

\begin{rem}
For a non-degenerate Hamiltonian $H$, the PSS-image spectral invariant $c_\im([\Sigma]; H)$ admits a purely dynamical interpretation on symplectic surfaces $(\Sigma,\omega)$ in terms of braids; see \cite[Theorem 5.13]{Connery-Grigg24}.    
\end{rem}

Given a Hamiltonian $H$ we abuse notation and let $\gamma_{\im}(H)$ denote $\gamma_{\im}([\Phi_H])$. Since the fundamental group of $\Ham(\Sigma,\omega)$ is given by
\[
    \pi_1\left(\Ham(\Sigma,\omega), \id_\Sigma \right)=
	\begin{cases}
	\ZZ_2, & \Sigma=S^2\\
		\set{e}, & \Sigma \ne S^2,
	\end{cases}
\]
the norm $\gamma_{\im}(\phi_H^1)$ is independent of the choice of $H$ when $\Sigma\neq S^2$. When $\Sigma=S^2$, we have
\[
\gamma_{\im}(\phi^1_H) = \min \left(\gamma_{\im}([\Phi_H]), \gamma_{\im}([\Phi_{H\#G}]) \right)
\]
where $G$ generates the rotation of $S^2\subset\mathbb{R}^3$ about the
$z$-axis. In particular, the infimum in \eqref{eq: definition of gamma norm} can be changed to the minimum in the case of symplectic surfaces.

We conclude with two lemmas that will be used later.

Given an open set $U\subset M$, define
\[
    E_{\mathrm{conn}}(U)
    :=\inf_{\phi\in\Ham(M,\omega)}\{\norm{\phi}_\Hofer\},
\]
where the infimum is taken over all Hamiltonian diffeomorphisms $\phi$ that displace each connected component of $U$, i.e.,
\[
    \phi(V)\cap V=\varnothing\quad \text{for every connected component $V$ of $U$.}
\]

\begin{lem}
	\label{Lemma: Spectral invariant of [M] supported on the displacable open set}
	Let $U$ be an open set of a symplectic manifold $(M,\omega)$. Suppose that two Hamiltonians $H,K:S^1 \times M \to \RR$ satisfy
    \begin{itemize}
        \item $\phi_K^1(V)\cap V=\varnothing\quad \text{for every connected component $V$ of $U$.}$
        \item $\mathrm{supp}(H) \subset  S^1 \times U$.
    \end{itemize}
	Then, we have
    \[c_\im([M]; H) \le \gamma_\im(K).\]
\end{lem}	

Note that here we assume that each connected component of $U$ is displaced by $K$, rather than that $U$ itself is displaced. Still, the proof is verbatim the same as in the case of the Oh--Schwarz spectral distance. Nevertheless, for completeness, we include the proof below. As a direct corollary of Lemma~\ref{Lemma: Spectral invariant of [M] supported on the displacable open set} and the (\textbf{Hofer-Continuity}) property, we obtain the following.
\begin{cor}
	\label{Corollary: Spectral norm supported on the displacable open set}
	Let $U$ be an open set in $M$. For a Hamiltonian $H$ supported on $U$, we have the inequality
	\begin{equation}
		\label{Equation: Energy-capacity inequality for PSS-image spectral invariants}
		\gamma_\im(H) \le 2\cdot E_{\mathrm{conn}}(U).
	\end{equation}
\end{cor}

\begin{proof}[Proof of Lemma \ref{Lemma: Spectral invariant of [M] supported on the displacable open set}] 

Since $\gamma_{\mathrm{im}}(H)$ depends only on the relative homotopy class in $\widetilde{\Ham}(M,\omega)$, we may assume that $H$ and $K$ vanish near $1\in S^1$. We then define the concatenation $K\natural H$ by
\[
(K\natural H)*t=
\begin{cases}
2 K_{2t} & t \in [0,1/2], \\
2 H_{2t-1} & t \in [1/2,1].\end{cases}
\]
The Hamiltonian flow of $K\natural H$ corresponds to running the Hamiltonian flow of $K$ followed by that of $H$.
In the proof, instead of working with $K\#H$, we work with $K\natural H$, which also represents the same homotopy class $[K\#H]$. In particular, we have $\phi^1_{K\natural H}=\phi^1_K \circ \phi^1_H$.
	
We begin by claiming that the sets of $1$-periodic orbits of $K$ and $K \natural H$ are identical. By the assumption $\phi^1_K(V) \cap V = \varnothing$ for any connected component $V$ of $U$, no point in $U$ can be a $1$-periodic orbit of $K$. It is clear that $\phi^1_{K}\circ \phi^1_H(V) \cap V = \varnothing$ for for any connected component $V$ of $U$, so that the same holds for  $K \natural H$.
For $x \notin U$, we have $\phi^1_{K}\circ \phi^1_H(x)=\phi^1_K(x)$ since the Hamiltonian vector fields $X_{H_t}$ vanishes outside of $U$.  Thus, a point $x$ is a fixed point of $\phi^1_{K}\circ \phi^1_H$ if and only if it is a fixed point of $\phi^1_K$. 

The discussion above further implies that each $1$-periodic orbit of $K\natural H$ is a smooth reparametrization of some $1$-periodic orbit of $K$. Hence, there exists a canonical bijection between $\widetilde{\mathcal{P}}(H;\alpha)$ and $\widetilde{\mathcal{P}}(K \natural H;\alpha)$ for any free homotopy class $\alpha\in \tilde{\pi}_1(M)$.
		
	Next, we compare the Hamiltonian actions of $K\natural H$ and $K$. Given a capped orbit $(y,w)$ for $K\natural H$, we denote the corresponding reparametrized capped orbit for $K$ by $(\tilde{y}, \tilde{w})$. Then, we compute the action functional:
	\begin{align*}
		\mathcal{A}_{K\natural H}(y, w)
		&=\int w^*\omega - \int_0^1 K\natural H(t, y(t))\, dt\\
		&=\int \tilde{w}^*\omega - \int_0^1 K(t,\tilde{y}(t))- \int_0^1 H(t,(y(0))\,dt\\
		&=\mathcal{A}_K(\tilde{y},\tilde{w}).
	\end{align*}
	For the last equality, we used the fact that $H(0,y(0))$ vanishes, since $y(0) \notin U$. This implies the action spectra are the same, i.e. $\Spec(K\natural H)=\Spec(K)$.
	
	Now, consider the continuous path of Hamiltonians  $\epsilon\to\epsilon H$ for $\epsilon \in [0,1]$. This induces a continuous function
	\[\epsilon \mapsto c_\im([M]; K \natural \epsilon H)\]
	which takes values in the nowhere dense set $\Spec(K) \subset \RR$. Therefore, this function must be constant, implying that
	\[c_\im([M]; K\natural H)=c_\im([M]; K).\]
	Finally, applying the ({\bf Weak triangle inequality}) property, we obtain
	\begin{align*}
		c_\im([M]; H)&=c_\im([M]; \overbar{K} \# (K \# H))\\
		&\le c_\im ([M]; \overbar{K}) + c_\im ([M]; K \# H)\\
		&= c_\im([M]; \overbar{K}) + c_\im([M]; K)=\gamma_\im(K).
	\end{align*}
\end{proof}

The second lemma establishes a bound on the PSS-image spectral distance between iterates of Hamiltonian diffeomorphisms:
\begin{lem}
	\label{Lemma: Spectral distance of the iteration}
	Let $\phi_1,\phi_2$ be Hamiltonian diffeomorphisms of $(M,\omega)$. Then, for every $n\in \N$, we have
	\begin{align*}
		d_{\im}(\phi_1^n,\phi_2^n) &\le n \cdot d_{\im}(\phi_1,\phi_2).
	\end{align*}
\end{lem}
\begin{proof}
	By definition, the PSS-image spectral distance satisfies
	\[d_\im(\phi_1^n, \phi_2^n)=\gamma_\im(\phi_2^{-n} \circ \phi_1^n).\]
	We write $\phi_2^{-n}~\circ~\phi_1^n$ as the product 
	\begin{align*}
	\phi_2^{-n}\,\circ\, \phi_1^n = (\phi_2^{-n}\,\circ\,\phi_1 \circ \phi_2^{n-1})\circ(\phi_2^{-n}\,\circ\,\phi_1 \circ \phi_2^{n-1})\,\circ\cdots\,\circ\,(\phi_2^{-1}\circ \phi_1).
	\end{align*}
	By the ({\bf Triangle inequality}) property of $\gamma_\im$, we have
	\[\gamma_\im(\phi_2^{-n}\,\circ\,\phi_1^n)\le\sum_{m=1}^n \gamma_\im(\phi_2^{-m}\,\circ\,\phi_1\,\circ\,\phi_2^{m-1}).\]
	Furthermore, by the ({\bf Symplectic Invariance}) property, we have
	\[\gamma_\im(\phi_2^{-1}\,\circ\,\phi_1)=\gamma_\im(\phi_2^{-m}\,\circ\,\phi_1\,\circ\,\phi_2^{m-1})\quad\text{for all $m$.}\]
	Combining these results, we conclude that
	\[d_\im(\phi_1^n,\phi_2^n)=\gamma_\im(\phi_2^{-n}\circ \phi_1^n) \le n\cdot \gamma_\im(\phi_1^{-1}\circ\phi_2)=n\cdot d_\im(\phi_1,\phi_2).\]
	This completes the proof.
\end{proof}

\vspace{.5cm}

\section[Proof of the main theorems]{Proof of the main theorems}
\label{Proof of the main theorems}
In this section, we study the stability of braids under $\gamma_\im$-small perturbations. Our main objective is to prove Theorems \ref{Main Theorem: Braid Stability under PSS-image spectral distance}, \ref{Main Theorem: Lower semicontinuity of topological entropy}, and then \ref{Main theorem: Topological entropy after perturbation supported on a disk}. Throughout this section, we assume that $(\Sigma,\omega)$ is a closed symplectic surface. For simplicity, we omit explicit references to our choice of regular (family of) almost complex structures.



\subsection{Enhanced Floer continuation maps}
\label{Section: Enhanced Floer continuation maps}
In Section \ref{Section: Product structures on Hamiltonian Floer homology}, we observed that continuation maps can be defined not only using a homotopy of Hamiltonians but also using a Hamiltonian-valued $1$-form on a cylinder $Z$. Building on this observation, our goal in this subsection is to introduce and prove the following proposition:
\begin{prop}
\label{Proposition: Existence of enhanced continuation data and bijection through enhanced Floer cylinder}
	Let $\alpha \in \tilde{\pi}_1(\Sigma)$, $H: S^1 \times \Sigma \to \R$ and $K:S^1 \times \Sigma \to \R$ be non-degenerate Hamiltonians, and $(H,J_H)$ and $(K,J_K)$ be Floer data associated to $H$ and $K$, respectively. Let $\mathcal{Y}$ be a finite subset of $\mathcal{P}(H;\alpha)$ satisfying the inequality
\[\gamma_\im(\overbar{H}\#K)<\Delta(\mathcal{Y},\phi^1_H).\]
Then, there exist
\begin{enumerate}
	\item  enhanced continuation data $\mathcal{HK}$ from  $(H,J_H)$ to $(K,J_K)$ and $ \mathcal{KH}$  from $(K,J_K)$ to $(H,J_H)$, 
	\item a subset $\mathcal{Z} \subset \mathcal{P}(K;\alpha)$, 
	\item a injection $\mathfrak{R}_{\mathcal{HK}}: \mathcal{Y} \to \mathcal{Z}$ and a bijection $\mathfrak{S}_\mathcal{KH}: \mathcal{Y} \to \mathcal{Y}$ such that
	\[\mathcal{M}(y; \mathfrak{R}_{\mathcal{HK}}(y) ; \mathcal{HK})_0 ~ \text{ and } ~ \mathcal{M}(\mathfrak{R}_{\mathcal{HK}}(y), \mathfrak{S}_{\mathcal{KH}}(y) ; \mathcal{KH})_0\]
	are non-empty for every $y$ in $\mathcal{Y}$.
\end{enumerate}
\end{prop}

Here, the $1$-form $\mathcal{HK}$ is enhanced in the following sense: there exists a real number $0 \le E \le E^-(\overbar{H} \# K)$ such that the associated chain-level map
\[\Phi_{\mathcal{HK}}:\CF_*(H;\alpha)\to\CF_*(K;\alpha)\]
increases the Hamiltonian action at most $E$. The existence of bijections via enhanced continuation cylinders will be used in Section \ref{Section: Isotopy via enhanced Floer cylinders} to construct a braid isotopy between $\mathcal{B}(\mathcal{Y})$ and $\mathcal{B}(\mathcal{Z})$. 


The proof of Proposition \ref{Proposition: Existence of enhanced continuation data and bijection through enhanced Floer cylinder} will be given at the end of this subsection. The remainder of this subsection is dedicated to introducing the terminology used in the proposition and establishing key properties of enhanced continuation data, which will be essential for proving the existence of the $\mathcal{Z} \subset \mathcal{P}(K;\alpha)$ and the corresponding bijections.

\subsubsection{Construction of enhanced Floer continuation data}
\label{Section: Construction of enhanced Floer continuation data}
In this subsection, we construct enhanced continuation data $\mathcal{HK}$ and $\mathcal{KH}$ on a cylinder $Z$ by gluing a the PSS-data on a plane $\C$ and the Floer product data on a pair-of-pants $P$. We then prove that the associated maps $\Phi_\mathcal{HK}$ and $\Phi_\mathcal{KH}$ behave like Floer continuation maps while shifting the action by an amount controlled by the PSS-image spectral norm.

\begin{rem}
Throughout this section, we will use the same letter to denote an enhanced continuation data on $Z$ and its underlying Hamiltonian $1$-form. For example, $\mathcal{HK}$ and $\mathcal{KH}$ will denote both enhanced continuation data on $Z$ and the Hamiltonian $1$-forms associated to them. The same will apply to the PSS-data and the Floer product data considered in this section. 

It will always be clear from the context whether we are considering the Floer data or the Hamiltonian-valued $1$-form.
\end{rem}

As explained in Section \ref{Section:PSS-image-spectral-norm}, the $0$-dimensional moduli spaces associated with a PSS-data $\mathcal{G}$ on $\C$ determine the PSS-fundamental class $\Phi_\mathcal{G}([\Sigma])$. 

We begin by proving the following proposition; see Figure~\ref{fig: Proposition: PSS + Product = Continuation Filtered} for a pictorial illustration of the proof.
\begin{prop}
\label{Proposition: PSS + Product = Continuation Filtered}
	Let $H, G$, and $K$ be non-degenerate Hamiltonians on $\Sigma$. Then, we have the identity
	\begin{equation}
	\label{Equation: PSS + Product = Continuation}
		{\Phi}_{\mathcal{G}\#_\ell \mathcal{H}}(-) = {\Phi}_{\mathcal{H}}(-, {\Phi}_{\mathcal{G}}([\Sigma]))
	\end{equation}
	as chain maps $\CF_*(H;\alpha)\to \CF_*(K;\alpha)$, where $\mathcal{H}$ denotes a Floer product data from $H,G$ to $K$, $\mathcal{G}$ denotes a PSS data for $G$, and $\ell>0$ is a sufficiently large gluing parameter. In particular, ${\Phi}_{\mathcal{G}\#_\ell \mathcal{H}}$ defines filtered maps
	\[{\Phi}^a_{\mathcal{G}\#_\ell \mathcal{H}}:\HF^a_*(H;\alpha)\to\HF^{a+\mathcal{A}_G({\Phi}_{\mathcal{G}}([\Sigma]))+E^-(\mathcal{H})}(K;\alpha).\]
\end{prop}

\begin{proof}
	As discussed above, the cycle ${\Phi}_{\mathcal{G}}([\Sigma])\in\CF_2(G;\cont)$ is defined by the (mod $2$) count of $0$-dimensional moduli spaces of solutions to \linebreak $\mathcal{F}_{\mathcal{G}}(u)=0$. By the {\bf (Gluing)} principle), there exists a bijection 
	\begin{align*}
		\#_\ell: \{(u,u^\prime) \in \mathcal{M}_\C(\mathcal{G}) \times
		\mathcal{M}_{P} (\mathcal{H}): u(z_0)=u^\prime(z_2^\prime)\} &\rightarrow \mathcal{M}_{Z}(\mathcal{G}\#_\ell \mathcal{H})\\
		(u,u^\prime)&\mapsto u\#_\ell u^\prime. 
	\end{align*}
    Here $z_0$ is the unique (negative) puncture of $\mathcal{G}$ in $\C \simeq S^2 \setminus z_0$, and $z'_2$ is the puncture of $P$ on which $\mathcal{H}$ is asymptotic to $G$.
	In particular, for the $0$-dimensional components, we obtain
	\begin{equation*}
		\begin{aligned}
			&{\mathcal{M}_{Z}}([y_-,w_-],[y_+,w_+]; \mathcal{G} \#_\ell \mathcal{H})=\\
			&\bigcup_{\CZ([y_0,w_0])=2}
			\mathcal{M}_{\C}([y_0, w_0];\mathcal{G}) \times \mathcal{M}_{P}([y_-,w_-],[y_0,w_0],[y_+,w_+]; \mathcal{H})
		\end{aligned}
	\end{equation*}
	for capped orbits $[y_-, w_-]\in \widetilde{P}(H;\alpha)$ and $[y_+, w_+]\in \widetilde{P}(K;\alpha)$ with the same Conley-Zehnder index, i.e. $\CZ([y_-, w_-])=\CZ([y_+, w_+])$.
	
	Hence, for such capped orbits $[y_\pm, w_\pm]$, the (mod $2$) count
	\[\#_2\left({\mathcal{M}_{Z}}([y_-,w_-],[y_+,w_+]; \mathcal{G} \#_\ell \mathcal{H})\right)\]
	is equal to
	\[\sum_{\CZ([y_0,w_0])=2}
	\#_2\left(\mathcal{M}_{\C}([y_0, w_0];\mathcal{G})\right)\, \cdot \, \#_2\left(\mathcal{M}_{P}([y_-,w_-],[y_0,w_0],[y_+,w_+]; \mathcal{H})\right).\]
	This expression simplifies to
	\[\sum_{[y_0, w_0] \,\in\, \supp {\Phi}_{\mathcal{G}}([\Sigma])} \#_2\left(\mathcal{M}_{P}([y_-,w_-],[y_0,w_0],[y_+,w_+]; \mathcal{H})\right).\]
	In particular, we conclude that
	\[\#_2\left({\mathcal{M}_{Z}}([y_-,w_-],[y_+,w_+]; \mathcal{G} \#_\ell \mathcal{H})\right)=0\]
	whenever $\mathcal{A}_G([y_+,w_+]) > \mathcal{A}_H([y_-,w_-]) + \mathcal{A}_G({\Phi}_{\mathcal{G}}([\Sigma]))+E^-(\mathcal{H})$.
    Note that this does not necessarily imply that the moduli space is empty.
	
	This proves that the identity \eqref{Equation: PSS + Product = Continuation} holds as a map between unfiltered Floer chain complexes. Moreover, since $ {\Phi}_{\mathcal{H}}(-, {\Phi}_{\mathcal{G}}([\Sigma]))$ increases the Hamiltonian action by at most $\mathcal{A}_G({\Phi}_{\mathcal{G}}([\Sigma]))+E^-(\mathcal{H})$, the same is true for  $\Phi_{\mathcal{G}\#_\ell \mathcal{H}}$. Thus, $\Phi_{\mathcal{G}\#_\ell \mathcal{H}}$ defines a filtered map \[{\Phi}^a_{\mathcal{G}\#_\ell \mathcal{H}}:\CF^a_*(H;\alpha)\to\CF_*^{a+\mathcal{A}_G({\Phi}_{\mathcal{G}}([\Sigma]))+E^-(\mathcal{H})}(K;\alpha).\] This completes the proof.
\end{proof}

\begin{figure}
    \centering
    \includegraphics[width=\linewidth]{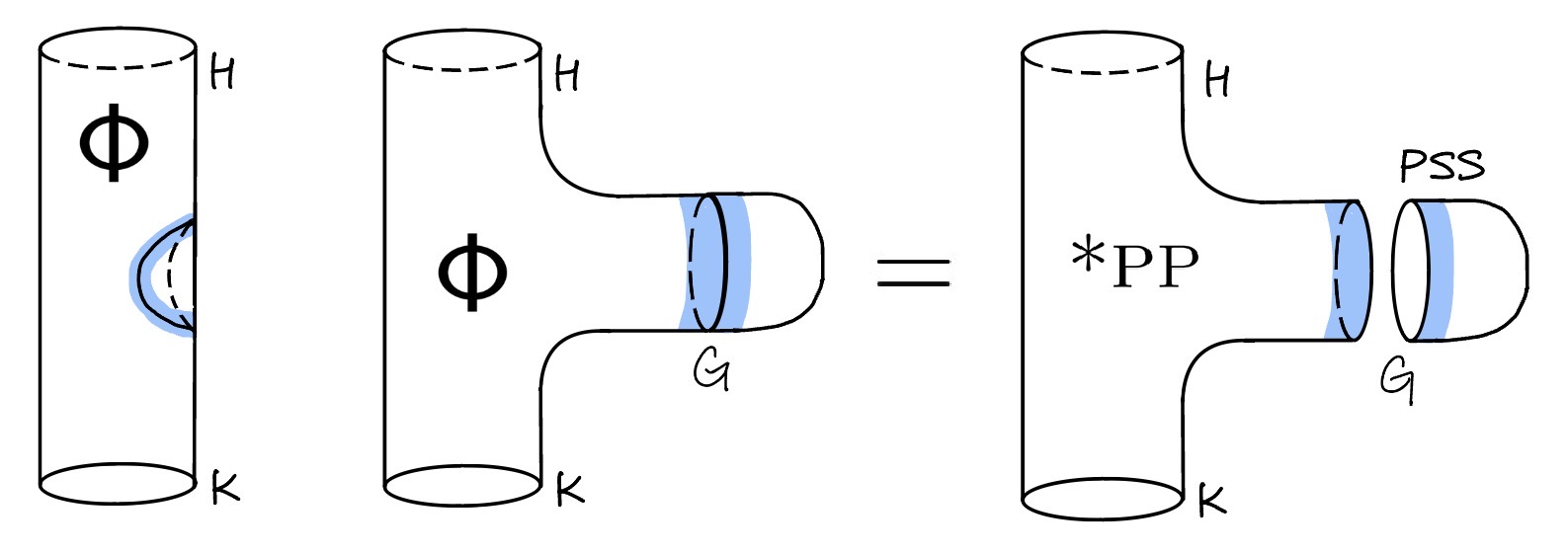}
    \caption{Construction of enhanced continuation data}
    \label{fig: Proposition: PSS + Product = Continuation Filtered}
\end{figure}
The next step is to choose Hamiltonian $1$-forms $\mathcal{G}$ and $\mathcal{H}$ so that the shift \[\mathcal{A}_G(\Phi_\mathcal{G}[\Sigma]) + E^-(\mathcal{H})\] is as small as possible. For this we proceed as follows.

\begin{choice}
For the rest of this section, we fix $\varepsilon>0$ small so that
\begin{equation} \label{eq:epsilon-for-action-gap}
\gamma_\im(\overbar{H}\#K) + 8\varepsilon < \Delta(\mathcal{Y},\phi^1_H).
\end{equation}
\end{choice}

To optimize $\mathcal{H}$, we choose $G$ as a non-degenerate Hamiltonian satisfying
\[\norm{\overbar{K}\# H \#G}_{\Hofer}<\varepsilon.\]
The Hamiltonian $G$ is a Hofer-small perturbation of $\overbar{H} \# K$: this is necessary because  $\overbar{H}\#K$ may not be non-degenerate. By Proposition \ref{Proposition: Hofer-small Hamiltonian 1-form}, there exists a Floer product data $\mathcal{H}$ from $H,G$ to $K$ satisfying $E^-(\mathcal{H})<\varepsilon$.

Moreover, by the definition of PSS-image fundamental class, there exists a PSS-data $\mathcal{G}$ satisfying
\begin{align*}
	\mathcal{A}_G({\Phi}_{\mathcal{G}}([\Sigma])) &\le c_\im(G;[\Sigma])+\varepsilon \le c_\im(\overbar{H}\#K;[\Sigma])+2\varepsilon.
\end{align*}
The second inequality follows from the {\bf (Hofer-Continuity)} property.

By Proposition \ref{Proposition: PSS + Product = Continuation Filtered}, for sufficiently large gluing parameter $\ell>0$, the glued enhanced continuation data $\mathcal{HK}:=\mathcal{G} \#_\ell \mathcal{H}$ defines filtered maps
\[{\Phi}^a_{\mathcal{HK}}:\CF^a_*(H;\alpha)\to\CF_*^{a+c_\im(\overbar{H}\#K;[\Sigma])+3\varepsilon}(K;\alpha).\]
Similarly, we construct the  glued enhanced continuation data $\mathcal{KH}:=\mathcal{\overbar{G}}\#_\ell \overbar{\mathcal{H}}$ that defines filtered maps
\[{\Phi}^a_{\mathcal{KH}}:\CF^a_*(K;\alpha)\to\CF_*^{a+c_\im(\overbar{K}\#H;[\Sigma])+3\varepsilon}(H;\alpha).\]

We introduce some terminology. Given our fixed $\varepsilon>0$ satisfying \eqref{eq:epsilon-for-action-gap}, we say that an enhanced continuation data $\mathcal{HK}$ is an {$\varepsilon$-admissible enhanced continuation data} from $H$ to $K$ if it is  obtained by gluing the following data:
\begin{itemize}
	\item A PSS-data $\mathcal{G}$ to $G$, satisfying $\mathcal{A}_G(\Phi_\mathcal{G}[\Sigma])\le c_\im(\overbar{H}\#K;[\Sigma])+2\varepsilon$,
	\item A Floer product data $\mathcal{H}$ from $H,G$ to $K$, satisfying $E^-(\mathcal{H})<\varepsilon$.
\end{itemize}
As explained above, $G$ is chosen as a non-degenerate perturbation of $\overbar{H}\#K$ which satisfies $\norm{\overbar{K}\# H \#G}_{\Hofer}<\varepsilon$. Similarly, we say that $\mathcal{KH}$ as an {$\varepsilon$-admissible enhanced continuation data} from $K$ to $H$ if it is  obtained by gluing the following data: 
\begin{itemize}
	\item A PSS-data $\overbar{\mathcal{G}}$ to $\overline{G}$, satisfying $\mathcal{A}_{\overbar{G}}(\Phi_{\overbar{\mathcal{G}}}[\Sigma])\le c_\im(\overbar{H}\#K;[\Sigma])+2\varepsilon$,
	\item A Floer product data $\mathcal{K}$ from $K,\overline{G}$ to $H$, satisfying $E^-(\mathcal{K})<\varepsilon$.
\end{itemize}

When considering two $\varepsilon$-admissible enhanced continuation data $\mathcal{HK}, \mathcal{KH}$ between $H$ and $K$ as above, we always assume that the Hofer-small perturbation of $\overbar{K}\#H$ is given by the inverse of the perturbation of $\overbar{H}\#K$.

Now, consider two $\varepsilon$-admissible enhanced continuation data $\mathcal{HK}$ and $ \mathcal{KH}$ from $H$ to $K$ and $K$ to $H$ respectively. Reasoning as in the proof of \linebreak $\Phi_{K \to H} \circ \Phi_{H \to K} = \id_{\HF_*(H;\alpha)}$, one shows that the map \[{\Phi}_{\mathcal{KH}}\circ{\Phi}_{\mathcal{HK}}:\CF_*(H;\alpha) \to \CF_*(H;\alpha)\]
is chain homotopic to the identity. Apriori, since
\[E^-(\mathcal{HK}) = E^-(\mathcal{G}) + E^-(\mathcal{H}) \text{ and } E^-(\mathcal{KH}) = E^-(\overbar{\mathcal{G}}) + E^-(\overbar{\mathcal{H}}), \]
the chain-homotopy can increase the Hamiltonian action by at most \[E^-(\mathcal{G})+E^-(\mathcal{\overbar{G}})+2\varepsilon\]
for two PSS data $\mathcal{G},\mathcal{\overbar{G}}$ which are used to define $\mathcal{HK},\mathcal{KH}$. Hence, with this approach, we only get $\norm{\overbar{H}\# K}_\Hofer+3\varepsilon$ as the upper bound of action shift.

However, we claim the upper bound can be further enhanced to $\gamma_\im(\overbar{H}\#K)+6\varepsilon$. Recall that we have the inequality
\[\gamma_\im(\overbar{H}\#K) \le \norm{\overbar{H}\# K}_\Hofer\]
from the (Hofer-continuity) property. Equivalently, we claim that 
\begin{equation*}
{\Phi}^{a+c_\im(\overbar{K}\#H;[\Sigma])+3\varepsilon}_{\mathcal{KH}}\circ{\Phi}^a_{\mathcal{HK}}:\CF^a_*(H;\alpha) \to \CF_*^{a+\gamma_\im(\overbar{H}\#K)+6\epsilon}(H;\alpha)
\end{equation*}
is chain-homotopic to the chain-level inclusion map
\[\iota^{a,a+\gamma_\im(\overbar{H}\#K)+6\epsilon}:\CF^a_*(H;\alpha) \to \CF_*^{a+\gamma_\im(\overbar{H}\#K)+6\epsilon}(H;\alpha).\]
The proof follows the same strategy as \cite[Proposition 37]{Kislev-Shelukhin21}. For completeness, we provide an adapted version of the proposition and its proof:
\begin{prop}
\label{Proposition: Quantitative version of composition = identity}
	Let $H,K$ be non-degenerate Hamiltonians, and let $G$ be a   perturbation of $\overbar{H} \# K$ such that $\norm{\overbar{K}\# H \#G}_{\Hofer}<\varepsilon$.
    Let $\mathcal{G}$ be a PSS-data for $G$ and $\overbar{\mathcal{G}}$ a PSS-data from $\overbar{G}$.
    Let $C$ and $\overbar{C}$ be the cycles 
	\[[C]=[\Phi_{\mathcal{G}}([\Sigma])] \in \CF_2(G),\qquad [\overbar{C}]=[\Phi_{\overbar{\mathcal{G}}}([\Sigma])]\in \CF_2(\overbar{G}).\]
	Moreover, suppose that we are given Floer product data $\mathcal{H}_{0,0}$ from $H,G$ to $K$ and Floer product data $\mathcal{H}_{1,0}$ from $K,\overbar{G}$ to $H$ that satisfy
	\[E^-(\mathcal{H}_{0,0})<\varepsilon, \qquad E^-(\mathcal{H}_{1,0})<\varepsilon.\]
	Then, the composition
	\[\Phi^{a+\mathcal{A}_G(C)+\varepsilon}_{\mathcal{H}_{1,0}}(\Phi^a_{\mathcal{H}_{0,0}}(-,C),\overbar{C}):\CF^a(H;\alpha)\to\CF^{a+\mathcal{A}_G(C)+\mathcal{A}_{\overbar{G}}(\overbar{C})+2\epsilon}(H;\alpha)\]
	is a chain-homotopic to the chain-level inclusion map as filtered map.
\end{prop}
\begin{proof}
Fix a $C^2$-small Morse function $F$ satisfying $\norm{F}_{\infty}<\varepsilon$. We further assume that $F$ has a unique maximum $p_{\max}$. To apply the (\textbf{Associativity}) property, we consider the tuple of Hamiltonians
\[(H_0,H_1,H_2,H_3,H_{1.5},H_{2.5})=(H,H,G,\overbar{G},K,F).\]
For the four pair-of-pants products, we choose Hamiltonian-valued $1$-forms using the given $1$-forms $\mathcal{H}_{0,0}, \mathcal{H}_{1,0}$, along with additional $1$-forms $\mathcal{H}_{0,1}$ from $G, \overbar{G}$ to $F$, and $\mathcal{H}_{1,1}$ from $H,F$ to $H$.

By Proposition \ref{Proposition: Hofer-small Hamiltonian 1-form}, we assume that
\[E^-(\mathcal{H}_{0,1})<\varepsilon, \qquad E^-(\mathcal{H}_{1,1})<\varepsilon.\]
Then, there exists a $[0,1]$-family of Floer data, denoted by $\mathfrak{H}$, interpolating the composition of pair-of-pants products and satisfying $E^-(\mathfrak{H})<2\epsilon$. More explicitly, near $\tau=0, 1$, we construct the underlying Hamiltonian valued $1$-form $\mathcal{H}_\tau$ to by gluing the chosen Floer product data and interpolating using a Hamiltonian-valued $1$-form on the three-legged pants as in Proposition~\ref{Proposition: Hofer-small Family of Hamiltonian 1-form}. The interpolation can be done by a bump-function supported near $\tau=0,1$.

One obtains the underlying homotopy of almost complex structures for $\mathfrak{H}$ by choosing a generic isotopy of almost complex structures on the three-legged pair of pants interpolating the almost complex structures  obtained by gluing the chosen Floer product data.

Applying the (\textbf{Associativity}) property to this $\mathfrak{H}$, we conclude that the maps
\[\Phi_{\mathcal{H}_{1,0}}(\Phi_{\mathcal{H}_{0,0}}(-,C),\overbar{C})~\text{ and }~
\Phi_{\mathcal{H}_{1,1}}(-,\Phi_{\mathcal{H}_{0,1}}(C,\overbar{C}))\]
are chain-homotopic as maps
\[\CF_*^a(H;\alpha)\to\CF_*^{a+\mathcal{A}_G(C)+\mathcal{A}_{\overbar{G}}(\overbar{C})+2\epsilon}(H;\alpha).\]

Since $[C]=[\Phi_{\mathcal{G}}([\Sigma])]$ and $[\overbar{C}]=[\Phi_{\overbar{\mathcal{G}}}([\Sigma])]$, the cycle $\Phi_{\mathcal{H}_{0,1}}(C,\overbar{C})$ satisfies \[[\Phi_{\mathcal{H}_{0,1}}(C,\overbar{C})]=[\Phi_{F}([\Sigma])].\]
Since $F$ is a $C^2$-small Morse function, we conclude that
\[C_{\max}:=[p_{\max}, w_{\max}]=\Phi_{\mathcal{H}_{0,1}}(C,\overbar{C}) \in \CF^{\mathcal{A}_G(C)+\mathcal{A}_{\overbar{G}}(\overbar{C})+2\varepsilon}(F;\cont),\]
where $w_{\max}$ denotes the canonical capping of $p_{\max}$. We note that $C_{\max}$ is the only cycle in $\CF_2(F;\cont)$, and $\mathcal{A}_G(C)+\mathcal{A}_{\overbar{G}}(\overbar{C})$ is always non-negative.

Next, we consider the map \[\Phi_{\mathcal{H}_{1,1}}(-,C_{\max}):\CF^a(H;\alpha)\to\CF^{a+F(p_{\max})+\varepsilon}(H;\alpha).\]

From a linear homotopy from $0$ to $F$, we obtain a PSS-data  $\mathcal{F}$ for $F$ such that $E^-(\mathcal{F})<\varepsilon$ and $\Phi_\mathcal{F}([\Sigma])=C_{\max}$. Hence, by Proposition \ref{Proposition: PSS + Product = Continuation Filtered}, we obtain the identity
\[\Phi_{\mathcal{H}_{1,1}}(-,C_{\max})=\Phi_{\mathcal{F} \#_\ell {\mathcal{H}_{1,1}}}(-)\] at the chain-level. By construction, we also have $E^-({\mathcal{F}\#_\ell \mathcal{H}_{1,1}})<2\varepsilon$. 

Taking a linear homotopy between ${\mathcal{F}\#_\ell \mathcal{H}_{1,1}}$ and the constant Floer continuation data $\mathcal{H}=H$ on $\RR \times S^1$, we conclude that
\[\Phi_{\mathcal{H}_{1,1}}(-,C_{\max})\]
is chain-homotopic to the chain-level inclusion map 
\[\iota^{a,a+2\varepsilon}:\CF^a(H;\alpha)\to\CF^{a+2\varepsilon}(H;\alpha).\]
Thus, the map $\Phi_{\mathcal{H}_{1,0}}(\Phi_{\mathcal{H}_{0,0}}(-,C),\overbar{C})$ is also chain-homotopic to the chain-level inclusion map as a map $\CF^a(H;\alpha)\to\CF^{a+\mathcal{A}_G(C)+\mathcal{A}_{\overbar{G}}(\overbar{C})+2\epsilon}(H;\alpha)$.
This completes the proof.
\end{proof}

For two $\varepsilon$-admissible enhanced continuation data $\mathcal{HK}, \mathcal{KH}$ between $H$ and $K$, Proposition \ref{Proposition: PSS + Product = Continuation Filtered} and Proposition \ref{Proposition: Quantitative version of composition = identity} together imply that the following corollary.
\begin{cor}
	For two enhanced continuation data $\mathcal{HK}$ and $\mathcal{KH}$ between $H$ and $K$, the composition
	\begin{equation*}
		{\Phi}^{a+c_\im(\overbar{K}\#H;[\Sigma])+3\varepsilon}_{\mathcal{KH}}\circ{\Phi}^a_{\mathcal{HK}}:\CF^a_*(H;\alpha) \to \CF_*^{a+\gamma_\im(\overbar{H}\#K)+6\epsilon}(H;\alpha)
	\end{equation*}
	is chain-homotopic to the chain-level inclusion map
	\[\iota^{a,a+\gamma_\im(\overbar{H}\#K)+6\epsilon}:\CF^a_*(H;\alpha) \to \CF_*^{a+\gamma_\im(\overbar{H}\#K)+6\epsilon}(H;\alpha).\]
\end{cor}

\vspace{1cm}

\subsubsection{Bijection through enhanced Floer cylinders}
In this section, we complete the proof of Proposition \ref{Proposition: Existence of enhanced continuation data and bijection through enhanced Floer cylinder}.
For this, we fix a non-degenerate Hamiltonian $H$ and a subset $\mathcal{Y} \subset \mathcal{P}(H;\alpha)$ for a free homotopy class $\alpha\in\tilde{\pi}_1(\Sigma)$. Let $K$ be a non-degenerate Hamiltonian satisfying
\[\gamma_\im(\overbar{H}\#K)<\Delta(\mathcal{Y},\phi^1_H).\]
For such $K$, we choose $\varepsilon>0$ such that 
\[\gamma_\im(\overbar{H}\#K)+8\varepsilon<\Delta(\mathcal{Y},\phi^1_H).\]
Given such $\varepsilon>0$, we denote by $\mathcal{HK}$ and $\mathcal{KH}$ for $\varepsilon$-admissible enhanced continuation data from $H$ to $K$ and from $K$ to $H$ respectively.

Our goal is to construct a subset $\mathcal{Z}\subset \mathcal{P}(K;\alpha)$, an injection \[\mathfrak{R}_\mathcal{HK}:\mathcal{Y}\to\mathcal{Z}\] such that for each $y \in \mathcal{Y}$ there exists a Floer continuation cylinder for $\mathcal{HK}$ from $y$ to the $1$-periodic orbit $\mathfrak{R}_\mathcal{HK}(y) \in \mathcal{Z}$, and a bijection $$\mathfrak{S}_\mathcal{KH}: \mathcal{Y} \to \mathcal{Y}$$ such that for every $y \in \mathcal{Y}$, there exists a Floer continuation cylinder for $\mathcal{KH}$ from the $1$-periodic orbit $\mathfrak{R}_\mathcal{HK}(y) \in \mathcal{Z}$ to the $1$-periodic orbit $\mathfrak{S}_\mathcal{HK}(y) \in \mathcal{Y}$. 
We begin by proving the following lemma.
\begin{lem}	\label{Lemma: Composition of continuation map drops action for small-perturbation}
		Let $H, K$ be non-degenerate Hamiltonians on $\Sigma$. Let $\varepsilon >0$ and suppose that $\mathcal{Y}$ is a finite subset of $\mathcal{P}(H;\alpha)$ and that  such that 
		\begin{equation} \label{eq:isolation-hypothesis-for-lemma}
		    \gamma_\im(\overbar{H} \# K) + 8\varepsilon < \Delta(\mathcal{Y},\phi^1_H).
		\end{equation}
        Let $J$ be a $S^1$-family of compatible almost complex structures on $(\Sigma,\omega)$ that is regular for $H$, and such that $\mathcal{Y}$ is $(\Delta(\mathcal{Y},\phi^1_H)-\epsilon)$-Floer isolated with respect to $J$.
		Let $\mathcal{HK}, \mathcal{KH}$ be $\varepsilon$-enhanced continuation data between $H$ and $K$.
		Then, for any $y\in \mathcal{Y}$ and any capping $w$ of $y$ the support of the chain $\mathrm{supp}\left(\left({\Phi_{\mathcal{KH}}\circ\Phi_{\mathcal{HK}}([y,w]) -[y,w]}\right)\right)$ has the following property: 
        \begin{itemize}
            \item if $[y',w'] \in  \mathrm{supp}\left(\left({\Phi_{\mathcal{KH}}\circ\Phi_{\mathcal{HK}}([y,w]) - [y,w]}\right)\right)$, then either
		$y^\prime$ is not an element of $\mathcal{Y}$ or
        \[\mathcal{A}_H([y^\prime,w']) < \mathcal{A}_H ([y,w]).\]
        \end{itemize}
	\end{lem}

\begin{proof}
Under these assumptions, we apply Proposition \ref{Proposition: Quantitative version of composition = identity} for $\varepsilon>0$ and obtain the chain-homotopy $S:\CF_*(H;\alpha)\to \CF_{*+1}(H;\alpha)$   between \linebreak
$\Phi_{\mathcal{KH}}\circ\Phi_{\mathcal{HK}}$ and the identity map, i.e.
\[\Phi_{\mathcal{KH}}\circ\Phi_{\mathcal{HK}} - \id = d_{H,J} \circ S + S \circ  d_{H,J}.\]
From the construction of $S$, we obtain the apriori bound
\begin{align} \label{eq:action-increase-S}
	\mathcal{A}_H(S([y,w])) &\le \mathcal{A}_H([y,w]) +\gamma_\im(\overbar{H}\#K) + 6\epsilon.
\end{align}
We thus conclude that $\mathrm{supp}\left( S([y,w])\right)$ only contains elements with action smaller than $\mathcal{A}_H([y,w])+\gamma_\im(\overbar{H}\#K) + 6\epsilon$.


Let $[y',w'] \in  \mathrm{supp}\left(\left({\Phi_{\mathcal{KH}}\circ\Phi_{\mathcal{HK}}([y,w]) - [y,w]}\right)\right)$ and assume that $y' \in \mathcal{Y}$. Then, one of the following must be true:
\begin{itemize}
    \item[a)] The capped orbit $[y',w']$ appears in the support of the cycle \linebreak $d_{H,J} \circ S ([y,w])$.
    \item[b)] The capped orbit $[y',w']$ appears in the support of the cycle  \linebreak $S \circ d_{H,J} ([y,w])$.
\end{itemize}

In case a) holds, exists a capped $1$-periodic orbit $\sigma  \in \mathrm{supp}\left( S([y,w])\right)$ and a Floer cylinder $u$ appearing in the differential $d_{H,J}(\sigma)$ such that $u$ is a Floer cylinder from $\sigma$ to $[y',w']$, where $w'$ is some capping of $y'$. Because $y' \in \mathcal{Y}$ and $\mathcal{Y}$ is $(\Delta(\mathcal{Y},\phi^1_H)-\epsilon)$-Floer isolated with respect to $J$, we conclude that $E(u)>(\Delta(\mathcal{Y},\phi^1_H)-\epsilon)$, and therefore 
$$\mathcal{A}_H([y',w']) <  \mathcal{A}_H(\sigma) - \Delta(\mathcal{Y},\phi^1_H) + \epsilon.$$
By \eqref{eq:action-increase-S}, we know that $\mathcal{A}_H(\sigma) \leq \mathcal{A}_H([y,w]) +\gamma_\im(\overbar{H}\#K) + 6\epsilon   $. Combining these two inequalities we conclude that $$\mathcal{A}_H([y',w']) \le \mathcal{A}_H([y,w]) +\gamma_\im(\overbar{H}\#K) + 7\epsilon - \Delta(\mathcal{Y},\phi^1_H)<\mathcal{A}_H([y,w]),$$
where the last inequality follows from \eqref{eq:isolation-hypothesis-for-lemma}. 

Possibility b) is obtained by an identical reasoning. Namely, if $[y',w']$ appears in the support of the cycle $S \circ d_{H,J} ([y,w])$ and $y' \in \mathcal{Y}$
\begin{equation*}
    \mathcal{A}_H(([y',w'])) < \mathcal{A}_H([y,w])).
\end{equation*}
This completes the proof.
\end{proof}

With Lemma~\ref{Lemma: Composition of continuation map drops action for small-perturbation}, we now prove the existence of an injection $\mathfrak{R}_{\mathcal{HK}}$ and a bijection $\mathfrak{S}_{\mathcal{KH}}$ by dividing the argument into two cases. 

\textit{Case 1: $\Sigma \ne T^2$}.\\
From the maps $\Phi_{\mathcal{HK}}$ and $ \Phi_{\mathcal{KH}}$, we construct linear maps between the $\ZZ_2$-vector space $V(\mathcal{Y})$ generated by $\mathcal{Y}$ and the $\ZZ_2$-vector space $V(K;\alpha)$ generated by $\mathcal{P}(K;\alpha)$:
\[L_{\mathcal{HK}}:V(\mathcal{Y})\to V(K;\alpha)\quad \text{and} \quad L_{\mathcal{KH}}:V(K;\alpha)\to V(\mathcal{Y}).\]
The linear map $L_\mathcal{HK}$ is then defined as follows:
For each $y \in \mathcal{Y}$,
\[L_\mathcal{HK}(y):=\sum_{z \in \mathcal{P}(K;\alpha)} \#_2 \mathcal{M}(y,z;\mathcal{HK})_0 \cdot z\]
where $\mathcal{M}(y,z;\mathcal{HK})_0$ denotes the moduli space of enhanced continuation cylinders of Fredholm index $0$, defined as follows:
\begin{equation}
\label{Equation: Existence of injection bijection - 1}
\begin{aligned}
    \mathcal{M}(y,z;\mathcal{HK}) = \{ u: &\R \times S^1 \to \Sigma~ |~ \mathcal{F}_{\mathcal{HK}}(u)=0, ~u(-\infty)=y, ~u(+\infty)=z,\\ & \mbox{and } \CZ(z, w\#u)= \CZ (y, w) \mbox{ for some capping } w \mbox{ of } y  \}.
\end{aligned}
\end{equation}
We note that the condition
\begin{equation}
\label{Equation: Existence of injection bijection - 2}
  \CZ(z, w\#u)= \CZ (y, w)  
\end{equation}
does not depend on the choice of $w$. Moreover, after fixing a capping $w$ for $y_i$, there is at most one equivalence class of cappings satisfying \eqref{Equation: Existence of injection bijection - 2}. 

If $\Sigma\neq S^2$ or $T^2$, this follows from the uniqueness of the equivalence class of cappings. If $\Sigma=S^2$, there is again at most one equivalence class for each index. Therefore, in either case, the collection of capping data satisfying the index condition is either empty or consists of a single equivalence class.

Consequently, the moduli space $\mathcal{M}(y, z;\mathcal{HK})$ is either empty or can be identified with a moduli space
\[
\mathcal{M}([y, w],[z,w'];\mathcal{HK})
\]
for a pair of cappings satisfying the relation \eqref{Equation: Existence of injection bijection - 2}. In particular, the mod 2 count
\[
\#_2\mathcal{M}(y,z;\mathcal{HK})_0
\]
is well-defined. We extend the map $L_{\mathcal{HK}}$ linearly to all elements of $V(\mathcal{Y})$.



Similarly, we define $L_{\mathcal{KH}}$. 
For each $z \in \mathcal{P}(K;\alpha)$,
\[L_\mathcal{HK}(z)=\sum_{y \in \mathcal{Y}} \#_2 \mathcal{M}(z,y;\mathcal{KH})_0 \cdot z\]
where $\mathcal{M}(z,y;\mathcal{KH})_0$ is the moduli space of enhanced continuation cylinders of Fredholm index $0$ obtained by
interchanging the roles of $H$ and $K$.

\begin{lem}
\label{Lemma: Composition = Unipotent}
	The composition of these linear maps $L_{\mathcal{HK}}:V(\mathcal{Y})\to V(K;\alpha)$ and $L_{\mathcal{KH}}:V(K;\alpha)\to V(\mathcal{Y})$ is unipotent.
    Equivalently, there exists a natural number $n>0$ such that
    \begin{equation}
    \label{Equation: Existence of injection bijection - 3}
        \left(L_\mathcal{KH} \circ L_\mathcal{HK} - \id_{V(\mathcal{Y})}\right)^n=0.
    \end{equation}
\end{lem}

\begin{proof}
Take $n=|\mathcal{Y}|$ and fix the ordered basis $\mathcal{Y}=\{y_1,\ldots,y_n\}$ of $V(\mathcal{Y})$. Then the linear map $L_\mathcal{KH}\circ L_\mathcal{HK}$ is represented by the matrix
whose $(j,i)$-entry is $a_{ji}$, where
\[
L_\mathcal{KH}\circ L_\mathcal{HK}(y_i)
=
\sum_{j=1}^n a_{ji}\,y_j,
\]
with $a_{ji}\in\mathbb{Z}_2$. By construction, coefficients $a_{ji}$ are given by
\[a_{j i}= \#_2 \left( \bigcup_{z \in \mathcal{P}(K;\alpha)} \mathcal{M}(y_i,z;\mathcal{HK})_0 \times \mathcal{M}(z,y_j;\mathcal{KH})_0 \right).\]
When $a_{ji} \neq 0$, fix a capping $w$ of $y_i$. Then, from the discussion above, there exists a unique equivalence class of cappings $w'$ of $y_j$ satisfying
\[
\mu(y_i,w)=\mu(y_j,w').
\]
Moreover, $a_{ji}$ agrees with the coefficient of $[y_j,w']$ in the image
\[
\Phi_{\mathcal{KH}}\circ\Phi_{\mathcal{HK}}([y_i,w]).
\]

Consider, the matrix
\[
L_\mathcal{KH} \circ L_\mathcal{HK} - \id_{V(\mathcal{Y})}.
\]
Together with Lemma \ref{Lemma: Composition of continuation map drops action for small-perturbation}, this implies that the $(j,i)$-entry $a_{ji}-\delta_{ji}$ is non-zero only if
\[
\mathcal{A}_H([y_j,w']) < \mathcal{A}_H([y_i,w]),
\]
where $\delta_{ij}$ denotes the Kronecker delta.

{ We next prove that equation \eqref{Equation: Existence of injection bijection - 3} holds for
$n=|\mathcal{Y}|$.} If the $(j,i)$-entry of the $|\mathcal{Y}|$-th iterate is non-zero, then there exists a sequence
\[
i=k_0,~k_1,~\ldots,~k_n=j
\]
such that
\[
a_{k_\ell k_{\ell-1}}\neq 0
\]
for every $\ell=1,\ldots,n$. Hence, there exist cappings $w_\ell$ for each $1$-periodic orbit $y_{k_\ell}$ such that
\begin{equation*}
\begin{aligned}
    \mathcal{A}_H([y_{k_n},w_n]) < \mathcal{A}_H([y_{k_{n-1}},w_{n-1}]) < \cdots < \mathcal{A}_H([y_{k_0},w_0]),\\
    \CZ([y_{k_n},w_n]) = \CZ([y_{k_{n-1}},w_{n-1}]) = \cdots = \CZ([y_{k_0},w_0]).
\end{aligned}
\end{equation*}
Hence, there exist distinct subscripts $m$ and $m'$ such that
\[
y_{k_m}=y_{k_{m'}}=y_k.
\]
Since there is at most one equivalence class of cappings of $y_k$ for a given index, we obtain a contradiction from the strict inequality of Hamiltonian actions. This shows that \eqref{Equation: Existence of injection bijection - 3} holds and completes the proof the lemma.
\end{proof}

Since the determinant of a unipotent matrix is equal to $1$, we have
\[
\det(L_\mathcal{KH}\circ L_\mathcal{HK})=1.
\]
We now present the proof of Proposition~\ref{Proposition: Existence of enhanced continuation data and bijection through enhanced Floer cylinder}. To this end, we recall the following form of the Cauchy--Binet formula. For two matrices $A=(a_{ij})$ and $B=(b_{ij})$ of sizes $m\times n$ and $n\times m$, respectively,
\begin{equation}
\begin{aligned}
\label{Equation: Cauchy-Binet formula}
\det(A\cdot B)=\sum_{f,g}\prod_{i\in\set{1,2,\ldots,n}}
a_{g(i)f(i)}\cdot b_{f(i)i},
\end{aligned}
\end{equation}
where the summation is taken over all injections  $f:\{1,\ldots,n\}\hookrightarrow\{1,\ldots,m\}$ and all bijections $g:\{1,\ldots,n\}\rightarrow\{1,\ldots,n\}$.

\begin{proof}[Proof of Proposition \ref{Proposition: Existence of enhanced continuation data and bijection through enhanced Floer cylinder} for $\Sigma \ne T^2$] First, we take a Hofer-small perturbation $G$ of $\overbar{H}\# K$. More precisely we choose $G$ to be a non-degenerate Hamiltonian satisfying \[\norm{\overbar{K}\# H \#G}_{\Hofer}<\varepsilon,\]
where we recall that our standing choice of $\varepsilon>0$ satisfies \eqref{eq:epsilon-for-action-gap}. 
From the construction in Section \ref{Section: Construction of enhanced Floer continuation data}, we obtain an $\varepsilon$-admissible enhanced continuation data $\mathcal{HK}$ from $H$ to $K$. Similarly, by considering $\overbar{G}$, we obtain  an $\varepsilon$-admissible enhanced continuation $\mathcal{KH}$ data from $K$ to $H$.

Next, we consider the $\ZZ_2$-vector space $V(\mathcal{Y})$ generated by $\mathcal{Y}$ and $\ZZ_2$-vector space $V(K;\alpha)$ generated by $\mathcal{P}(K;\alpha)$. Using the moduli spaces
\[ \mathcal{M}(y, z;\mathcal{HK})_0~\text{ and }~\mathcal{M}(z, y;\mathcal{KH})_0\]
defined in \eqref{Equation: Existence of injection bijection - 1}, we define the linear maps \[L_{\mathcal{HK}}:V(\mathcal{Y})\to V(K;\alpha)\quad \text{and} \quad L_{\mathcal{KH}}:V(K;\alpha)\to V(\mathcal{Y})\]
as described above. By Lemma \ref{Lemma: Composition = Unipotent}, the composition of these linear maps is unipotent. In particular, the determinant of the composition is non-zero.

Again, we regard
\[
\mathcal{Y}=\{y_1,\ldots,y_n\} \quad\text{and}\quad \mathcal{P}(K;\alpha)=\{z_1,\ldots,z_m\}
\]
as ordered bases of $V(\mathcal{Y})$ and $V(K;\alpha)$, respectively. Since the determinant is non-zero, the Cauchy--Binet formula \eqref{Equation: Cauchy-Binet formula} implies that there exist a bijection
\[
\mathfrak{R}_{\mathcal{HK}}:\mathcal{Y}\to\mathcal{Z}
\quad\text{and}\quad
\mathfrak{S}_{\mathcal{KH}}:\mathcal{Y}\to\mathcal{Y}
\]
such that, for every $1$-periodic orbit $y\in\mathcal{Y}$, the following hold:
\begin{itemize}
	\item The $\mathfrak{R}_{\mathcal{HK}}(y)$-coefficient of $L_{\mathcal{HK}}(y)$ is non-zero.
	\item The $\mathfrak{S}_{\mathcal{KH}}(y)$-coefficient of
	$L_{\mathcal{KH}}(\mathfrak{R}_{\mathcal{HK}}(y))$ is non-zero.
\end{itemize}

By the definitions of $L_\mathcal{HK}$ and $L_\mathcal{KH}$, this implies that both moduli spaces
\[
\mathcal{M}(y,\mathfrak{R}_{\mathcal{HK}}(y);\mathcal{HK})_0
\quad\text{and}\quad
\mathcal{M}(\mathfrak{R}_{\mathcal{HK}}(y),\mathfrak{S}_{\mathcal{KH}}(y);\mathcal{KH})_0
\]
are non-empty for every $1$-periodic orbit $y\in\mathcal{Y}$. This completes the proof of the proposition.
\end{proof}

\textit{Case 2: $\Sigma = T^2$}.\\
When $\alpha=\cont$, the proof is exactly the same as in Case 1. Hence, we only need to consider the case of non-contractible loops. In this case, the main difficulty is that a given loop with a fixed CZ-index may admit infinitely many cappings. In principle, the $0$-dimensional moduli space
\[
\mathcal{M}(y,z;\mathcal{HK})_0
\]
defined as in \eqref{Equation: Existence of injection bijection - 1} need not be finite.

Here, as in the construction of the Floer chain complex in Section~\ref{sec:Hamontorus} on $T^2$, we work with the downward completion. To this end, we introduce the Novikov field
\[
\Lambda := \left\{ \sum_{j\geq 0} a_j T^{e_j} \;\middle|\; a_j\in\mathbb{Z}_2,\ e_j\in\mathbb{R},\ \#\{j\mid e_j<\lambda\}<\infty \ \forall \lambda\in\mathbb{R} \right\}.
\]
This time, we define linear maps between the $\Lambda$-vector spaces $V(\mathcal{Y})$ and $V(K;\alpha)$ generated by $\mathcal{Y}$ and $\mathcal{P}(K;\alpha)$, respectively:
\[L_{\mathcal{HK}}:V(\mathcal{Y})\to V(K;\alpha)\quad \text{and} \quad L_{\mathcal{KH}}:V(K;\alpha)\to V(\mathcal{Y}).\]

The linear map $L_\mathcal{HK}$ is then defined as follows: For each $y \in \mathcal{Y}$,
\begin{equation}
\label{Equation: Map L for torus case}
    L_\mathcal{HK}(y):=\sum_{z \in \mathcal{P}(K;\alpha)}\,\sum_{u \in \mathcal{M}(y,z;\mathcal{HK})_0} T^{e(u)}  \cdot z
\end{equation}
where $\mathcal{M}(y,z;\mathcal{HK})_0$ denotes the moduli space of enhanced continuation cylinders of index $0$, defined as \eqref{Equation: Existence of injection bijection - 1}. The exponent $e(u)$ is given by
\[
e(u):=
\int_{S^1}\bigl(H(y_i(t))-K(z_j(t))\bigr)\,dt -\int_{\mathbb{R}\times S^1}u^*\omega.
\]
Equivalently, for a capping $w$ of $y_i$, we have
\[
e(u)=\mathcal{A}_H([y_i,w])-\mathcal{A}_K([z_j,w\#u]).
\]

Again, note that since the first Chern class $c_1(TT^2)$ vanishes, the CZ-index of a capped orbit is independent of the choice of capping for a given reference data $(\eta_\alpha, \tau_\alpha)$. Hence, in this case, we omit the capping from the notation for the CZ-index. As in Case 1, the moduli space $\mathcal{M}(y, z; \mathcal{HK})_0$ can be non-empty only if
\begin{equation}
\label{Equation: Existence of injection bijection - 4}
\CZ(y) - \CZ(z)=0.
\end{equation}
Recall that the CZ-index may depend on the choice of the reference data $\eta_\alpha$ and the trivialization $\tau_\alpha$. However, with a common choice of these data, the difference does not depend on such choices.

When the relation \eqref{Equation: Existence of injection bijection - 4} holds, the $z$-coefficient of $L_\mathcal{HK}(y)$ can be written as
\begin{equation}
\label{Equation: Existence of injection bijection - 5}
\sum_{w'} \#_2 \mathcal{M}([y,w],[z,w'];\mathcal{HK}) \cdot T^{\mathcal{A}_H([y,w])-\mathcal{A}_H([z,w'])}
\end{equation}
for some capping $w$ of $y_i$.  Again, the coefficient does not depend on the choice of capping. Moreover, \eqref{Equation: Existence of injection bijection - 5} defines a formal power series with coefficients in $\mathbb{Z}_2$ and possibly negative real exponents. By Gromov--Floer compactness, the $z$-coefficient satisfies the Novikov finiteness condition and hence belongs to the Novikov field $\Lambda$.


To see this, recall that the exponent and the energy are related by
\begin{equation}
    \label{Equation: Existence of injection bijection - 6}
    -E^+(\mathcal{H}K) \le E_{\mathcal{HK}}(u)-(\mathcal{A}_H([y,w])-\mathcal{A}_K([z,w']))\le E^-(\mathcal{HK}).
\end{equation}
By Gromov--Floer compactness, for any given upper bound on the energy, there are only finitely many enhanced continuation cylinders. Since the orbit and the continuation data are fixed, inequality \eqref{Equation: Existence of injection bijection - 6} then implies that there are only finitely many enhanced continuation cylinders satisfying
\[
e(u)\leq \lambda,
\]
for any $\lambda\in\mathbb{R}$. Hence, the coefficient belongs to $\Lambda$. We extend the map $\mathcal{L}_{\mathcal{HK}}$ linearly to all elements of $V(\mathcal{Y})$. Similarly, we define $\mathcal{L}_{\mathcal{KH}}$.

Here, Lemma~\ref{Lemma: Composition of continuation map drops action for small-perturbation} also applies. As a result, we obtain the following lemma, which is a version of Lemma~\ref{Lemma: Composition = Unipotent} for Case 2. To state the result, we define 
\[
    \Lambda_+ := \left\{ \sum_{j\geq 0} a_j T^{e_j} \;\middle|\; a_j\in\mathbb{Z}_2,\ e_j \in \mathbb{R}^+,\ \#\{j\mid e_j<\lambda\}<\infty \ \forall \lambda\in\mathbb{R} \right\}.
    \]
\begin{lem}
\label{Lemma: Composition = Identity + higher terms}
    The composition of these linear maps $L_{\mathcal{HK}}:V(\mathcal{Y})\to V(K;\alpha)$ and $L_{\mathcal{KH}}:V(K;\alpha)\to V(\mathcal{Y})$ satisfies the following: with respect to the ordered basis $\mathcal{Y}=\{y_1,\ldots,y_n\}$ of $V(\mathcal{Y})$, every entry satisfies
    \[
    \left(L_\mathcal{KH} \circ L_\mathcal{HK} - \id_{V(\mathcal{Y})}\right)_{j i} \in \Lambda_+.
    \]
\end{lem}
\begin{proof}
    Assume that the $(j,i)$-entry
    \[
        \left(L_\mathcal{KH}\circ L_\mathcal{HK}-\id_{V(\mathcal{Y})}\right)_{ji}
    \]
    is non-zero. Recall that when $i\neq j$, this entry is the same as the $(j,i)$-entry of
    $L_\mathcal{KH}\circ L_\mathcal{HK}$, and it can be non-zero only if
    \[
    \CZ(y_i)=\CZ(y_j)
    \]
    for some cappings $w$ and $w'$ of $y_i$ and $y_j$, respectively. This condition is automatically satisfied when $i=j$. Hence, we only consider the case where the CZ-indices are the same.

    From Equation \eqref{Equation: Existence of injection bijection - 5}, the $(j,i)$-entry of $L_\mathcal{KH}\circ L_\mathcal{HK}$ is given by
    \begin{equation*}
    \begin{aligned}
    \sum_{w'}\sum_{[z,w'']} \#_2\mathcal{M}([y_i,w],[z,w''];\mathcal{HK})
    \cdot \#_2 \mathcal{M}([z,w''],[y_j,w'];\mathcal{KH})\cdot
    T^{e_{ij}(w,w')},
\end{aligned}
\end{equation*}
where, for convenience, we denote the action difference by
\[
e_{ij}(w,w'):=\mathcal{A}_H([y_i,w])-\mathcal{A}_H([y_j,w']).
\]
The first summation runs over all possible cappings $w'$ of $y_j$, and the second summation runs over all possible capped orbits $[z,w'']$ satisfying
\[
\CZ(y_i)=\CZ(z)=\CZ(y_j).
\]    
    
    In particular, for each capping $w'$ of $y_j$, the sum
    \[
    \sum_{[z,w'']} \#_2 \mathcal{M}([y,w],[z,w''];\mathcal{HK}) \cdot \#_2 \mathcal{M}([z,w''],[y_j,w'];\mathcal{KH})
    \]
    agrees precisely with the $[y_j,w']$-coefficient of
    \[\Phi_{\mathcal{KH}}\circ \Phi_{\mathcal{HK}}([y_i,w]).\]
    By Lemma~\ref{Lemma: Composition of continuation map drops action for small-perturbation}, this sum can only be non-zero if either
    \[
        [y_j,w']=[y_i,w] \quad\text{or}\quad \mathcal{A}_H([y_j,w''])<\mathcal{A}_H([y_i,w]).
    \]
    In the first case $[y_j,w'']=[y_i,w]$, the coefficient is precisely $1$.
    
    Combining the above observations, the $(j,i)$-entry
    \[(L_{\mathcal{KH}}\circ L_{\mathcal{HK}}-\id_{V(\mathcal{Y})})_{ji}\] is given by
    \begin{equation*}
    \begin{aligned}
        \sum_{\substack{w':\text{ capping of }y_j\\ \mathcal{A}_H([y_j,w'])<\mathcal{A}_H([y_i,w])}} \sum_{\substack{[z,w''] \in \widetilde{\mathcal{P}}(K;\alpha) \\ \CZ(z) = \CZ(y_i)}}\#_2 &\mathcal{M}([y,w],[z,w''];\mathcal{HK}) \\ \cdot \#_2 &\mathcal{M}([z,w''],[y_j,w'];\mathcal{KH}) \cdot T^{e_{ij}(w,w')}.
    \end{aligned}
    \end{equation*}
    Hence, every exponent appearing above is positive, and the $(j,i)$-entry belongs to $\Lambda_+$. This completes the proof.
\end{proof}

\begin{proof}[Proof of Proposition \ref{Proposition: Existence of enhanced continuation data and bijection through enhanced Floer cylinder} for $\Sigma = T^2$] First, we take a Hofer-small perturbation $G$ of $\overbar{H}\# K$. More precisely we choose $G$ to be a non-degenerate Hamiltonian satisfying \[\norm{\overbar{K}\# H \#G}_{\Hofer}<\varepsilon,\]
where we recall that our standing choice of $\varepsilon>0$ satisfies  \eqref{eq:epsilon-for-action-gap}. 
From the construction in Section \ref{Section: Construction of enhanced Floer continuation data}, we obtain an $\varepsilon$-admissible enhanced continuation data $\mathcal{HK}$ from $H$ to $K$. Similarly, by considering $\overbar{G}$, we obtain  an $\varepsilon$-admissible enhanced continuation $\mathcal{KH}$ data from $K$ to $H$.

When $\alpha=\cont$, the proof is exactly the same as in the case $\Sigma\ne T^2$. Next, we consider the $\Lambda$-vector space $V(\mathcal{Y})$ generated by $\mathcal{Y}$ and the $\Lambda$-vector space $V(K;\alpha)$ generated by $\mathcal{P}(K;\alpha)$. Using the moduli spaces
\[ \mathcal{M}(y, z;\mathcal{HK})_0~\text{ and }~\mathcal{M}(z, y;\mathcal{KH})_0\]
defined in \eqref{Equation: Existence of injection bijection - 1}, we define the linear maps
\[L_{\mathcal{HK}}:V(\mathcal{Y})\to V(K;\alpha)\quad \text{and} \quad L_{\mathcal{KH}}:V(K;\alpha)\to V(\mathcal{Y})\]
as described above.

By Lemma~\ref{Lemma: Composition = Identity + higher terms}, with respect to the ordered basis $\mathcal{Y}$, the matrix of the composition satisfies
\[
L_{\mathcal{KH}}\circ L_{\mathcal{HK}} = \id_{V(\mathcal{Y})}+N,
\]
where every entry of $N$ belongs to $\Lambda_+$. Therefore, the determinant satisfies
\[
\det(L_{\mathcal{KH}}\circ L_{\mathcal{HK}}) = 1+\text{(terms in }\Lambda_+).
\]
In particular, the determinant of the composition is non-zero. The rest of the proof is the same as in Case 1. This completes the proof of the proposition for Case 2.
\end{proof}

\vspace{0.5cm}

\subsection{Braid isotopy through enhanced Floer cylinders}
\label{Section: Isotopy via enhanced Floer cylinders}

\

Let $\alpha \in \tilde{\pi}_1(\Sigma)$, $H,K:S^1 \times \Sigma \to \R$ be non-degenerate Hamiltonians, and $\mathcal{Y} \in \mathcal{P}(H;\alpha)$ satisfy the hypothesis of Proposition  \ref{Proposition: Existence of enhanced continuation data and bijection through enhanced Floer cylinder}.
In this subsection, we prove that for the collection $\mathcal{Z}$ of $1$-periodic orbits of $K$ obtained in that proposition, the braid induced by  $(\mathcal{Z},K)$ is freely isotopic to the braid induced by $(\mathcal{Y},H)$. Moreover, the free isotopy is realized by the enhanced Floer cylinders constructed in the proof of Proposition \ref{Proposition: Existence of enhanced continuation data and bijection through enhanced Floer cylinder}; see Figure \ref{fig:Free isotopy through enhanced Floer cylinders} below. This is the content of the following theorem.

\begin{figure}
    \centering
    \includegraphics[width=\linewidth]{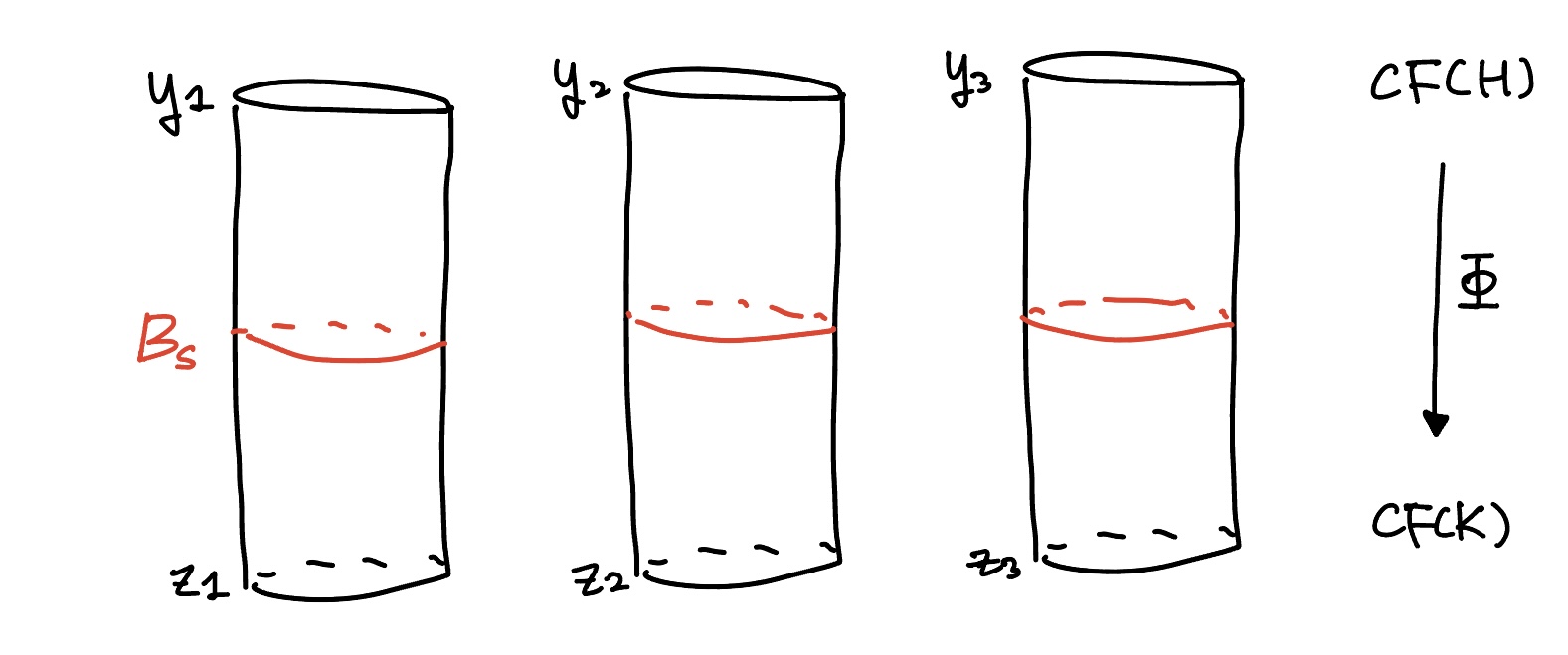}
    \caption{Free isotopy through enhanced Floer cylinders}
    \label{fig:Free isotopy through enhanced Floer cylinders}
\end{figure}

\begin{thm}
\label{Theorem: Braid Stability under PSS-image spectral distance}
	Let $(\Sigma,\omega)$ be a closed symplectic surface, and let $\alpha\in\tilde{\pi}_1(\Sigma)$ be a free homotopy class of loops in $\Sigma$. Suppose that two non-degenerate Hamiltonians $H, K$ and a subset $\mathcal{Y} \subset \mathcal{P}(H;\alpha)$ satisfy the inequality
	\[\gamma_\im(\overbar{H}\#K)<\Delta(\mathcal{Y},\phi^1_H).\]
	Then, there exist a subset $\mathcal{Z} \subset \mathcal{P}(K;\alpha)$ such that the braids
	\[\BB(\mathcal{Y},H)\quad\text{and}\quad \BB(\mathcal{Z};K)\]
	are freely isotopic as braids.
\end{thm}

\begin{proof}


By Proposition \ref{Proposition: Existence of enhanced continuation data and bijection through enhanced Floer cylinder}, there exist
\begin{enumerate}
	\item  two enhanced continuation data
    $\mathcal{HK}, \mathcal{KH}$ between Floer data $(H,J_H)$ and $(K,J_K)$, chosen under the assumption that $\varepsilon>0$ is sufficiently small to satisfy
	\[\gamma_\im(\overbar{H}\#K) + 6\varepsilon < \Delta_H(y),\]
	\item a subset $\mathcal{Z} \subset \mathcal{P}(K;\alpha)$, and
	\item a injection $\mathfrak{R}_{\mathcal{HK}}: \mathcal{Y} \to \mathcal{Z}$ and a bijection $\mathfrak{S}_{\mathcal{KH}}: \mathcal{Y} \to \mathcal{Y}$ such that both moduli spaces
    \begin{align*}
    & \mathcal{M}(y; \mathfrak{R}_{\mathcal{HK}}(y) ; \mathcal{HK})_0, \\ & \mathcal{M}(\mathfrak{R}_{\mathcal{HK}}(y), \mathfrak{S}_{\mathcal{KH}}(y) ; \mathcal{KH})_0
    \end{align*}
	are not empty for every $1$-periodic orbit $y$ in $\mathcal{Y}$.
\end{enumerate}

We denote by $\mathcal{J}_1$ the underlying family of almost complex structures that is part of the enhanced continuation data $\mathcal{HK}$, and by $\mathcal{J}_2$ the underlying family of almost complex structures that is part of the enhanced continuation data $\mathcal{KH}$. These will play an important role in the arguments that follow.

From now on, we abbreviate $\mathfrak{R}_{\mathcal{HK}}$ and $\mathfrak{S}_{\mathcal{KH}}$ as $\mathfrak{R}$ and $\mathfrak{S}$, respectively. For each $1$-periodic orbit $y$ in $\mathcal{Y}$, fix an enhanced continuation cylinder $u_y \in \mathcal{M}(y; \mathfrak{R}(y) ; \mathcal{HK})_0$.

We introduce the following terminology. If $u: Z \to \Sigma$ is a differentiable map we let $\tilde{u}$  denote the embedded cylinder in $\widetilde{\Sigma}:=Z\times \Sigma$ obtained as the graph  of the map $u$. We proceed to prove the following result for the set of enhanced continuation cylinders $\{ u_y \}_{y \in \mathcal{Y}}$ chosen above.

\begin{prop}
\label{Proposition: No intersection between graph of enhanced cylinders}
	For every pair of distinct $1$-periodic orbits $y, y^\prime$ in $\mathcal{Y}$, we have that
	\[\tilde{u}_y \cap \tilde{u}_{y^\prime}=\emptyset.\]
\end{prop}

We first proceed to complete the proof of Theorem \ref{Theorem: Braid Stability under PSS-image spectral distance} assuming Proposition \ref{Proposition: No intersection between graph of enhanced cylinders}. We claim that the $\mathcal{Y}$-family of enhanced Floer cylinder $\set{u_y}_{y\in\mathcal{Y}}$ provides a free isotopy between the braids $\BB(\mathcal{Y},H)$ and $\BB(\mathcal{Z},K)$. Let $\mathfrak{f}:(0,1)\to \RR$ be a diffeomorphism that extends continuously to a homeomorphism $\bar{\mathfrak{f}}:[0,1]\to \overline{\RR}=\RR \cup \set{-\infty,\infty}$. Define a $[0,1]$-family of maps $\set{\BB_\tau}_{\tau\in[0,1]}$ by
\[\BB_\tau:\bigcup_{y \in \mathcal{Y}} S^1_y \to S^1\times \Sigma\]
where each $\BB_{\tau, y}(t)$ is given by
\begin{align*}
	\BB_{\tau,y}(t)=
	\begin{cases}
		(t, y(t)) & \text{for } \tau=0,\\
		(t, u_y(\mathfrak{f}(\tau), t)) & \text{for } \tau \in (0,1),\\
		(t, \mathfrak{R}(y)(t)) & \text{for } \tau=1.
	\end{cases}
\end{align*}
In particular, we have $\mathcal{B}_0 = \mathcal{B}(\mathcal{Y, H})$ and $\mathcal{B}_1 = \mathcal{B}(\mathcal{Z,K})$.

By Proposition \ref{Proposition: No intersection between graph of enhanced cylinders}, for every distinct $y, y^\prime$ in $\mathcal{Y}$, we have
\[ u_y(s,t) \ne u_{y^\prime}(s,t) \quad \text{ for all } (s, t) \in \RR \times S^1.\]
Equivalently, for every distinct $y, y^\prime$ in $\mathcal{Y}$, we have
\[ \mathcal{B}_{\tau, y}(t) \ne \mathcal{B}_{\tau, y^\prime}(t) \quad \text{ for all } (\tau,t)\in (0,1) \times S^1\]
ensuring that each $\mathcal{B}_\tau$ is a braid. Therefore, the family $\set{\BB_\tau}_{\tau\in[0,1]}$ defines a free isotopy between $\mathcal{B}(\mathcal{Y, H})$ and $\mathcal{B}(\mathcal{Z, K})$. This completes the proof.
\end{proof}


Hence, it remains to prove Proposition \ref{Proposition: No intersection between graph of enhanced cylinders}. The proof consists of two steps. First, we observe that the graphs of Floer continuation cylinders always intersect positively. Next, we prove that the total intersection number is zero. From these two results, we conclude that the graphs do not intersect.

The first step will follow from Lemma \ref{Lemma: Positivity of intersection between graph of enhanced cylinders} below. To state it, we define the set of intersections between graphs $\tilde{u}_y$ and $\tilde{u}_{y^\prime}$ as
\[ I(u_y,u_{y^\prime}):=\set{(s,t) \in \RR\times S^1: u_y(s,t)=u_{y^\prime}(s,t)}.\]
The lemma will be obtained from combining the Gromov trick and the positivity of intersections for pseudoholomorphic curves in $4$-dimensional almost complex manifolds.

\begin{lem}
	\label{Lemma: Positivity of intersection between graph of enhanced cylinders}
	For every $y \ne y^\prime$ in $\mathcal{Y}$, the set $I(u_y, u_{y^\prime})$ is isolated in $\RR\times S^1$. Moreover, for each point $(s,t) \in I(u_y, u_{y^\prime})$, the local intersection number of $\tilde{u}_y, \tilde{u}_{y^\prime}$ at $(s,t)$, denoted by $\iota(\tilde{u}_y,\tilde{u}_{y^\prime}; s, t)$, is well-defined and positive.
\end{lem}

\begin{proof}
For any continuation data $(\mathcal{H},\mathcal{J})$, there exists an almost complex structure $\widetilde{\mathcal{J}}_{\mathcal{H},\mathcal{J}}$ on $\widetilde{\Sigma}=\R\times S^1 \times \Sigma$ such that, for any map $u:\RR\times S^1 \to \Sigma$, the following conditions are equivalent:
\begin{itemize}
	\item the map $u$ is a solution to $\mathcal{F}_{H,J}(u)=0$
	\item the graph $\tilde{u}$ is a $\widetilde{\mathcal{J}}_{\mathcal{H},\mathcal{J}}$-holomorphic section.
\end{itemize}
Here, we identify graphs as sections of the trivial fibration $\pi:\widetilde{\Sigma}\to \RR \times S^1$.

The construction of such an almost complex structure $\widetilde{\mathcal{J}}_{\mathcal{H},\mathcal{J}}$ is known as the Gromov trick. Explicitly, for $\mathcal{H}=F(s,t) ds+G(s,t) dt$ on a cylinder and a $\R \times S^1$-family of almost complex structures $\mathcal{J}=\set{J_{s,t}}_{(s,t)\in \RR\times S^1}$, the almost complex structure $\widetilde{\mathcal{J}}_{\mathcal{H},\mathcal{J}}$ on $\widetilde{\Sigma}$ is given by
\begin{equation*}
	\widetilde{\mathcal{J}}_{\mathcal{H},\mathcal{J}}:=
	\begin{pmatrix}
		~0 &-1 &~0\\
		~1 &0 &~0\\
		~-J_{s,t}X_{F} + X_{G} &-J_{s,t}X_{G}-X_{F} &~J_{s,t}~
	\end{pmatrix},
\end{equation*}
where the matrix representation is written with respect to the decomposition $T_s\RR \oplus T_t S^1 \oplus T_p \Sigma$.

Since, for $y,y'\in \mathcal{Y}$, the cylinders $u_y$ and $u_{y'}$ are solutions of the Floer equation $\mathcal{F}_{\mathcal{HK},\mathcal{J}_1}(u)=0$, it follows that their graphs $\tilde{u}_y$ and $\tilde{u}_{y^\prime}$ are \linebreak $\widetilde{\mathcal{J}}_{\mathcal{HK},\mathcal{J}_1}$-holomorphic curves in almost complex $4$-manifold $(\widetilde{\Sigma},\widetilde{\mathcal{J}}_{\mathcal{HK},\mathcal{J}_1})$. Applying the positivity of intersections for holomorphic curves in almost complex $4$-manifolds, we conclude that the set $I(u_y, u_{y^\prime})$ is isolated in $\RR\times S^1$ for every pair $y \neq  y^\prime$ in $\mathcal{Y}$. Moreover, positivity of intersections guarantees that for each point $(s,t) \in I(u_y, u_{y^\prime})$, the local intersection number of $\tilde{u}_y, \tilde{u}_{y^\prime}$ at $(s,t)$, denoted by $\iota(\tilde{u}_y,\tilde{u}_{y^\prime}; s, t)$, is well-defined and positive. This completes the proof.
\end{proof}

Since the Floer cylinder $u_y$ converges to $y$ at $-\infty$ and to $F(y)$ at $+\infty$, intersections between the graphs $\tilde{u}_y$ and $\tilde{u}_{y^\prime}$ can only appear on a compact set of $\RR \times S^1 \times \Sigma$. It then follows from Lemma \ref{Lemma: Positivity of intersection between graph of enhanced cylinders}, that for every pair $y \neq y' \in \mathcal{Y}$ the set $I(u_y,u_{y'})$ is finite, and that the sum of local intersection numbers
\[m_u(y,y^\prime):=\sum_{(s,t) \in I(u_y,u_{y'})} \iota(\tilde{u}_y,\tilde{u}_{y^\prime}; s, t)\]
is finite. Moreover, $m_u(y,y^\prime)=0$ if and only if $I(u_y,u_{y'})=\emptyset$, or equivalently if $\tilde{u}_y \cap \tilde{u}_{y^\prime}=\phi$.

We will show that for any pair $y \neq y' \in \mathcal{Y}$ we have $m_u(y,y^\prime)=0$.
To prove that, we will also consider enhanced continuation cylinders from the moduli spaces $\mathcal{M}_\im(\mathfrak{R}(y), \mathfrak{S}(y) ; \mathcal{KH})$.  For each $1$-periodic orbit $y$ in $\mathcal{Y}$, fix any enhanced continuation cylinder $v_y \in \mathcal{M}_\im(\mathfrak{R}(y), \mathfrak{S}(y) ; \mathcal{KH})$.
The intersection count $m_v(y,y^\prime)$ is defined similarly as $m_u(y, y^\prime)$ but using the cylinders $\tilde{v}_y, \tilde{v}_{y^\prime}$ instead. The same reasoning applied above implies that $m_u(y, y^\prime) \geq 0$, and that $m_u(y, y^\prime) = 0$ if and only if $I(v_y,v_{y'})=\emptyset$.
We are now in position to prove Proposition \ref{Proposition: No intersection between graph of enhanced cylinders}.

\begin{proof}[Proof of Proposition \ref{Proposition: No intersection between graph of enhanced cylinders}] Let $M$ be a natural number such that the iterate \linebreak $\mathfrak{S}^M:\mathcal{Y} \to \mathcal{Y}$ is an identity. For each $1$-periodic orbit $y$ in $\mathcal{Y}$, we consider the tuples of enhanced Floer continuation cylinders
\[U_y:=\left(u_y, v_y, u_{\mathfrak{S}(y)}, v_{\mathfrak{S}(y)}, \ldots, u_{\mathfrak{S}^{M-1}(y)}, v_{\mathfrak{S}^{M-1}(y)}\right).\]
We denote by $U_{y,i}$ the $i$-th component of the tuple $U_y$. From the definition, two adjacent Floer cylinders have the matching asymptotics, i.e.
\[U_{y,i}(+\infty)=U_{y, i+1}(-\infty)\]
for any $1\le i\le 2M-1$. Since $\mathfrak{S}^M(y)=y$, we also have
\[U_{y,2M}(+\infty)=U_{y,1}(-\infty).\]

Hence, we can define the concatenated cylinder $C_y:[0,1]\times S^1 \to \Sigma$ by
\begin{equation*}
C_y(s,t)=
\begin{cases}
		\mathfrak{S}^{Ms}(y)(t) &\text{for } 2Ms\in 2\ZZ\\
		\mathfrak{R}^{1/2 + Ms}(y)(t) &\text{for } 2Ms\in 2\ZZ+1\\
		U_{y,[2Ms]}(\mathfrak{f}(2Ms-[2Ms]),\,t) &\text{for } 2Ms\notin \ZZ
	\end{cases}.
\end{equation*}
Since $C_y(0,t)= y(t) = C_y(1, t)$, the map $C_y$ defines a torus map
\[T_y: S^1_s \times S^1_t \to \Sigma,\]
obtained by identifying the two boundaries of $[0,1]\times S^1$ in the obvious way.
We denote by $\widetilde{T}_y$ the graph of the torus $T_y$. Note that the map $C_\gamma$ is continuous on $[0,1]\times S^1$ and smooth for $(s,t)$ with $2Ms \notin \ZZ$. Therefore, the map $T_\gamma$ is also continuous. 
\begin{figure}
    \centering
    \includegraphics[width=.6\linewidth]{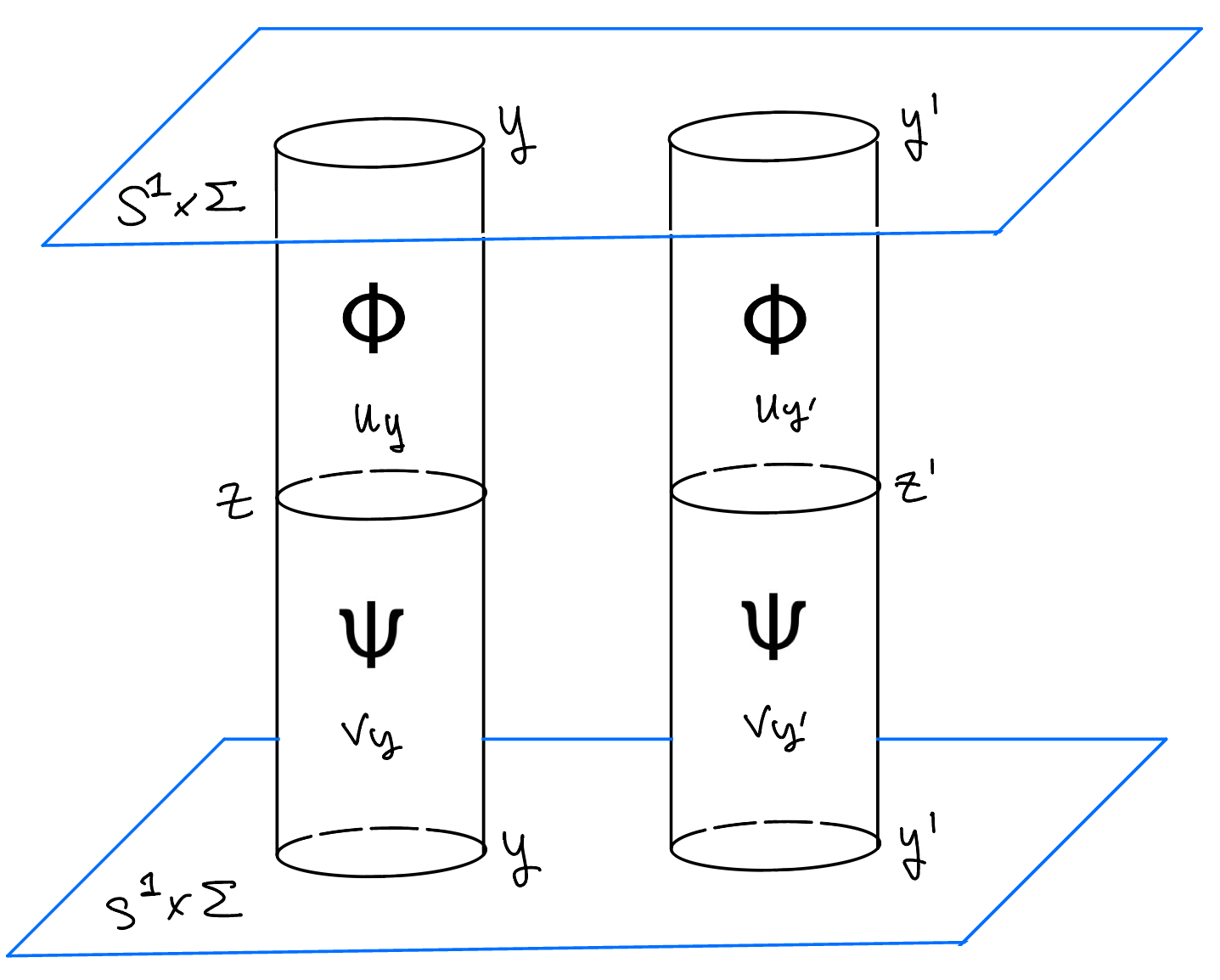}
    \caption{Tori of enhanced Floer cylinders}
    \label{fig:Torus of enhanced Floer cylinders}
\end{figure}
By the construction of the tori $T_y$, the algebraic intersection number $\widetilde{T}_y \cdot \widetilde{T}_{y^\prime}$ is given by
\begin{equation}
	\label{Equation: Sum of intersection points}
	\widetilde{T}_y \cdot \widetilde{T}_{y^\prime}= \sum_{0 \le m \le M-1} m_u(\mathfrak{S}^m(y), \mathfrak{S}^m(y^\prime))+m_v(\mathfrak{S}^m(y), \mathfrak{S}^m(y^\prime)),
\end{equation}
for any pair $y\neq y' \in \mathcal{Y}$.

By Lemma \ref{Lemma: Positivity of intersection between graph of enhanced cylinders}, each $m_u$ and $m_v$ is non-negative. Hence, for any pair $y\neq y' \in \mathcal{Y}$, we must have
\begin{equation}
m_u(y,y^\prime)=0
\end{equation}
if 
\[[\widetilde{T}_y] \cdot [\widetilde{T}_{y^\prime}]=0.\]

For $y\in \mathcal{Y}$, define the loop $\gamma_y:S^1 \to \Sigma$ by
\[\gamma_{y}(\tau)=T_y(\tau,0).\]
Then, the homology class of $[\widetilde{T}_y]$ is given by
\begin{equation} \label{eq:decomposition-via-künneth}
[\widetilde{T}_y] = [S^1_a \times S^1_b]\otimes[pt] + \big( [S^1_a]\otimes [y]- [S^1_b]\otimes[\gamma_y] \big)   + [pt]\otimes (T_y)_*[S^1_a \times S^1_b],
\end{equation}
where $[\widetilde{T}_y]\in H_2(S^1_a \times S^1_b \times \Sigma)$ is expressed using the decomposition:
\begin{align*}
	H_2(S^1_a \times S^1_b \times \Sigma) =\,&H_2(S^1_a \times S^1_b) \otimes H_0(\Sigma) \oplus H_1(S^1_a \times S^1_b) \otimes H_1(\Sigma) \\
	\oplus\, & H_0(S^1_a \times S^1_b) \otimes H_2(\Sigma),
\end{align*}
which follows from the Künneth theorem.

We will use this decomposition to compute $[\widetilde{T}_y] \cdot [\widetilde{T}_{y^\prime}]$. For this, we recall the formula for the intersection pairing on tensor products.
Given \linebreak $\alpha \in H_i(S^1_a \times S^1_b)$, $\alpha' \in H_{2-i}(S^1_a \times S^1_b)$, $\beta \in H_{2-i}(\Sigma)$ and $\beta' \in H_{i}(\Sigma)$ we have
\[
(\alpha \otimes \beta)\cdot(\alpha' \otimes \beta')
=
(-1)^{i}
(\alpha\cdot \alpha')(\beta\cdot \beta').
\]
Moreover,
\[(\alpha \otimes \beta)\cdot(\alpha' \otimes \beta')
= 0, \]
unless
\[
\deg(\alpha)+\deg(\alpha')=2,
\qquad
\deg(\beta)+\deg(\beta')=2.
\]
Using these formulas and equation \eqref{eq:decomposition-via-künneth}, we obtain 
\begin{align} \label{eq:intersection-first-approximation}
[\widetilde{T}_y] \cdot [\widetilde{T}_{y^\prime}] & = \mathrm{deg}(T_y) + \mathrm{deg}({T}_{y^\prime}) - (  [y'] \circ [\gamma_y]  +    [y] \circ [\gamma_{y'}] ) \\
& \nonumber = \mathrm{deg}(T_y) + \mathrm{deg}({T}_{y^\prime}) + ([\gamma_y] \circ  [y'] + [\gamma_{y'}] \circ  [y]),
\end{align}
where $\circ$ denotes the intersection product $\circ: H_1(\Sigma)\otimes H_1(\Sigma)\to \ZZ$, and $\mathrm{deg}(T_y)$ and $ \mathrm{deg}({T}_{y^\prime})$ denote the degrees of the maps $T_y$ and $T_{y'}$.

We introduce the following terminology, which will be very useful later on. Given a map $T: S^1_a \times S^1_b \to \Sigma $, we define $\mathrm{Area}(T):=\int_{S^1_a \times S^1_b}T^* \omega$. Because $\int_{\Sigma} \omega=1$, and $\Sigma$ is oriented by the symplectic form $\omega$, it follows that $[\omega]$ is the  fundamental cohomology class of $\Sigma$ for the de-Rham cohomology. Therefore, we obtain
\begin{align} \label{eq:area-of-Ty}
 & \mathrm{deg}(T_y) = \int_{S^1_a \times S^1_b} (T_y)^*\omega = \mathrm{Area}(T_y), \\
\nonumber & \mathrm{deg}(T_{y'}) = \int_{S^1_a \times S^1_b} (T_{y'})^*\omega = \mathrm{Area}(T_{y'}).
\end{align}

Since the free homotopy class of $y$ is $\alpha$, the homology class $[y]$ satisfies \[[y]=[\alpha]\in H_1(\Sigma, \ZZ).\]
In particular, \[ [y]=[y^\prime]\quad \text{ for any } y,y^\prime \in \mathcal{Y}.\]
We thus have an equality
\begin{equation} \label{eq:swappin-y-and-y'}
[\gamma_y] \circ  [y'] + [\gamma_{y'}] \circ  [y]= [\gamma_y] \circ  [y] + [\gamma_{y'}] \circ  [y'].
\end{equation}

Substituting \eqref{eq:area-of-Ty} and \eqref{eq:swappin-y-and-y'} in the right side of \eqref{eq:intersection-first-approximation} we obtain
\begin{equation} \label{eq:intersection-final-form}
    [\widetilde{T}_y] \cdot 	[\widetilde{T}_{y^\prime}] = \Area(T_y) + \Area(T_{y^\prime}) + [\gamma_y] \circ  [y] + [\gamma_{y'}] \circ  [y'].
\end{equation}
To complete the proof, we show that the right-hand side of \eqref{eq:intersection-final-form} vanish by considering the possible cases. This implies $m_u(y,y^\prime)=0,$
and hence
\[
    \tilde{u}_y\cap\tilde{u}_{y^\prime}=\emptyset
\]
for any distinct $y,y^\prime\in\mathcal{Y}$, completing the proof.

\textit{Case 1: $\Sigma \ne T^2$}.\\
In this case, $\Area(T_y)$ in fact vanishes for every $y\in\mathcal{Y}$. This follows from
\[
\Area(T_y)
=\mathcal{A}_H([y,w])-\mathcal{A}_H([y,C_y\# w])
\]
and the fact that there is at most one class of capped loops for each CZ--index. Note that, since we count only $0$-dimensional moduli spaces, indices does not change along the tori.

Next $[\gamma_y] \circ  [y] =0$ for any $y\in \mathcal{Y}$. For this, let $y\in \mathcal{Y}$.
\begin{enumerate}
    \item $\Sigma = S^2$. Since $H_1(S^2)=0$, we have that  $[y] = [\gamma_y] = 0$, so that $[\gamma_y] \circ  [y] = 0$.
    \item $\Sigma$ has genus $\geq 2$. In this case, we know that any abelian subgroup of $\pi_1(\Sigma)$ is either trivial or isomorphic to $\mathbb{Z}$. Since $\pi_1(T^2)$ is abelian, it follows that $f_*(\pi_1(T^2))$ is either $0$ or isomorphic to $\Z$. Using the fact that $H_1(\Sigma)$ is the abelianization of $\pi_1(\Sigma)$, it follows that $f_*(H_1(T^2))$ is either $0$ or isomorphic to $\Z$: in any case $f_*(H_1(T^2))$ is cyclic. Now, since $[y]$ and $ [\gamma_y]$ are both elements of the cyclic  group $f_*(H_1(T^2))$, there exists $\vartheta \in f_*(H_1(T^2))$ and integers $n,m$ such that $[y]=n\vartheta$ and $[\gamma_y]=m\vartheta$, which implies that \linebreak $[\gamma_y] \circ  [y] = n\vartheta \ \circ \ m\vartheta=0$, since the intersection product \linebreak $\circ : H_1(\Sigma) \otimes H_1(\Sigma) \to \Z$ is antisymmetric.
\end{enumerate}

\textit{Case 2: $\Sigma = T^2$}.\\
In this case,
    \[
    T_y:S^1_a\times S^1_b\longrightarrow T^2
    \]
    is a map between oriented tori. Recall that the orientation of $S^1_a\times S^1_b$ is the one for which  $\{[S^1_a],[S^1_b]\}$ is an oriented basis, and the orientaion of $T^2$ is determined by the symplectic form $\omega$. Let
    \[
    (T_y)_*:H_1(S^1_a\times S^1_b)\longrightarrow H_1(T^2)
    \]
    be the induced map on homology. Choosing oriented bases for the homology groups of the domain and target, let \(A\) be the matrix representing \((T_y)_*\). The degree of a map between oriented tori is given by the determinant of the induced map on first homology: indeed $\deg(T_y)=\det(A) = \mathrm{Area}(T_y)$.
    With respect to the basis $\{[S^1_a],[S^1_b]\}$
    for \(H_1(S^1_a\times S^1_b)\), the columns of \(A\) are the coordinates of
    \[
    (T_y)_*[S^1_a]=[y],
    \qquad
    (T_y)_*[S^1_b]=[\gamma_y].
    \]
   Moreover, the algebraic intersection pairing on \(H_1(T^2;\mathbb{Z})\) is represented, with respect to any oriented basis, by the matrix
    \[
    \begin{pmatrix}
    0&1\\
    -1&0
    \end{pmatrix}.
    \] 
    Hence, if two homology classes are written as column vectors in such a basis, their intersection number is given by the determinant of the matrix formed by these two columns. Therefore,
    \[
    [y]\circ[\gamma_y]=\det(A).
    \]
    Combining the two equalities above, we obtain
    \[
    [y]\circ[\gamma_y]=\deg(T_y) = \mathrm{Area}(T_y),
    \]
    Hence, the right-hand side of \eqref{eq:intersection-final-form} vanishes, and completes the proof.
\end{proof}

\subsubsection{Variation of Theorem \ref{Theorem: Braid Stability under PSS-image spectral distance}}
In this section, we prove a variation of Theorem \ref{Theorem: Braid Stability under PSS-image spectral distance} which relaxes the assumption that the Hamiltonain $H$ be non-degenerate. Instead, we only assume that the subset $\mathcal{Y} \in \mathcal{P}(H;\alpha)$ consists of non-degenerate $1$-periodic orbits.

\begin{thm}
	\label{Theorem: Variation of Braid Stability under PSS-image spectral distance}
	Let $(\Sigma,\omega)$ be a closed symplectic surface,  let $\alpha\in\tilde{\pi}_1(\Sigma)$ be a free homotopy class of loops in $\Sigma$, and $H:S^1 \times \Sigma \to \R$ be a Hamiltonian. Suppose $\mathcal{Y}$ is a finite subset of non-degenerate $1$-periodic orbits in $\mathcal{P}(H;\alpha)$. Then, for any non-degenerate Hamiltonian $K$ satisfying the inequality
	\[\gamma_\im(\overbar{H}\#K)<\Delta(\mathcal{Y},\phi^1_H),\]
	there exists a subset $\mathcal{Z} \subset \mathcal{P}(K;\alpha)$ such that the braids
	\[\BB(\mathcal{Y},H)\quad\text{and}\quad \BB(\mathcal{Z};K)\]
	are freely isotopic as braids.
\end{thm}

\begin{proof}
If $H$ is non-degenerate, then the result was already established in Theorem \ref{Theorem: Braid Stability under PSS-image spectral distance}. 
We thus consider the case when $H$ is degenerate.

It is well-known that non-degenerate Hamiltonians are $C^\infty$-dense. Thus, there exist arbitrary  perturbations $G$ such that $H \# G$ is non-degenerate. Since every fixed point in $\mathcal{Y}$ is already non-degenerate, we can further assume that $G$ is compactly supported in the complement $(S^1 \times \Sigma) \setminus \BB(\mathcal{Y},H)$. For such a perturbation $G$, the set $\mathcal{Y}$ remains a subset of $\mathcal{P}(H \# G; \alpha)$, and the braids $\BB(\mathcal{Y},H)$ and $\BB(\mathcal{Y},H \# G)$ remain identical.

Let $\delta>0$ be small enough so that
\begin{itemize}
    \item $\Delta(\mathcal{Y},\phi^1_{H})>  \gamma_{\rm im}(\overline{H}\#K) + 4\delta.$
\end{itemize}
We claim that if $G$ is a small enough perturbation in the $C^\infty$-topology, then the following hold:
\begin{itemize}
    \item[i)] $\Delta(\mathcal{Y},\phi^1_{H \# G}) > \Delta(\mathcal{Y},\phi^1_{H}) - \delta$,
    \item[ii)]  $\gamma_\im(\overline{H \# G} \# K) < \gamma_\im(\overline{H} \# K) + \delta$.
\end{itemize}
Claim i) follows from Proposition \ref{Proposition: Stabillity of quasi-isolatedness}, while claim ii) follows from the $C^\infty$-continuity  of $\gamma_{\rm im}$.

In then follows that 
\begin{equation}
    \gamma_\im(\overline{H \# G} \# K) < \Delta(\mathcal{Y},\phi^1_{H \# G}).
\end{equation}
	
	
We can thus apply Theorem \ref{Theorem: Braid Stability under PSS-image spectral distance} to $H\#G$ and $K$.
It follows that there exists a finite subset $\mathcal{Z} \subset \mathcal{P}(K;\alpha)$ such that the braids
\[\BB(\mathcal{Y},H\#G)\quad \text{and} \quad \BB(\mathcal{Z};K)\]
are freely isotopic as braids. Since the braid $\BB(\mathcal{Y},H\#G)$ is identical to $\BB(\mathcal{Y},H)$, it follows that there exists a $\mathcal{Z} \subset \mathcal{P}(K;\alpha)$ such that $\BB(\mathcal{Y},H)$ and $\BB(\mathcal{Z};K)$ are freely isotopic. This completes the proof.	
\end{proof}

\subsection{Proof of main Theorems}
\label{Section: Proof of main Theorem subsection}
In this subsection, we prove the main results of the paper which were stated from Section~\ref{sec:mainresults}.

\subsubsection{Proof of Theorem \ref{Main Theorem: Braid Stability under PSS-image spectral distance}} 
\label{Section: Proof of Theorem A}
\begin{thm}[= Theorem \ref{Main Theorem: Braid Stability under PSS-image spectral distance}]
	\label{Theorem: Braid type Stability under PSS-image spectral distance}
	Let $(\Sigma,\omega)$ be a closed symplectic surface. Let $\phi$ be a Hamiltonian diffeomorphism of $(\Sigma, \omega)$, $H$ a Hamiltonian generating $\phi$,  and let $\mathcal{Y}$ be a finite collection of non-degenerate fixed points of $\phi$ in the same free homotopy class of loops $\alpha \in \tilde{\pi}_1(\Sigma)$. Then, for any non-degenerate Hamiltonian diffeomorphism $\phi^\prime$ satisfying \[d_{\im}(\phi,\phi^\prime)<\Delta (\mathcal{Y},\phi),\]
	there exist
    \begin{itemize}
		\item a Hamiltonian $K$ generating $\phi^{\prime}$,
		\item and a finite collection $\mathcal{Z}$ of fixed points of $\phi^\prime$
	\end{itemize}
such that $\BB(\mathcal{Z},K)$ is freely isotopic as a braid to $\BB(\mathcal{Y},H)$.
\end{thm}
\begin{proof}[Proof of Theorem \ref{Main Theorem: Braid Stability under PSS-image spectral distance}]
	By the definition of the PSS-image spectral distance, there exist Hamiltonians $H$ and $K$ on $(\Sigma,\omega)$ generating $\phi$ and $\phi^\prime$, respectively, such that
	\[d_\im(\phi, \phi^\prime) = \gamma_\im(\overbar{H}\# K).\]
	
	We again denote by $\mathcal{Y}$ the set of $1$-periodic orbits of $H$ corresponding to the fixed points in $\mathcal{Y}$. 
    By assumption, 
	\[\gamma_\im(\overbar{H}\#K)<\Delta (\mathcal{Y},\phi).\]
	
    By applying Theorem \ref{Theorem: Variation of Braid Stability under PSS-image spectral distance}, there exists a subset $\mathcal{Z} \subset \mathcal{P}(K;\alpha)$ such that the braids
	$\BB(\mathcal{Y},H)$ and $\BB(\mathcal{Z};K)$ are freely isotopic. 
\end{proof}

\subsubsection{Proof of Theorem \ref{Main Theorem: Lower semicontinuity of topological entropy}}
\label{Subsubsection: Proof of Theorem C}

\

In order to deduce Theorem \ref{Main Theorem: Lower semicontinuity of topological entropy}  from Theorem \ref{Main Theorem: Braid Stability under PSS-image spectral distance} we recall the notion of braid type as used in \cite{Birman,Boyland}. 

For a finite subset $\mathcal{Y} \subset \Sigma$, let $\mathrm{Homeo}_0(\Sigma, \mathcal{Y})$ denote the group of homeomorphism $f:\Sigma \to \Sigma$ isotopic to identity and satisfy $f(\mathcal{Y})=\mathcal{Y}$. In particular, the points in $\mathcal{Y}$ are periodic points, but not necessarily fixed points.

Two homeomorphisms $f,g \in \mathrm{Homeo}_0(\Sigma, \mathcal{Y})$ are said to be isotopic relative to $\mathcal{Y}$ if there exists an isotopy within $\mathrm{Homeo}_0(\Sigma, \mathcal{Y})$. The relative isotopy class of $f$ is denoted by $[f]$, and we define the group of isotopy classes as $\mathrm{Iso}_0(\Sigma, \mathcal{Y})$. This group inherits a group structure given by the isotopy class of the composition of any two representatives.

For a given homeomorphism $f$ and a finite set $\mathcal{Y}$ with $f(\mathcal{Y})=\mathcal{Y}$, we define the braid type of $(\mathcal{Y},f)$, denoted by $\mathcal{B}(\mathcal{Y}, f)$, as the conjugacy class of $[f]$. In particular,
\[\mathcal{B}(\mathcal{Y},f) \in \mathrm{Iso}_0(\Sigma, \mathcal{Y})/\mathrm{Conj}.\]

We say two braid types $\BB(\mathcal{Y}, f)$ and $\BB(\mathcal{Z},g )$ have the same braid type if there exists a homeomorphism $h:\Sigma\to\Sigma$ isotopic to identity such that
\begin{itemize}
	\item $h(\mathcal{Y})=\mathcal{Z}$ and
	\item $[f]=[h^{-1}\circ g \circ h]$ as elements in $\mathrm{Iso}_0(\Sigma, \mathcal{Y})/\mathrm{Conj}$.
\end{itemize} 
We denote the equivalence class of the braid type $\BB(\mathcal{Y}, f)$ by $[\BB(\mathcal{Y}, f)]$.

We now review the relationship between braid types and (geometric) braids. For further details and discussion on their relations, we refer to \cite{Birman,Boyland}.

For a given pair $(\mathcal{Y}, f)$, consider any path of homeomorphism $\set{f_t}_{t \in [0,1]}$ connecting $f_0 = \id_\Sigma$ and $f_1 = f$. This path defines a braid based at $\mathcal{Y}$ by
\[\BB(\mathcal{Y},\set{f_t}):=\bigcup_{t\in S^1} \set{t}\times f_t(\mathcal{Y}).\]
Conversely, for a given braid $\BB$ based at $\mathcal{Y}$, there exists a path of homeomorphisms $\set{f_t}_{t \in [0,1]}$ such that $f_0=\id$ and $\BB = \BB(\mathcal{Y},\set{f_t})$.

The braid type $\mathcal{B}(\mathcal{Y}, f_1)$ does not depend on the choice of the path $\set{f_t}$. Moreover, if two braids $\BB_1$ and $\BB_2$ are freely isotopic as braids, then their associated braid types are equivalent. 

\begin{rem}
	When $\BB$ is smooth, the existence of a path $\set{f_t}_{t \in [0,1]}$ as above can be established by constructing explicit time-dependent vector fields $\set{X_t}_{t\in [0,1]}$ supported near a neighborhood of $\BB(t) \cap (\set{t}\times \Sigma)$.
	
	The braid type does not depend on the choice of time-dependent vector fields. Indeed, for two vector fields $\set{X_t}, \set{Y_t}$, the linear interpolation between them gives a relative isotopy between their time $1$-maps.
	
	A smooth free isotopy between braids can be used to define a time-dependent vector field, and its time $1$-map guarantees that the two braid types are equivalent.
\end{rem}

Summarizing the discussion above, a pair $(\mathcal{Y}, H)$ defines a braid type $\mathcal{B}(\mathcal{Y}, \phi_H^1)$. When two braids $\BB(\mathcal{Y},H), \BB(\mathcal{Z};K)$ are freely isotopic as braids, their associated braid types are equivalent. This discussion combined with Theorem \ref{Theorem: Braid type Stability under PSS-image spectral distance} implies the following corollary.
\begin{cor} \label{Corollary: Braid type Stability under PSS-image spectral distance}
Let $(\Sigma,\omega)$ be a closed symplectic surface. Let $\phi$ be a Hamiltonian diffeomorphism of $(\Sigma, \omega)$, 
and let $\mathcal{Y}$ be a finite collection of non-degenerate fixed points of $\phi$ in the same free homotopy class of loops $\alpha \in \tilde{\pi}_1(\Sigma)$. Then, for any non-degenerate Hamiltonian diffeomorphism $\phi^\prime$ satisfying \[d_{\im}(\phi,\phi^\prime)<\Delta (\mathcal{Y},\phi),\]
	there exist
    \begin{itemize}
		\item a finite collection $\mathcal{Z}$ of fixed points of $\phi^\prime$
	\end{itemize}
such that $\BB(\mathcal{Z},\phi')$ has the same braid type as $\BB(\mathcal{Y},\phi)$.    
\end{cor}

We are now ready to  prove
\begin{thm}[=Theorem \ref{Main Theorem: Lower semicontinuity of topological entropy}]
	\label{Theorem: Lower semicontinuity of topological entropy}
	Let $(\Sigma,\omega)$ be a closed symplectic surface. Then, the topological entropy
	\[h_{\top} : (\Ham(\Sigma, \omega), d_\im) \to [0,\infty)\]
	is lower semicontinuous.
\end{thm}
The proof follows the argument given in \cite[Theorem 2.6]{Alves-Meiwes24}, where the first two authors establish the lower semicontinuity of topological entropy with respect to the Hofer metric. To be more specific, we consider the exponential growth rate associated with braid types and approximate the topological entropy using these growth rates.

Given a braid type $\BB(\mathcal{Y}, f)$, we let $\Gamma_{\pi_1}(\BB(\mathcal{Y}, f))$ be the exponential growth rate of the induced action of $f$ on the fundamental group of $\Sigma \setminus \mathcal{Y}$:
\[\Gamma_{\pi_1}(\BB(\mathcal{Y}, f)):=\sup_{g\in\pi_1(\Sigma\setminus \mathcal{Y}, x_0)}\limsup_{n\to \infty} \frac{\log(\ell_S(f^n_*(g)))}{n},\]
where $f_*$ is the automorphism of $\pi_1(\Sigma \setminus \mathcal{Y},x_0)$ with respect to some basepoint $x_0 \in \Sigma \setminus \mathcal{Y}$ and a path $\sigma$ from $x_0$ to $f(x_0)$. Here $S$ is a set of generators, and $\ell_S(h)$ is the minimal length of a word in $S$ and $S^{-1}$ that is needed to represent $h$.

The definition is independent of the choices made, and moreover gives an invariant for the equivalence class of braid types, i.e. the exponential growth rate $\Gamma_{\pi_1}([\BB(\mathcal{Y}, f)])$ is well-defined. Furthermore, it follows from elementary properties of $f_*$, that for all $k\in \N$, 
\[\Gamma_{\pi_1}([\BB(\mathcal{Y}, f^k)]) = k \cdot \Gamma_{\pi_1}([\BB(\mathcal{Y}, f)]).\]

By an inequality of Manning (see \cite{Bowen} 
for a proof in the present setting of a punctured surface and $f$ differentiable)
\begin{equation}
	\label{Equation: topological entropy of braid vs homeomorphism}
	\Gamma_{\pi_1}([\BB(\mathcal{Y}, f)]) \le h_{\top}(f).
\end{equation}

We define the topological entropy of the braid type $[\BB(\mathcal{Y}, f)]$ as
\[h_{\top}([\BB(\mathcal{Y},f]) =\inf_{g} h_{\top}(g),\]
where the infimum runs over all $g$ which has an invariant set $\mathcal{Z}$ such that $[(\mathcal{Y}, f)]=[(\mathcal{Z}, g)]$. It follows from the Nielsen-Thurston classification theory of surface diffeomorphisms, that:
\[	\Gamma_{\pi_1}([\BB(\mathcal{Y}, f)]) = h_{\top}([\BB(\mathcal{Y},f]).\]

Denote by $\mathrm{BT}(f)$ the set of all braid types of $f$:
\[\mathrm{BT}(f):=\set{[\BB(\mathcal{Y},f)]:f(\mathcal{Y})=\mathcal{Y}, \#\mathcal{Y}<+\infty}.\]
When $f:\Sigma \to \Sigma$ is a diffeomorphism, the inequality
\[\sup_{[\BB(\mathcal{Y},f)]\in \mathrm{BT}(f)} \Gamma_{\pi_1}([\BB(\mathcal{Y}, f)]) \le h_{\top}(f).\]
follows from Equation \eqref{Equation: topological entropy of braid vs homeomorphism}.

\medskip

The following theorem states that this inequality is, in fact, an equality.

\begin{thm}[\hspace{1sp}{\cite[Theorem B.1]{Alves-Meiwes24}}]
	\label{Theorem: Topological entropy can be approximated by braids}
	Let $\Sigma$ be a compact surface, and let $f:\Sigma \to \Sigma$ be a diffeomorphism such that $h_{\top}(f)>0$. Then, for any $\epsilon>0$, there is a hyperbolic periodic point $\mathcal{P}$ of $f$ such that
	\[\Gamma_{\pi_1}([\BB(\mathcal{P}, f)]) > h_{\top}(f)-\epsilon.\]
\end{thm}

Combining this result  with Corollary \ref{Corollary: Braid type Stability under PSS-image spectral distance}, we now prove Theorem \ref{Main Theorem: Lower semicontinuity of topological entropy}.

\begin{proof}[Proof of Theorem \ref{Main Theorem: Lower semicontinuity of topological entropy}]
	Let $\phi$ be a Hamiltonian diffeomorphism of a symplectic surface $(\Sigma, \omega)$. For any $\epsilon>0$, we prove that there exists $\delta>0$ such that
	\[h_\top(\psi) > h_\top(\phi)-\epsilon\] for every Hamiltonian diffeomorphism $\psi$ with  $d_\im(\phi,\psi)<\delta$.
	
	If $h_{\top}(\phi)=0$, there is nothing to prove. So, assume that $h_{\top}(\phi)>0$. By Theorem \ref{Theorem: Topological entropy can be approximated by braids}, there exists a hyperbolic $k$-periodic orbit \linebreak $\mathcal{P}=\set{x, \phi(x),\ldots, \phi^{k-1}(x)}$ of $\phi$, for some $k\in \N$, such that
	\[\Gamma_{\pi_1}([\BB(\mathcal{P}, \phi)]) > h_{\top}(\phi)-\epsilon.\]
	In particular, each periodic point in $\mathcal{P}$ is non-degenerate.
      
	Hence, it is enough to prove the following claim.
	
	\textit{Claim: For any Hamiltonian diffeomorphism $\psi$ satisfying}
	\[d_\im({\phi,\psi})< \Delta(\mathcal{P},\phi^k)/k,\]
	\textit{we have}
	\begin{equation}
	\label{Equation: Lower semicontinuity}
		h_{\top}({\psi})>h_{\top}(\phi)-\epsilon.
	\end{equation}

	We first prove the claim under the assumption that $\psi$ is strongly \linebreak non-degenerate, i.e. $\psi^n$ is a non-degenerate Hamiltonian for any $n \in \N$. Applying Corollary~\ref{Corollary: Braid type Stability under PSS-image spectral distance} for $(\mathcal{P}, \phi^k)$, 
    there exists a set $\mathcal{Q}$ of fixed points of $\psi^k$, such that the braid types $\BB(\mathcal{P},\phi^k)$ and $\BB(\mathcal{Q}, \psi^k)$ are equivalent. Note that by Lemma~\ref{Lemma: Spectral distance of the iteration}, the $d_\im(\phi^k,\psi^k)$ satisfies
	\[d_\im(\phi^k,\psi^k)< \Delta(\mathcal{P},\phi^k).\]
	
    We conclude that
	\begin{align*}
		k\cdot h_{\top}(\psi) &= h_{\top}(\psi^k)\\
		&\ge \Gamma_{\pi_1}([\mathcal{Q}, \psi^k])
		= \Gamma_{\pi_1}([\mathcal{P}, \phi^k])
		= k\cdot \Gamma_{\pi_1}([\mathcal{P}, \phi])\\
		&> k\cdot (h_{\top}(\phi)-\epsilon).
	\end{align*}
	In particular, we have
	\[h_{\top}(\psi) > h_{\top}(\phi)-\epsilon\]
	for every strongly non-degenerate $\psi$ satisfying $d_\im(\phi,\psi)<\Delta(\mathcal{P},{\phi^k})/k$.

	Note that strongly non-degenerate Hamiltonian diffeomorphisms are $C^\infty$-dense in $\Ham(\Sigma,\omega)$. Since both the $d_\im$-distance and the topological entropy $h_\top$ are continuous with respect to $C^\infty$-topology, the claim holds for any Hamiltonian diffeomorphism $\psi$ satisfying $d_\im(\phi,\psi)<\Delta(\mathcal{P},\phi^k)/k$. This completes the proof.
\end{proof}

\subsubsection{Proof of Theorem \ref{Main theorem: Topological entropy after perturbation supported on a disk}}

\


\begin{thm}[=Theorem \ref{Main theorem: Topological entropy after perturbation supported on a disk}]
\label{Theorem: Topological entropy after perturbation supported on a disk}
	Let $(\Sigma,\omega)$ be a symplectic surface. For every Hamiltonian diffeomorphism $\phi$ and every $\varepsilon>0$, there exists a real number $A(\phi,\varepsilon)>0$ such that 
	\[h_{\top}(\psi \circ \phi)>h_{\top}(\phi)-\varepsilon\]
	for every area preserving diffeomorphism $\psi$ whose support is contained in a disjoint union $\sqcup^{n}_{i=1}D_{i}$ of open disks $D_i\subset \Sigma$ with area  $ <A(\phi,\varepsilon)$.
\end{thm}
\begin{proof}
	By Theorem \ref{Theorem: Topological entropy can be approximated by braids}, there exists a hyperbolic periodic point $x$ of $\phi$,
	such that \[\Gamma_{\pi_1}([\BB(\mathcal{P}, \phi)]) > h_{\top}(\phi)-\epsilon,\]
	for $\mathcal{P}=\set{x, \phi(x),\ldots, \phi^{k-1}(x)}$. We define the area bound \[A(\phi,\varepsilon):=\Delta(\mathcal{P},{\phi^k})/{2k}.\]
	
Let $\sqcup^{n}_{i=1}D_{i}$ be a union of disjoint open disks $D_i$, each with area at most $A(\phi, \varepsilon)$. By Corollary \ref{Corollary: Symp(D2) is contractible}, every area preserving diffeomorphism $\psi$ compactly supported on $\sqcup^{n}_{i=1}D_{i}$ is a Hamiltonian diffeomorphism. Moreover, there exists a Hamiltonian $G$ compactly supported on $\sqcup^{n}_{i=1}D_{i}$ that generates~$\psi$.

Then, as in the proof of Theorem \ref{Main Theorem: Lower semicontinuity of topological entropy}, it will follow that 
\[h_{\top}(\psi \circ \phi)>h_{\top}(\phi)-\epsilon\]
if we can show that
\begin{equation}\label{eq:lower bound to show}
d_\im(\psi \circ \phi, \phi) <  2A(\phi, \varepsilon) = \Delta(\mathcal{P},\phi^k)/{k}.
\end{equation}

We proceed to show that \eqref{eq:lower bound to show} holds. To prove that, we consider two distinct cases: $\Sigma =S^2$ and $\Sigma\neq S^2$. 


	\textit{Case 1: $\Sigma=S^2$}.\\
Note that by definition, the action gap $\Delta(\mathcal{P},\phi^k)$ is at most $1$, which means that
	\[A(\phi, \varepsilon) \le 1/2.\]
    Let us first assume that $n=1$, i.e., we consider only one disk $D$.   Since any open disk in $S^2$ with area at most $1/2$ is Hamiltonian displaceable, 
Corollary \ref{Corollary: Spectral norm supported on the displacable open set} implies that in this case 
	\[d_\im(\psi \circ \phi, \phi) = \gamma_{\im}(\psi) \le 2 \cdot E_{\mathrm{conn}}(D) = 2\mathrm{Area}(D) < 2A(\phi,\varepsilon).\]

 Let us now assume that $n\geq 2$. Let $\delta>0$ be so small that $$\sum_{i=1}^{n} (\mathrm{Area}(D_i) + \delta ) < 1,$$
 and $(\mathrm{Area}(D_i) + \delta )< A(\phi,\varepsilon) \leq \frac{1}{2}$, for all $1 \leq i \leq n$.

Identify  $S^2 = \{x^2 +y^2 + z^2 = 1\} \subset \R^3$ equipped with the standard (normalized) area form with area $1$. We 
let $$S_j = S^2 \cap \{x=\sin \theta, y=\cos\theta, a_{j-1} \leq \theta < a_j, -1<z<1\}, $$
where $$a_j = 2\pi \sum^{j}_{l=1} (\mathrm{Area}(D_l)+\delta).$$

\medskip
Note that the area of each $S_j$ is equal to $\mathrm{Area}(D_j) + \delta$. 
By a theorem of Moser there is an $\omega$-preserving diffeomorphism $\tau_1:S_1\to S_1$ with $\tau_1(D_1) \subset S_1$. Choose a small closed neighborhood $K_1$ of $\tau_1(D_1)$ that does not intersect the sets $\tau_1(D_2), \cdots \tau_1(D_n)$. Again, there is a $\omega$-preserving diffeomorphism $\tau_2$ supported in the complement of $K_1$ with $\tau_2(\tau_1(D_2))\subset S_2$. Repeating that argument, we find $\omega$-preserving diffeomorphisms $\tau_3, \ldots, \tau_n$ with $\tau_n\circ \cdots \circ \tau_1(D_j) \subset S_j$ for all $j=1,\ldots, n$. Set $$\Phi := \tau_n\circ \cdots \circ \tau_1.$$ 
Let $R:S^2 \to S^2$ denote the rotation by angle $\theta = \max_{ 1 \leq i \leq n} \{\mathrm{Area}(S_i) \}$. Since $$\mathrm{Area}(S_j)=\mathrm{Area}(D_j) + \delta< A(\phi,\varepsilon) \leq \frac{1}{2},$$
it follows that $R$ displaces each set $S_j$. Thus, $\Phi^{-1}R\Phi$ displaces every connected component of $\sqcup^{n}_{i=1}D_{i}$.
Moreover, $$\|R\|_{\mathrm{Hofer}} = \max_{ 1 \leq i \leq n} \{\mathrm{Area}(S_i) \} \leq A(\phi,\varepsilon)+\delta.$$ 
By Corollary~\ref{Corollary: Spectral norm supported on the displacable open set}, 
$$\gamma_{\mathrm{im}}(G) \leq 2 E_{\mathrm{conn}}(\sqcup^{n}_{i=1}D_{i})\leq 2(A(\phi,\varepsilon)+\delta).$$
Since this argument works for arbitrarily small $\delta>0$, we conclude that \eqref{eq:lower bound to show} holds, thus establishing the theorem in the case $\Sigma = S^2$. 

	
	\textit{Case 2: $\Sigma \ne S^2$.}\\
	We consider the case $n=1$: we thus have only one disk $D$.  Then, it is enough to show
    \begin{equation} \label{eq:thiswillfinishtheproof}
    c_{\im}([\Sigma],G) \le \Area(D).
    \end{equation}
    Indeed, it will follow directly from \eqref{eq:thiswillfinishtheproof} that 
	\[\gamma_{\im}(\psi) \le 2\cdot \Area(D) < 2A(\phi,\varepsilon).\] 
    Once the case for $n=1$ is established, the case for $n\geq 2$ follows then from the max inequality for the spectral invariant for surfaces of positive genus \cite{Buhovsky-Humiliere-Seyfaddini21}. 
    	
	To prove \eqref{eq:thiswillfinishtheproof}, we start by remarking that any Hamiltonian compactly supported on $D \subset \Sigma$ can be identified with a Hamiltonian supported on a standard open disk $\mathbb{D}(A) \subset \RR^2$ of area $A:=\Area(D)$. We add superscript $*$ to distinguish Hamiltonians defined on $\RR^2$.
	
	For $a \in (0,1)$ and a Hamiltonian $H^*$ on $\mathbb{R}^2$, we define $H_a^*$ by
	\[H^{*}_{a,t}(x,y)=a \cdot H^{*}_t(x/a, y/a)\quad \text{for any }(x,y)\in\RR^2.\]
	Then, a loop $z^*$ is a $1$-periodic orbit of $H^*$, if and only if $a\cdot z^*$ is a $1$-periodic orbit of $H^*_a$. 
    
    
	

	We denote by $H_a$ the Hamiltonian corresponding to $H^*_a$ and defined on $\Sigma$. Each $1$-periodic orbit $z$ of $X_{H_a}$ is either constant or contained in $D$. Hence, each $1$-periodic orbit is contractible. From a direct computation we deduce 
	\[\Spec(H_a)=\Spec(H_a;\cont) = a \cdot \Spec(H;\cont).\]	
	
		
	By the \textbf{Spectrality property} of PSS-image spectral invariant,
	\[c_\im([\Sigma];H_a)\in \Spec(H_a;\cont).\]
	Consider the function $c:(0,1] \to \RR$ defined by
	\[c(a) = c_\im([\Sigma];H_a) / a.\]
	By the \textbf{Hofer-Continuity property}, $c$ is a continuous function mapping into $\Spec(H;\cont)$, a nowhere dense subset of $\R$. Therefore, $c$ must be constant.
	
	Let $G$ be a Hamiltonian compactly supported in $D$ that generates $\psi$. Note that there exists another open disk $D^\prime$ such that
	\[\supp(G) \subset D^\prime \subsetneq D.\]
	
	The Hamiltonian $G_a$ is compactly supported on an open disk $D^\prime_a$ with \linebreak $\Area(D_a^\prime)=a\cdot \Area(D^\prime)$. This is immediate from the symplectic identification of $D \subset \Sigma$ and the standard disk $\D(A) \subset \RR^2$ of area $A$. In particular, for sufficiently small $a>0$, the disk $D_a^\prime$ is displaceable even in $D$. 
	It then follows from Lemma \ref{Lemma: Spectral invariant of [M] supported on the displacable open set} that
	\[c_\im([\Sigma];G_a) \le E(D^\prime_a),\]
	or equivalently
	\[c_\im([\Sigma];G) \le \Area(D^\prime).\]
	
	Applying the same argument to $c_\im([\Sigma],\overbar{G})$, we obtain 
    \[c_\im([\Sigma];\overline{G}) \le \Area(D^\prime).\]
    We thus conclude that 
	\[d_{\mathrm{im}}(\psi \circ \phi,\phi) = \gamma_\im(G) < 2 \cdot \Area(D) < 2A(\phi,\varepsilon).\]

	
    This completes the proof of the theorem in the case $\Sigma \neq S^2$.
\end{proof}

We note that the idea used to prove the second case is known as the
symplectic contraction principle, introduced in  \cite{Polterovich14}. This appears also in \cite{Humiliere-LeRoux-Seyfaddini16} to prove the Max property of the Oh--Schwarz spectral invariant, which includes the second case.


\newpage

\appendix
\section{Diffeomorphism groups on symplectic surfaces}
\label{app: Diffeomorphism groups on symplectic surfaces}

In this appendix, we recall and prove some topological properties of the group of symplectomorphisms $\Symp(\Sigma,\omega)$ and the group of Hamiltonian diffeomorphisms $\Ham(\Sigma,\omega)$ on a symplectic surface $(\Sigma, \omega)$.

For an oriented surface $\Sigma$, a $2$-form $\omega$ is a symplectic form if and only if it is nowhere vanishing. We always assume $\omega$ is positive with respect to the volume form induced by the given orientation.

For a compact oriented surface $\Sigma$, let $\Omega(\Sigma, A)$ denote the space of symplectic forms on $\Sigma$ satisfying $\Area(\Sigma,\omega)=A$, equipped with the $C^\infty$-topology. Two key observations for any $\omega_0,\omega_1 \in \Omega(\Sigma, A)$ are:
\begin{enumerate}
	\item $\omega_1-\omega_0$ is an exact $2$-form, and
	\item $\omega_t:=t\cdot\omega_1 + (1-t)\cdot\omega_0 \in \Omega(\Sigma, A)$ for any $t\in[0,1]$.
\end{enumerate}
By (2), the set $\Omega(\Sigma, A)$ is a convex subset of linear space $\Omega^2(\Sigma)$.

As a low-dimensional phenomenon, symplectic forms on a compact orientable surface are classified by their area.
\begin{lem}
\label{Lemma: symplectic forms on surfaces}
	Let $\Sigma$ be a compact orientable surface possibly with a boundary. For two symplectic forms $\omega_0, \omega_1$ on $\Sigma$, the followings are equivalent:
	\begin{enumerate}
		\item $\Area(\Sigma,\omega_0) = \Area(\Sigma,\omega_1)$
		\item There exists a diffeomorphism $\Phi: \Sigma \to \Sigma$ such that,
		\[\Phi^*\omega_1=\omega_0 \quad\text{and}\quad\Phi|_{\partial \Sigma} = \id.\] 
	\end{enumerate}
\end{lem}
This follows immediately from the Moser's trick \cite{Moser65} for closed surfaces and the work of Banyaga \cite{Banyaga74} for surfaces with boundary. The linear interpolation $\set{\omega_t}_{t\in[0,1]}$ generates a smooth path of diffeomorphisms $\set{\Phi_t}_{t\in[0,1]}$ such that $\Phi_t^*\omega_t=\omega_0$ for each $t\in [0,1]$, with $\Phi_0=\id$. If $\Sigma$ has a boundary, the diffeomorphisms $\Phi_t$ satisfy $\Phi_t|_{\partial \Sigma}=\id$. However, note that $\Phi_t$ is not necessarily the identity in a neighborhood of $\partial \Sigma$.

As a corollary, given a compact symplectic surface $(\Sigma, \omega_o)$ with $\Area(\Sigma,\omega_o)=1$, any other symplectic structure $\omega$ on $\Sigma$ is symplectomorphic to $(\Sigma, c \cdot \omega_0)$. Thus, up to rescaling by a constant, there is a unique symplectic form for each compact orientable surface. In the case of a disk, rather than rescaling, we take $(\mathbb{D}(A),\omega_{\std})$, the disk centered on the origin with area $A$, as a canonical representative.

Next, we consider symplectomorphisms on $(\Sigma, \omega)$. In this setting, symplectomorphisms are simply area-, and orientation-preserving diffeomorphisms. For non-closed $\Sigma$ such as the disk $\mathbb{D}^2$ or the plane $\RR^2$, we define:
\begin{itemize}[leftmargin=5.5mm]
	\item $\Symp_c(\Sigma,\omega)$: the group of compactly supported symplectomorphisms.
	\item $\Ham_c(\Sigma,\omega)$: the group of Hamiltonian diffeomorphism generated by copmactly supported Hamiltonians.
\end{itemize}
Here, compactly supported means that the maps or Hamiltonians have support contained in the interior $\Int(\Sigma)$. When $\Sigma$ is closed, every symplectomorphisms and every Hamiltonian diffeomorphisms are automatically compactly supported. For this reason, we may add subscript $c$ for diffeomorphism groups when $\Sigma$ is closed, to integrate the notations.

By Lemma \ref{Lemma: symplectic forms on surfaces}, there exists an isomorphism between two groups \[\Symp_c(\D, \omega_0)\quad \text{and} \quad \Symp_c(\D, \omega_1).\]
Hence, we only consider $(\mathbb{D}^2, \omega_\std)$ when $\Sigma=\mathbb{D}^2$.

We now review some topological properties associated morphism groups. More precisely, we consider the following groups with inclusions:
\[\Ham_c(\Sigma,\omega) \subset \Symp_{c,0}(\Sigma, \omega) \subset \Symp_c(\Sigma,\omega) \subset \Diff_c^+(\Sigma).\]
The proofs of the stated properties will be presented in the end of the appendix. 

As a low dimensional phenomenon, the inclusion $\Symp_c(\Sigma, \omega) \subset \Diff_c^+(\sigma, \omega)$ is a homotopy equivalence.
\begin{prop}
	\label{Proposition: Symp and Diff+ of symplectic surfaces}
	Let $(\Sigma,\omega)$ be a compact symplectic surface, possibly with non-emtpy boundary. Then, the group $\Symp_c\Sigma,\omega)$ (resp. $\Symp_{0,c}(\Sigma,\omega)$) is homotopy equivalent to the group $\Diff_c^+(\Sigma)$ (resp. $\Diff_{0,c}(\Sigma)$).
\end{prop}

We emphasize that this result does not extend to higher dimensions. For example, in the disk cotangent bundle $D^*S^2$, the square of Dehn-Seidel twist is isotopic to the identity map in $\Diff_c^+(D^*S^2)$ but not in $\Symp_c(D^*S^2, \omega_{\std})$. See \cite{Seidel97} for details and further discussion.

Even in dimension $2$, a symplectomorphism $\psi$ is not necessarily isotopic to the identity, such as the Dehn twist along a non-contractible loop on a closed surface of positive genus. Even when $\psi$ is isotopic to the identity, like translations on $T^2$, it may still fail to be a Hamiltonian diffeomorphism.

In contrast, when $\Sigma=S^2$ or $\mathbb{D}^2$, we obtain following corollaries.
\begin{cor}
	\label{Corollary: Symp(S2) is connected}
	The group $\Symp(S^2,\omega)$ is connected. Moreover, every symplectomorphism on $(S^2, \omega)$ is a Hamiltonian diffeomorphism, i.e.
	\[\Symp(S^2,\omega)=\Symp_0(S^2,\omega)=\Ham(S^2,\omega).\]
\end{cor}

\begin{cor}
	\label{Corollary: Symp(D2) is contractible}
	For any $A>0$, the group $\Symp_c(\mathbb{D}(A), \omega_{\std})$ is contractible. Moreover, every compactly supported symplectomorphism is a compactly supported Hamiltonian diffeomorphism, i.e.
	\[\Symp_c(\mathbb{D}(A),\omega_{\std})=\Symp_{0,c}(\mathbb{D}(A),\omega_{\std})=\Ham_c(\mathbb{D}^2,\omega).\]
\end{cor}

Corollary \ref{Corollary: Symp(D2) is contractible} leads to another result, which plays a crucial role in Theorem \ref{Main theorem: Topological entropy after perturbation supported on a disk}. Given an open $U \subset \Sigma$, we denote by $\Symp(\Sigma,\omega;U)$ the group of symplectomorphisms compactly supported in $U$ and by $\Ham(\Sigma,\omega;U)$ the group of Hamiltonian diffeomorphisms generated by Hamiltonians compactly supported in $U$. When $U$ is a open disk, these groups coincide.

\begin{prop}
	\label{Proposition: Symplectomorphism supported on a disk}
	Let $(\Sigma, \omega)$ be a symplectic surface possibly with boundary. For any open disk $D\subset \Sigma$, we have
	\[\Symp(\Sigma,\omega;D)=\Ham(\Sigma,\omega;D).\]
\end{prop}
This proposition also implies that the group $\Symp(\Sigma,\omega;D)$ is connected, which is not trivial. For example, if $U$ is a neighborhood of a non-contractible loop $\gamma$, there is no continuous path in $\Symp(\Sigma,\omega;U)$ from the identity map to the Dehn twist along $\gamma$.

We conclude by recalling the fundamental groups of diffeomorphism groups, based at the identity map.
\begin{prop}
	\label{Proposition: Fundamental group of Hamiltonian diffeomorphism groups}
	For a closed symplectic surface $(\Sigma, \omega)$, the fundamental group of $\Ham(\Sigma,\omega)$ is given by
	\[\pi_1\left(\Ham(\Sigma,\omega), \id_\Sigma \right)=
	\begin{cases}
	\ZZ_2, & \Sigma=S^2\\
		\set{e}, & \Sigma \ne S^2
	\end{cases}.\]
	Moreover, the fundamental group of $\Symp_{0}(\Sigma,\omega)$ is given by
	\[\pi_1\left(\Symp_0(\Sigma,\omega), \id_\Sigma \right)=
	\begin{cases}
		\ZZ_2,		& \Sigma=S^2\\
		\ZZ^2			& \Sigma=T^2\\
		\set{e},		& \Sigma \ne S^2, T^2
	\end{cases}.\]
\end{prop}
	Note that, by Proposition \ref{Proposition: Symp and Diff+ of symplectic surfaces}, the fundamental group of $\Diff_0(\Sigma, \omega)$ is equal to $\pi_1\left(\Symp_0(\Sigma,\omega), \id_\Sigma \right)$.
	For the disk, all of the following diffeomorphism groups
	\[\Ham_c(\mathbb{D}(A),\omega_\std), (\Symp_{0,c}(\mathbb{D}(A),\omega_\std), \text{ and } \Diff_{0,c}(\mathbb{D}(A)\]
	are contractible following by Proposition \ref{Proposition: Symp and Diff+ of symplectic surfaces} and Corollary \ref{Corollary: Symp(D2) is contractible}. This implies their fundamental groups are trivial.


\medskip

\noindent \textbf{Proof for Appendix.} Most of the proofs can be found in \cite[Chapter 7.2]{leonidsbook}. For completeness, we include detailed arguments, particularly for the cases where $\Sigma$ has non-empty boundary. These cases requires slight adaptations of the proofs in \cite{leonidsbook}.

\begin{proof}[Proof of Proposition 	\ref{Proposition: Symp and Diff+ of symplectic surfaces}]
	Denote by $\Omega(\Sigma, \omega)$ the set of symplectic forms on $\Sigma$ with symplectic area $A:=\Area(\Sigma,\omega)$. When $\Sigma$ has non-empty boundary, we additionally require that the forms agree with $\omega$ near the boundary. If $\Sigma$ is closed, then by definition \[\Omega(\Sigma,\omega)=\Omega(\Sigma, A)\]

	Since $\Sigma$ is a surface, the set $\Omega(\Sigma,\omega)$ is a convex subset of the linear space $\Omega^2(\Sigma)$, and is therefore contractible.
	
	Consider the map
	\[\pi:\Diff_c^+(\Sigma)\to \Omega(\Sigma, \omega), \qquad \pi(\Phi)=\Phi^*\omega.\]
	By the Moser's trick, this map is a surjection. More precisely, for any $\omega^\prime \in \Omega(\Sigma, \omega)$, there exists a $1$-form $\lambda$ such that $\omega-\omega^\prime = d\lambda$. When $\Sigma$ has non-empty boundary, we can further assume
	\[d \lambda = 0\quad \text{ near }\partial \Sigma.\]
	Applying the Moser's trick along the linear isotopy from $\omega^\prime$ to $\omega$ using $\lambda$, we obtain a compactly supported diffeomoprhism $\Phi$ isotopic to the identity such that
	\[\Phi^*\omega = \omega^\prime.\]
	
	{Moreover, the Moser's trick ensure that $\pi:\Diff_c^+(\Sigma)\to \Omega(\Sigma, \omega)$ is a Serre fibration.} Since the base $\Omega(\Sigma, \omega)$ is contractible, the total space $\Diff_c^+(\Sigma)$ is homotopy equivalent to any fiber $\pi^{-1}(\omega^\prime)$, with the homootpy equivalence given by the inclusion map.
	
	Since the fiber at $\omega$ precisely
	\[\pi^{-1}({\omega}) = \Symp_c(\Sigma, \omega),\]
	we conclude that 
	\[\Symp_c(\Sigma, \omega) \simeq \Diff_c^+(\Sigma).\]
	This also implies a homotopy equivalence for the identity components:
	\[\Symp_{0,c}(\Sigma, \omega) \simeq \Diff_{0,c}(\Sigma).\]
	This completes the proof.
\end{proof}
\begin{proof}[Proof of Corollary \ref{Corollary: Symp(S2) is connected}]
	Since $\Diff^+(S^2)$ has a strong deformation retract onto the rotation group $SO(3)$ (see \cite{Smale59}), it follows that $\Diff^+(S^2)$ is path-connected. By Proposition \ref{Proposition: Symp and Diff+ of symplectic surfaces}, this implies that $\Symp(S^2,\omega)$ is also path-connected. In particular, every symplectomorphism on $S^2$ is isotopic to the identity, i.e.
	\[\Symp_0(S^2,\omega)=\Symp(S^2, \omega).\]
	
	For $\phi \in \Symp(S^2,\omega)$, let $\alpha$ be a smooth path of symplectorphisms from $\id_{S^2}$ to $\phi$. Let $\set{X_t}_{t\in[0,1]}$ be time dependent vector fields generating the path $\alpha$. Each $X_t$ is a symplectic, i.e. $\mathcal{L}_{X_t}\omega=0$. By Cartan's formula,
	\[0=\mathcal{L}_{X_t}\omega=d(\iota_{X_t}\omega)+\iota_{X_t}d\omega=d(\iota_{X_t}\omega).\]
	Since every closed $1$-form on $S^2$ is exact, there exists a Hamiltonian $H_t$ satisfying \[\iota_{X_t}\omega=-dH_t.\]
	Therefore, $H$ generates $\phi$, proving that every symplectomorphism on $S^2$ is a Hamiltonian diffeomorphism.
\end{proof}

\begin{proof}[Proof of Corollary \ref{Corollary: Symp(D2) is contractible}]
	By the Alexander's trick, $\Diff^+_c(\mathbb{D}(A))$ has a strong deformation retract onto the identity map. The rest of the proof follows the same argument as in Corollary \ref{Corollary: Symp(S2) is connected}, using the fact that every compactly supported closed $2$-form on $\mathbb{D}(A)$ can be written as the exterior differential of a compactly supported $1$-form.
\end{proof}

\begin{proof}[Proof of Proposition \ref{Proposition: Symplectomorphism supported on a disk}]
	Let $D$ be the image of $i: \Int(\mathbb{D}^2) \to \Sigma$. Then, any symplectomorphism $\phi \in \Symp(\Sigma,\omega;D)$ defines a symplectomorphism $\phi^* \in \Symp_c(\Int(\mathbb{D}^2),i^*\omega)$ by conjugation with $i$. We may assume that the support is contained in the interior of smaller closed disk $D^\prime$ in $\mathbb{D}^2$.
	
	By Lemma \ref{Lemma: symplectic forms on surfaces}, there exists a diffeomorphism
	$\Phi: D^\prime \to D^\prime$ satisfying \[\Phi^*(i^*\omega)=c\cdot \omega_{\std}\]
	for some $c>0$. Furthermore, there exists a dilation symplectomorphism
	\[\lambda: (\mathbb{D}(A), \omega_{\std}) \to (D^\prime, c \cdot \omega_{\std}).\]

	Then, the composition $i \circ \Phi \circ \lambda : (\mathbb{D}(A),\omega_{\std}) \to (\Sigma,\omega)$ gives a symplectic embedding, and $\phi$ induces a compactly supported symplectomorphism on $\mathbb{D}(A)$.		
	By Corollary \ref{Corollary: Symp(D2) is contractible}, there exists a Hamiltonian $H_{\mathbb{D}}$ that genreates the symplectomorphism $(i \circ \Phi \circ \lambda)^*(\phi)$ compactly supported on $\mathbb{D}(A)$.
	
	The Hamiltonian $H_{\mathbb{D}}$ defines a Hamiltonain $H_\Sigma$ on $\Sigma$, defined by
\[H_{\Sigma}(p)=\begin{cases}
		H_{\mathbb{D}}(\lambda^{-1}\circ\Phi^{-1}\circ i^{-1}(p)),			& p \in D\\
		0,		& p \notin D
	\end{cases}.\]
	From the construction, the $H_\Sigma$ generates $\phi$ and is compactly supported on $D$. Hence, every symplectomorphism compactly supported on $D$ is a Hamiltonian diffeomorphism.
\end{proof}

\begin{proof}[Proof of Proposition \ref{Proposition: Fundamental group of Hamiltonian diffeomorphism groups}]
When $\Sigma=S^2$, since $\Diff^+(S^2)$ has a strong deformation retract onto the rotation group $SO(3)$, we have
\[\pi_1(\Diff^+(S^2),\id)=\ZZ_2.\]
Then, by Proposition \ref{Proposition: Symp and Diff+ of symplectic surfaces} and \ref{Corollary: Symp(S2) is connected}, we obtain
\[\pi_1\left(\Ham(S^2,\omega), \id_{S^2} \right)=
\pi_1\left(\Symp_0(S^2,\omega), \id_{S^2} \right)=
\pi_1(\Diff^+(S^2),\id_{S^2})=\ZZ_2.\]

Next, applying Proposition \ref{Proposition: Symp and Diff+ of symplectic surfaces} and known results on $\pi_1(\Diff_0(\Sigma),\id_\Sigma)$, we conclude that
\[\pi_1\left(\Symp_0(\Sigma,\omega), \id_\Sigma \right)=
\begin{cases}
	\ZZ^2,			& \Sigma=T^2\\
	\set{e},		& \Sigma \ne S^2, T^2
\end{cases}.\]

Since the inclusion-induced map $j_*:\pi_1(\Ham(\Sigma,\omega))\to\pi_1(\Symp_0(\Sigma,\omega))$ is injective (see \cite[Proposition 10.2.13]{McDuff-Salamon17}), we deduce
\[\pi_1\left(\Ham(\Sigma,\omega), \id_\Sigma\right) = 0,\]
for $\Sigma$ of genus at least $2$.

Now, for $\Sigma=T^2$, fix a point $q\in T^2$ and consider the evaluation map
\[ev_x:\Diff_0(T^2)\to T^2,\quad f\mapsto f(q).\]
This induces a homomorphism
\[e_D:\pi_1(\Diff_0(T^2),\id_{T^2})\to \pi_1(T^2,q).\]
It is known that $e_D$ is an isomorphism. Consequently, the map
\[e_H:\pi_1(\Ham(T^2,\omega),\id_{T^2})\to \pi_1(T^2,q)\]
obtained by composing the inclusion-indcued map is injective.
 
Let $\gamma$ be a smooth loop in $\Ham(T^2,\omega)$ based at the identity. For the Hamiltonian $H$ generating the loop $\gamma$, every point in $T^2$ is a fixed point of $\phi^1_H=\id_{T^2}$. Moreover, the loops
\[\gamma_p : S^1\to \Sigma, \quad \gamma_p(t)=\phi_H^t(p),\]
are freely homotopic each other, since any path $\set{p_t}_{t\in[0,1]}$ connecting two points $p_0,p_1 \in T^2$ induces a free homotopy $\set{\gamma_{p_t}}_{t\in[0,1]}$ between $\gamma_{p_0},\gamma_{p_1}$.

By Floer theory, we know that for any Hamiltonian $H$, there exists at least one contractible $1$-periodic orbit. Hence, the common free homotopy class of the loops $\gamma_p$ is trivial in $\pi_1(T^2,q)$. Since $e_H$ is injective, we conclude that
\[\pi_1(\Ham(T^2,\omega),\id_{T^2})=\set{e}.\]
This completes the proof.
\end{proof}

\bibliographystyle{plain}
\bibliography{ReferencesS5}
\end{document}